\documentclass{amsart}

\usepackage[english]{babel}
\usepackage{amsmath,amssymb,mathtools,enumitem}
\usepackage[T1]{fontenc}

\usepackage{titlesec}
\usepackage{amsthm}
\usepackage{cite}
\usepackage{amsrefs}

\usepackage{varioref}
\usepackage{hyperref}
\usepackage[capitalise]{cleveref}

\newcommand{\mc}[1]{\mathcal{#1}}
\newcommand{\mf}[1]{\mathfrak{#1}}

\newcommand{\ol}[1]{\overline{#1}}

\newcommand{\ti}[1]{\textit{#1}}
\newcommand{\tx}[1]{\textrm{#1}}

\newcommand{\N}{\mathbb{N}}

\newcommand{\R}{\mathbb{R}}

\newcommand{\Z}{\mathbb{Z}}

\newcommand{\dd}{d}

\newcommand{\ddt}{\frac{d}{dt}}

\newcommand{\del}{\delta}
\newcommand{\deta}{\,d\eta}

\newcommand{\ds}{\,ds}

\newcommand{\dx}{\,dx}
\newcommand{\dxi}{\,d\xi}
\newcommand{\dy}{\,dy}

\newcommand{\eps}{\epsilon}

\newcommand{\bigo}{O}

\newcommand{\fatp}{\widetilde{P}}
\newcommand{\ft}{\mc{F}}
\newcommand{\hkdv}{H_{\tx{KdV}}}

\newcommand{\I}{\mf{I}}
\newcommand{\mhi}{m_{\tx{hi}}}
\newcommand{\mlo}{m_{\tx{lo}}}
\newcommand{\norm}[1]{\left\lVert#1\right\rVert}
\newcommand{\op}{\tx{op}}
\newcommand{\pihi}{\Pi_{\geq N}}
\newcommand{\pilo}{\Pi_{< N}}

\newcommand{\snorm}[1]{\lVert#1\rVert}
\newcommand{\thk}{H^{W}_{\kappa}}
\newcommand{\thvk}{H^{W}_{\varkappa}}
\newcommand{\tr}{\operatorname{tr}}
\newcommand{\wh}[1]{\widehat{#1}}

\theoremstyle{plain}\newtheorem{thm}{Theorem}[section]
\theoremstyle{plain}\newtheorem{cor}[thm]{Corollary}
\theoremstyle{plain}\newtheorem{prop}[thm]{Proposition}
\theoremstyle{plain}\newtheorem{lem}[thm]{Lemma}
\theoremstyle{definition}
\theoremstyle{definition}
\theoremstyle{remark}
\theoremstyle{definition}\newtheorem*{ack}{Acknowledgments}

\crefname{sec}{Section}{Sections}
\crefname{dfn}{Definition}{Definitions}
\crefname{hyp}{Hypothesis}{Hypotheses}
\crefname{lem}{Lemma}{Lemmas}
\crefname{prop}{Proposition}{Propositions}
\crefname{thm}{Theorem}{Theorems}
\crefname{cor}{Corollary}{Corollaries}

\titleformat{\section}[block]{\normalfont\normalsize\scshape\filcenter}{\thesection.}{0.5em}{}
\titleformat{\subsection}[runin]{\normalfont}{\thesubsection.}{0.5em}{\bfseries}

\numberwithin{equation}{section}

\allowdisplaybreaks

\begin{document}

\title{{K}d{V} on an incoming tide}

\author{Thierry Laurens}
\address{Thierry Laurens \\
	Department of Mathematics\\
	University of California, Los Angeles, CA 90095, USA}
\email{laurenst@math.ucla.edu}

\maketitle

\begin{abstract}
	Given smooth step-like initial data $V(0,x)$ on the real line, we show that the Korteweg--de Vries equation is globally well-posed for initial data $u(0,x) \in V(0,x) + H^{-1}(\mathbb{R})$.  The proof uses our general well-posedness result from~\cite{Laurens2021}.
	
	As a prerequisite, we show that KdV is globally well-posed for $H^3(\mathbb{R})$ perturbations of step-like initial data.  In the case $V \equiv 0$, we obtain a new proof of the Bona--Smith theorem~\cite{Bona1975} using the low-regularity methods that established the sharp well-posedness of KdV in $H^{-1}$~\cite{Killip2019}.
\end{abstract}


\section{Introduction}

The Korteweg--de Vries (KdV) equation
\begin{equation}
\ddt u = - u''' + 6uu'
\label{eq:kdv}
\end{equation}
(where primes $u' = \partial_xu$ denote spatial differentiation) was proposed in~\cite{Korteweg1895} to describe the phenomena of solitary traveling waves (solitons) in shallow channels.  Since its introduction over a century ago, the KdV equation has been thoroughly studied on the line $\R$ and the circle $\R/\Z$ and has been shown to exhibit a multitude of special features.

A fundamental direction of investigation for KdV has been well-posedness in the $L^2$-based Sobolev spaces $H^s(\R)$ and $H^s(\R/\Z)$.  However, the derivative in the nonlinearity of KdV prevents straightforward contraction mapping arguments from closing, so preliminary results produced continuous dependence in a weaker norm than the space of initial data.  One of the first results to overcome this loss of derivatives phenomenon was obtained by Bona and Smith~\cite{Bona1975} who established global well-posedness in $H^3(\R)$.  Numerous methods were developed in the following decades in the effort to lower the regularity $s$; see for example~\cite{Bona1976,Kato1975,Saut1976,Temam1969,Tsutsumi1971,Kenig1991,Bourgain1993,Christ2003,Kenig1996,Colliander2003,Guo2009,Kishimoto2009}.  Recently, Killip and Vi\c{s}an~\cite{Killip2019} introduced the method of commuting flows to demonstrate global well-posedness in $H^{-1}(\R)$ and $H^{-1}(\R/\Z)$, a result that is sharp in both topologies.  In the $\R/\Z$ case, this result was already known~\cite{Kappeler2006}.

Solutions in $H^s(\R/\Z)$ spaces are spatially periodic and solutions in $H^s(\R)$ spaces decay at infinity.  However, there are other classes of initial data which are of physical interest.  In particular, waveforms that are step-like---in the sense that $u(0,x)$ asymptotically approaches distinct constant values as $x\to \pm\infty$---arise naturally in the study of bore propagation and rarefaction waves.  Such asymptotic behavior has real physical consequences.  Indeed, we shall see below that the polynomial conservation laws are broken, and in the case of an incoming tide there is an infinite influx of energy into the system.

Our objective in this paper is to extend low-regularity methods for well-posedness to the regime of nonzero spatial asymptotics.  We define the smooth step function
\begin{equation*}
W(x) = c_1 \tanh(x) + c_2 \quad \tx{with }c_1,c_2\in\R\tx{ fixed},
\end{equation*}
which exponentially decays to its asymptotic values.  As $-u$ is proportional to the water wave height, $W$ models an incoming tide if $c_1 > 0$ and an outgoing tide if $c_1 < 0$.  In fact, we can always perform a boost to prescribe $c_2$ courtesy of the Galilean symmetries of KdV~\eqref{eq:kdv}, but we will not make use of this.

A classical result in the study of step-like asymptotics is:
\begin{thm}
	\label{thm:intro H3}
	Fix an integer $s\geq 3$.  The KdV equation~\eqref{eq:kdv} with initial data $u(0) \in W + H^s(\R)$ is globally well-posed in the following sense:	$u(t) = W + q(t)$ where $q(t)$ is the global solution to
	\begin{equation}
	\ddt q = -(q+W)''' + 6(q+W)(q+W)'
	\label{eq:tkdv}
	\end{equation}	
	with initial data $q(0) = u(0) - W$ in $H^s(\R)$.  Moreover, $q(t)$ is in $C_tH^s([-T,T]\times\R)$ for all $T>0$, $q(t)$ is unique in this class, and $q(t)$ depends continuously upon the initial data $q(0)$ in $H^s(\R)$.
\end{thm}

\Cref{thm:intro H3} is not new (as we will discuss below), but we will use its statement to formulate our main result.  Applying \cref{thm:intro H3} to the initial data $q(0) \equiv 0$, we conclude that given $W$ there is a unique global solution $V(t) = W + q(t)$ to KdV~\eqref{eq:kdv} with initial data $W$, and $t\mapsto V(t) - W$ is a continuous function into $H^s(\R)$ for all $s\geq 3$.  The main thrust of this work is to show that KdV is globally well-posed for $H^{-1}(\R)$ perturbations of $V(t)$:
\begin{thm}
	\label{thm:intro H-1}
	The KdV equation~\eqref{eq:kdv} with initial data $u(0) \in W + H^{-1}(\R)$ is globally well-posed in the following sense: $u(t) = V(t) + q(t)$ where $V(t)$ solves KdV with initial data $W$ and the equation
	\begin{equation}
	\ddt q = - q''' + 6qq' + 6(Vq)'
	\label{eq:kdv with potential}
	\end{equation}
	for $q(t)$ with initial data in $H^{-1}(\R)$ is globally well-posed.
\end{thm}

Let us clarify the notion of well-posedness in \cref{thm:intro H-1}.   As we cannot make sense of the nonlinearity of KdV for arbitrary functions in $H^{-1}(\R)$ (even in the distributional sense), the solutions in \cref{thm:intro H-1} are constructed as limits of solutions to a family of approximate equations.  We then show that the data-to-solution map $q(0)\mapsto q(t)$ is a jointly continuous function of $t\in\R$ and $q(0)\in H^{-1}(\R)$ into $H^{-1}(\R)$. The notions of solution and uniqueness is that for the dense subset $H^3(\R)$ of initial data $q(0)$ the functions $q(t)$ coincide with classical solutions (cf.~\cite[Th.~1.3]{Laurens2021}) and the data-to-solution map is continuous.

The proof of \cref{thm:intro H-1} relies on our general well-posedness result~\cite{Laurens2021}, which proves that the equation~\eqref{eq:kdv with potential} is well-posed in $H^{-1}(\R)$ provided that the background wave $V(t)$ satisfies certain criteria (which we will formulate below).  Verifying that $V(t)$ satisfies these criteria for the step-like initial data $W$ will be accomplished by certain ingredients in the proof of \cref{thm:intro H3}, namely \cref{thm:thk gwp,thm:conv,thm:tkdv exist}.

It is natural to ask whether KdV is also well-posed for $H^{-1}(\R)$ perturbations of $W$.  \Cref{thm:intro H-1,,thm:intro H3} provide an affirmative answer to this question.  By \cref{thm:intro H-1}, there exists a solution $u(t) = V(t) + q(t)$ to KdV~\eqref{eq:kdv} with initial data $u(0) = W + q(0)$ in $W + H^{-1}(\R)$.  Together with \cref{thm:intro H3}, we also obtain that $t\mapsto u(t) - W$ is a continuous function into $H^{-1}(\R)$ that depends continuously upon the initial data.  For a precise statement of this well-posedness, see \cref{thm:intro H-1 2}.  We do not use this formulation in the statement of \cref{thm:intro H-1} because it does not reflect the reality of the proof.

Just as $H^{-1}(\R)$ is the lowest regularity for which we can hope to have well-posedness in the case $W\equiv 0$~\cite{Molinet2011}, we expect that \cref{thm:intro H-1} is sharp in the class of $H^s(\R)$ spaces.  There is a known technique~\cite[Cor.~5.3]{Killip2019} for extending $H^{-1}(\R)$ well-posedness to $H^s(\R)$, $s> -1$, using equicontinuity, and so $H^{-1}(\R)$ is the key space for establishing well-posedness.  

Next we turn our attention to a discussion of prior work.  In~\cite[\S3]{Benjamin1972}, Benjamin, Bona, and Mahony discuss well-posedness for the (closely related) BBM equation with step-like initial data.  In the case $W\equiv 0$, Bona and Smith~\cite{Bona1975} proved that KdV is well-posed in $H^s(\R)$ for $s\geq 3$ by approximating KdV by a family of BBM equations.  The formulation of \cref{thm:intro H3} is inspired by~\cite{Bona1975}; indeed, \cref{thm:intro H3} can be proved using their original argument.  However, it is our proof of \cref{thm:intro H3}, not the formulation, that we need as an ingredient for \cref{thm:intro H-1}.  In contrast to the Bona--Smith approach, we approximate KdV by a family of commuting flows introduced by Killip and Vi\c{s}an~\cite{Killip2019}.  This has the advantage that the \ti{a priori} estimates are the same as those for KdV, and convergence can be demonstrated in a transparent way (by upgrading continuity in a lower regularity norm using equicontinuity).

Lower regularity than $H^3(\R)$ has been obtained in the study of well-posedness for perturbations of a fixed step-like background wave.  The first result was recorded in~\cite{Iorio1998}, who proved local well-posedness for perturbations in $H^s(\R)$, $s>\tfrac{3}{2}$, and global well-posedness for $s\geq 2$.  Local well-posedness was then extended to $s>1$ in~\cite{Gallo2005} for the same family of background waves.  Independently, local well-posedness for $H^2(\R)$ perturbations was proved for gKdV in~\cite{Zhidkov2001}, along with global-in-time existence in the case of a kink solution background wave and initial data that is small in $H^{1}(\R)$.

Subsequent to our work, a new result~\cite{Palacios2021} for gKdV demonstrates local well-posedness for perturbations in $H^s(\R)$, $s>\tfrac{1}{2}$ and global well-posedness for $s\geq 1$.  In addition to a larger class of equations, this work also applies to a wide variety of background waves, including both step-like and periodic asymptotics.  In particular, the background wave is not assumed to be time-independent nor an exact solution, but rather is allowed to solve the equation modulo a localized error term.

The primary tool used in the literature to study step-like solutions of KdV has been the inverse scattering transform.  In the case of a highly regular step-like background, existence for the Cauchy problem has been examined in~\cite{Buslaev1962,Cohen1984,Kappeler1986,Cohen1987,Egorova2009,Egorova2011}.  In order to employ the inverse scattering transform these results assume that $u(0)-W$ is integrable against $1+|x|^N$ for some $N\geq 1$, and consequently such methods are not suitable for $H^s(\R)$ spaces.  Nevertheless, as shown in~\cite{Egorova2009}, these methods do yield existence for Schwartz class perturbations.  Classes of one-sided step-like initial data were treated in~\cite{Rybkin2011,Grudsky2014,Rybkin2018} and one-sided step-like elements of $H^{-1}_{\tx{loc}}(\R)$ were treated in~\cite{Grudsky2015}.  Despite the lack of assumptions at $-\infty$ (the direction in which radiation propagates), these low-regularity arguments require rapid decay at $+\infty$ and global boundedness from below.  By comparison, our argument is symmetric in $\pm x$ and in $\pm u$.  

The inverse scattering transform is also used to study the long-time behavior of such solutions; see for example~\cite{Hruslov1976,Kotlyarov1986,Kotlyarov1986a,Bikbaev1989,Bikbaev1989a,Bikbaev1989b,Khruslov1994,Khruslov1998,Baranetskiui2001,Novokshenov2003,Egorova2013,Andreiev2016}.  The asymptotics are spatially asymmetric and differ in the cases of tidal bores and rarefaction waves.

In this paper, we employ the method of commuting flows introduced in~\cite{Killip2019}.  This method was used to prove both symplectic non-squeezing~\cite{Ntekoume2019} and invariance of white noise~\cite{Killip2020} for KdV on the line.  The method of commuting flows has also been adapted to other completely integrable systems, including the cubic NLS and mKdV equations~\cite{HarropGriffiths2020}, the fifth-order KdV equation~\cite{Bringmann2019}, and the derivative NLS equation~\cite{Killip2021}.  Together with~\cite{Laurens2021}, this is the first application of this method to exotic spatial asymptotics.

The presence of the background wave $W$ breaks the macroscopic conservation laws of KdV.  A solution of KdV~\eqref{eq:kdv} must obey the microscopic conservation law
\begin{equation*}
\ddt \left( \tfrac{1}{2} u^2 \right)
= \left[ -uu'' + \tfrac{1}{2} (u')^2 + 2u^3 \right]' .
\end{equation*}
For Schwartz solutions $u$ to KdV, integrating in space yields (macroscopic) conservation of the momentum
\begin{equation}
P(u) := \tfrac{1}{2} \int u(x)^2\dx .
\label{eq:momentum}
\end{equation}
However, if merely $u-W$ is Schwartz then we obtain
\begin{equation}
\ddt \int \tfrac{1}{2} \left[ u(t,x)^2 - u(0,x)^2 \right] \dx
= 2 W(x)^3 \bigg|_{x=-\infty}^{x=+\infty} .
\label{eq:intro momentum dot}
\end{equation}
In the case $c_1>0$, $c_2 = 0$ of an incoming tide, the RHS is equal to $4c_1^3>0$.  The momentum's growth is manifested in a dispersive shock that develops in the long-time asymptotics~\cite[Fig.~1]{Egorova2013}.

Interpreting $W$ as an incoming or outgoing tide, we will refer to~\eqref{eq:tkdv} as \ti{tidal KdV}.  To prove \cref{thm:intro H3} we will show that tidal KdV is well-posed in $H^s(\R)$ for $s\geq 3$.  Computations similar to~\eqref{eq:intro momentum dot} show that the presence of $W$ in tidal KdV breaks all of the polynomial conservation laws of KdV.  Despite this, we are able to adapt the method of commuting flows to tidal KdV because these conserved quantities do not blow up in finite time.

In order to introduce our methods, we will first present some notation.  The KdV equation~\eqref{eq:kdv} is governed by the Hamiltonian functional
\begin{equation}
	\hkdv(q) := \int \big( \tfrac{1}{2} q'(x)^2 + q(x)^3 \big) \dx
	\label{eq:hkdv}
\end{equation}
via the Poisson structure
\begin{equation*}
	\{ F,G \} = \int \frac{\del F}{\del q}(x) \bigg( \frac{\del G}{\del q} \bigg)'(x) \dx .
\end{equation*}
Here we are using the notation
\begin{equation*}
	dF|_q(f)
	= \frac{d}{ds} \bigg|_{s=0} F(q+sf)
	= \int \frac{\del F}{\del q}(x) f(x) \dx
\end{equation*}
for the derivative of the functional $F(q)$.  This Poisson structure is the bracket associated to the almost complex structure $J := \partial_x$ and the $L^2$ pairing.  In accordance with its name, the momentum functional~\eqref{eq:momentum} generates translations under this structure.

Our analysis will not rely upon these concepts, but we will borrow the convenient notations
\begin{equation*}
	q(t) = e^{t J \nabla H} q(0) \quad\tx{for the solution to}\quad \frac{d q}{d t} = \partial_{x} \frac{\delta H}{\delta q} ,
\end{equation*}
and
\begin{equation*}
	\ddt F \circ e^{tJ\nabla H} = \{ F,H \} \circ e^{tJ\nabla H}
	\quad\tx{for the quantity }F(q)\tx{ with }q(t) = e^{t J \nabla H} q(0) .
\end{equation*}

In the case $W\equiv 0$, the authors of~\cite{Killip2019} introduced a family of commuting flows that approximate that of KdV.  This approximation relies on the existence of a generating function $\alpha(\kappa,q)$ for the KdV hierarchy of conserved quantities with the asymptotic expansion 
\begin{equation}
	\alpha(\kappa , q) = \frac{1}{4 \kappa^{3}} P(q) - \frac{1}{16 \kappa^{5}} \hkdv(q) + \bigo(\kappa^{-7})
	\label{eq:alpha intro}
\end{equation}
for Schwartz $q$.  Here $P$ and $\hkdv$ are the momentum and KdV energy functionals~\eqref{eq:momentum} and~\eqref{eq:hkdv} respectively.  The quantity $\alpha(\kappa, q)$ is a renormalized logarithm of the transmission coefficient for the Schr\"odinger operator with potential $q$ (i.e. perturbation determinant) at energy $-\kappa^2$, and is a real analytic functional of $q$ in a neighborhood of the origin in $H^{-1}(\R)$ for all $\kappa \geq 1$.

Rearranging the expansion~\eqref{eq:alpha intro}, the authors of~\cite{Killip2019} introduced the Hamiltonians
\begin{equation}
	H_\kappa(q) := -16\kappa^5 \alpha(\kappa,q) + 4\kappa^2 P(q)
	\label{eq:hk intro}
\end{equation}
and showed that their flow converges to that of KdV in $H^{-1}(\R)$ as $\kappa\to\infty$.  The $H_\kappa$ flows are easier to work with, as well-posedness follows from straightforward ODE arguments.  Moreover, two flows with different energy parameters $\kappa$ commute with one another, which facilitates the demonstration of convergence as $\kappa\to\infty$.

Our general result~\cite{Laurens2021} is that the equation~\eqref{eq:kdv with potential} is well-posed in $H^{-1}(\R)$ provided that for every $T>0$ the background wave $V : \R\times\R \to \R$ satisfies the following:
\begin{enumerate}[label=(\roman*)]
	\item $V$ solves KdV~\eqref{eq:kdv} and is bounded in $W^{2,\infty}(\R_x)$ uniformly for $|t|\leq T$,
	\item The solutions $V_\kappa(t)$ to the $H_\kappa$ flows with initial data $V(0)$ are bounded in $W^{4,\infty}(\R_x)$ uniformly for $|t|\leq T$ and $\kappa>0$ sufficiently large,
	\item $V_\kappa - V\to 0$ in $W^{2,\infty}(\R_x)$ as $\kappa\to\infty$ uniformly for $|t|\leq T$ and initial data in the set $\{V_\varkappa(t) : |t|\leq T,\ \varkappa \geq \kappa \}$.
\end{enumerate}
To prove \cref{thm:intro H-1} we need to study the $H_\kappa$ flows $V_\kappa(t)$ for step-like initial data $W$.  After subtracting the background profile $W$, this is tantamount to showing that the method of commuting flows can be applied to tidal KdV~\eqref{eq:tkdv}.

As the $H_\kappa$ flows approximate KdV, we will need to construct analogous approximate equations for tidal KdV~\eqref{eq:tkdv}.  Just as how we obtained tidal KdV from KdV, we subtract the background wave $W$ from $u$ to obtain the tidal $H_\kappa$ flow for $q = u - W$ with Hamiltonian $\thk$:
\begin{equation*}
e^{tJ\nabla \thk} q = e^{tJ\nabla H_\kappa} (q+W) - W .
\end{equation*}
This tidal $H_\kappa$ flow is indeed Hamiltonian, but we will not need the formula for the Hamiltonian; we only formally introduce $\thk$ so that we have a succinct notation for its flow.  In proving \cref{thm:intro H-1,thm:intro H3}, we will show that the $\thk$ flow is well-posed in $H^s(\R)$ for $s\geq 3$, commutes with any other $\thvk$ flow, and converges to tidal KdV in $H^s(\R)$ as $\kappa \to \infty$ uniformly on bounded time intervals.

This paper is organized as follows.  In \cref{sec:prelim} we define the diagonal Green's function $g$ for perturbations $q\in H^{-1}(\R)$ of the background $W$ which we will use to formulate the tidal $H_\kappa$ flow.  In \cref{sec:thk} we prove \ti{a priori} estimates and global well-posedness for the tidal $H_\kappa$ flow.  As a stepping stone to convergence in $H^s$ norm, we prove in \cref{sec:low reg} that the tidal $H_\kappa$ flow converges in the weaker $H^{-2}$ norm.  The entirety of \cref{sec:equicty} is dedicated to controlling the Fourier tail growth in time.  We then combine the low-regularity convergence and Fourier tail control in \cref{sec:wp} to obtain convergence in $H^s$ norm and conclude our main result.

\begin{ack}
	I was supported in part by NSF grants DMS-1856755 and DMS-1763074.  I would also like to thank my advisors, Rowan Killip and Monica Vi\c{s}an, for their guidance.
\end{ack}

\section{Diagonal Green's function}
\label{sec:prelim}

We begin by reviewing our notation and the necessary tools from~\cite{Killip2019}, which can be consulted for further details.

For a Sobolev space $W^{k,p}(\R)$ we use the spacetime norm
\begin{equation*}
	\norm{ q }_{C_tW^{k,p}(I\times\R)} := \sup_{t\in I} \norm{ q(t) }_{W^{k,p}(\R)}
\end{equation*}
for $I\subset\R$ an interval.  In addition to the usual Sobolev spaces $W^{k,p}$ and $H^s$ we define the norm
\begin{equation}
	\norm{f}^2_{H^s_\kappa(\R)} := \int_\R (\xi^2+4\kappa^2)^s |\hat{f}(\xi)|^2\dxi ,
	\label{eq:Hskappa norm dfn}
\end{equation}
where our convention for the Fourier transform is
\begin{equation*}
	\hat{f}(\xi) = \frac{1}{\sqrt{2\pi}} \int_{\R} e^{-i\xi x}f(x)\dx ,\qquad
	\snorm{\hat{f}}_{L^2} = \norm{f}_{L^2} .
\end{equation*}
In analogy with the usual $H^s$ spaces, we have the elementary facts
\begin{equation}
	\norm{ w f }_{H_{\kappa}^{\pm 1}} \lesssim \norm{ w }_{W^{1,\infty}} \norm{f}_{H_{\kappa}^{\pm 1}} ,  \qquad
	\norm{ w f }_{H_{\kappa}^{\pm 1}} \lesssim \norm{ w }_{H^{1}} \norm{ f }_{H_{\kappa}^{\pm 1}}
	\label{eq:hskappa ests}
\end{equation}
uniformly for $\kappa \geq 1$.  We will exclusively use the $L^2$ pairing $\langle\cdot,\cdot\rangle$; the space $H_{\kappa}^{-1}$ is dual to $H_{\kappa}^{1}$ with respect to this pairing, and so the inequalities~\eqref{eq:hskappa ests} for $H_{\kappa}^{-1}$ are implied by those for $H_{\kappa}^{1}$.

We write $\I_p$ for the Schatten classes (also called trace ideals) of compact operators on the Hilbert space $L^2(\R)$ whose singular values are $\ell^p$-summable.  Of particular importance will be the Hilbert--Schmidt class $\I_2$: recall that an operator $A$ on $L^2(\R)$ is Hilbert--Schmidt if and only if it admits an integral kernel $a(x,y) \in L^2(\R\times\R)$, and we have
\begin{equation*}
	\norm{ A }_\op \leq \norm{ A}_{\I_2} = \iint |a(x,y)|^2 \dx\dy .
\end{equation*}
The product of two Hilbert--Schmidt operators $A$ and $B$ is of trace class $\I_1$, the trace is cyclic:
\begin{equation*}
	\tr(AB) := \iint a(x,y)b(y,x) \dy\dx = \tr(BA) ,
\end{equation*}
and we have the estimate
\begin{equation*}
	|\tr(AB)| \leq \norm{A}_{\I_2} \norm{B}_{\I_2} .
\end{equation*}
Additionally, Hilbert--Schmidt operators form a two-sided ideal in the algebra of bounded operators, due to the inequality
\begin{equation*}
	\norm{ BAC }_\op \leq \norm{B}_\op \norm{A}_{\I_2} \norm{C}_\op .
\end{equation*}

We denote the resolvent of the Schr\"odinger operator with zero potential by
\begin{equation*}
	R_0(\kappa) := \left( - \partial_x^2 + \kappa^2 \right)^{-1}
	\quad\tx{with integral kernel}\quad
	\langle \del_x, R_0(\kappa)\del_y \rangle = \tfrac{1}{2\kappa} e^{-\kappa|x-y|} .
\end{equation*}
The energy parameter $\kappa$ will always be real and positive.  Consequently, $R_0(\kappa)$ will always be positive definite and so we may consider its positive definite square-root $\sqrt{ R_0(\kappa) }$.

The following computation from~\cite[Prop.~2.1]{Killip2019} lies at the heart of our analysis:
\begin{lem}
	For $q\in H^{-1}(\R)$ we have
	\begin{equation}
		\norm{ \sqrt{ R_0(\kappa) } q \sqrt{ R_0(\kappa) } }_{\I_2}^2
		= \frac{1}{\kappa} \int \frac{|\hat{q}(\xi)|^2}{\xi^2+4\kappa^2}\dxi
		= \frac{1}{\kappa} \norm{ q }^2_{H^{-1}_\kappa} .
		\label{eq:hskappa identity 1}
	\end{equation}
\end{lem}

The identity~\eqref{eq:hskappa identity 1} guarantees that the Neumann series for the resolvent of $-\partial^2 + q$ with $q\in H^{-1}$ converges for $\kappa$ sufficiently large.  This construction also works for $q$ perturbations of $W$:
\begin{lem}[Resolvents]
	Given $q\in H^{-1}(\R)$, there is a unique self-adjoint operator corresponding to $-\partial^2_x + W+ q$ with domain $H^1(\R)$.  Moreover, given $A>0$ there exists a constant $\kappa_0$ so that the series
	\begin{equation}
		R(\kappa,W) 
		= (-\partial^2+W+\kappa^2)^{-1}
		= \sum_{\ell=0}^\infty (-1)^\ell \sqrt{ R_0 } \big( \sqrt{R_0} W \sqrt{R_0} \big)^\ell \sqrt{ R_0 }
		\label{eq:resolvent series 1}
	\end{equation}
	converges absolutely to a positive definite operator for $\kappa \geq \kappa_0$, and the series
	\begin{equation}
		R(\kappa,W+q) = \sum_{\ell=0}^\infty (-1)^\ell \sqrt{ R(\kappa,W) } \big( \sqrt{R(\kappa,W)} q \sqrt{R(\kappa,W)} \big)^\ell \sqrt{ R(\kappa,W) }
		\label{eq:resolvent series 2}
	\end{equation}
	converges absolutely for all $\norm{q}_{H^{-1}} \leq A$ and $\kappa \geq \kappa_0$.
\end{lem}
\begin{proof}
	Initially we require that $\kappa\geq 1$.  As $W\in L^\infty$, we may define the operator $-\partial^2 + W$ via the quadratic form
	\begin{equation*}
		\phi \mapsto \int \big( |\phi'(x)|^2 + W(x) |\phi(x)|^2 \big)\dx 
	\end{equation*}
	equipped with the domain $H^1(\R)$.  Using the elementary estimates $\norm{R_0}_\op \leq \kappa^{-2}$ and $\norm{W}_\op \leq \norm{W}_{L^\infty}$, it is clear that the series~\eqref{eq:resolvent series 1} for $R(\kappa,V)$ is absolutely convergent for all $\kappa^2 \geq 2\norm{W}_{L^\infty}$.
	
	Expanding the series~\eqref{eq:resolvent series 1} and using the identity~\eqref{eq:hskappa identity 1} we estimate
	\begin{align*}
		&\norm{ \sqrt{R(\kappa,W)} q \sqrt{R(\kappa,W)} }_{\I_2}^2
		= \tr\{ R(\kappa,W) q R(\kappa,W) \ol{q} \} \\
		&\leq \sum_{\ell,m = 0}^\infty \norm{ \sqrt{R_0}W\sqrt{R_0} }_{\op}^{\ell+m} \norm{ \sqrt{R_0}q\sqrt{R_0} }^2_{\I_2}
		\leq 4 \kappa^{-1} \norm{q}^2_{H^{-1}_\kappa}
	\end{align*}
	for all $\kappa^2 \geq 2\norm{W}_{L^\infty}$, and hence
	\begin{equation}
		\norm{ \sqrt{R(\kappa,W)} q \sqrt{R(\kappa,W)} }_{\I_2} \leq 2 \kappa^{-1/2} \norm{q}_{H^{-1}_\kappa} .
		\label{eq:hskappa identity 2}
	\end{equation}
	Consequently, for $\phi\in H^1(\R)$ we have
	\begin{align*}
		\int q(x) |\phi(x)|^2 \dx
		&\leq \norm{ \sqrt{R(\kappa,W)} q \sqrt{R(\kappa,W)} }_\op \int \big( |\phi'(x)|^2 + |W(x)| |\phi(x)|^2 \big) \dx \\
		&\leq \tfrac{1}{2} \int \big( |\phi'(x)|^2 + |W(x)| |\phi(x)|^2 \big) \dx
	\end{align*}
	provided that $\kappa \geq 16 A^2$.  We conclude that $-\partial^2 + W + q$ is a form-bounded perturbation of $-\partial^2 + W$ with relative norm strictly less than 1; this guarantees that $-\partial^2 + W + q$ exists, is unique, and has the same form domain $H^1(\R)$ (cf.~\cite[Th.~X.17]{Reed1975}).  The estimate~\eqref{eq:hskappa identity 2} then demonstrates that the series~\eqref{eq:resolvent series 2} for $R(\kappa,W+q)$ is absolutely convergent for all $\kappa \gg A^2$.
\end{proof}

In~\cite{Killip2019} the restriction of the integral kernel of $R(\kappa,q) - R_0(\kappa)$ to the diagonal was instrumental in controlling $q$ in $H^{-1}$.  This construction also works for $q$ perturbations of $W$:
\begin{prop}[Diagonal Green's function]
	\label{thm:diffeo prop}
	Fix an integer $s\geq -1$ and $A>0$, and let $B_A\subset H^s(\R)$ denote the closed ball of radius $A$.  There exists a constant $\kappa_0$ such that for $\kappa\geq\kappa_0$ the diagonal Green's function $g(x;\kappa,W+q) := G(x,x;\kappa,W+q)$ exists for $q\in B_A$, the functional
	\begin{equation}
		q \mapsto g(x;\kappa,W+q) - g(x;\kappa,W)
		\label{eq:diffeo prop 2}
	\end{equation}
	is real-analytic $B_A \to H^{s+2}_\kappa$, and we have the estimate
	\begin{equation}
		\norm{ g(x;\kappa,W+q) - g(x;\kappa,W) }_{H^{s+2}_\kappa}
		\lesssim \kappa^{-1} \norm{q}_{H^{s}_\kappa}
		\label{eq:diffeo prop}
	\end{equation}
	uniformly for $q\in B_A$ and $\kappa\geq \kappa_0$.
\end{prop}
\begin{proof}
	In Fourier variables, for $\kappa\geq 1$ we compute
	\begin{equation*}
	\snorm{ \sqrt{R_0}\del_x }_{L^2}^2 \lesssim \kappa^{-1}, \qquad 
	\snorm{ \sqrt{R_0} \del_{x+h} - \sqrt{R_0} \del_x }_{L^2}^2
	\leq \int \frac{|e^{i\xi h}-1|^2}{\xi^2 + 1} \dxi 
	\lesssim |h| .
	\end{equation*}
	This demonstrates that $x\mapsto \sqrt{R_0} \del_x$ is $\tfrac{1}{2}$-H\"older continuous as a mapping $\R\to L^2$.  Therefore, from the series~\eqref{eq:resolvent series 1} we see that
	\begin{align*}
	&\left| \left\langle\del_x , \left[R(\kappa, W)-R_{0}(\kappa)\right] \del_y \right\rangle - \left\langle\del_{x'} , \left[R(\kappa, W)-R_{0}(\kappa)\right] \del_{y'} \right\rangle \right| \\
	&\lesssim \kappa^{-1/2} \big( |x-x'|^{1/2} + |y-y'|^{1/2} \big) \sum_{\ell =1}^{\infty} \left( \kappa^{-2} \norm{W}_{L^\infty} \right)^\ell .
	\end{align*}
	The series converges provided that $\kappa \gg \norm{W}_{L^\infty}^{1/2}$.  Consequently, the Green's function $G(x,y)=\langle \del_x, R(\kappa,W) \del_y \rangle$ is continuous in both $x$ and $y$, and so we may unambiguously define
	\begin{equation}
	g(x;\kappa,W)
	= \tfrac{1}{2\kappa} + \sum_{\ell=1}^\infty (-1)^\ell \big\langle \sqrt{R_0} \del_x, (\sqrt{R_0}W\sqrt{R_0})^\ell \sqrt{R_0} \del_x \big\rangle .
	\label{eq:g series 1}
	\end{equation}
	The zeroth order term $\tfrac{1}{2\kappa}$ can be seen directly from the integral kernel for the free resolvent $R_0(\kappa)$.
	
	Using the series~\eqref{eq:resolvent series 1} we obtain
	\begin{align*}
	&\snorm{ \sqrt{R(\kappa,W)}\del_x }_{L^2}^2
	\leq \snorm{\sqrt{R_0}\del_x}_{L^2}^2 \sum_{\ell=1}^\infty \big( \kappa^{-2} \norm{W}_{L^\infty} \big)^\ell
	\lesssim \kappa^{-1}, \\ 
	&\snorm{ \sqrt{R(\kappa,W)} \del_{x+h} - \sqrt{R_0} \del_x }_{L^2}^2
	\leq \snorm{ \sqrt{R_0}\del_{x+h} - \sqrt{R_0}\del_{x} }_{L^2}^2 \sum_{\ell=1}^\infty \big( \kappa^{-2} \norm{W}_{L^\infty} \big)^\ell
	\lesssim |h|
	\end{align*}
	provided that $\kappa \gg \norm{W}^{1/2}_{L^\infty}$.  From the series~\eqref{eq:resolvent series 2} and the estimate~\eqref{eq:hskappa identity 2} we then have
	\begin{align*}
	&\left| \left\langle\del_x , \left[R(\kappa, W+q)-R(\kappa,W)\right] \del_y \right\rangle - \left\langle\del_{x'} , \left[R(\kappa, W+q)-R(\kappa,W)\right] \del_{y'} \right\rangle \right| \\
	&\lesssim \kappa^{-1/2} \big( |x-x'|^{1/2} + |y-y'|^{1/2} \big) \sum_{\ell =1}^{\infty} \big( 2 \kappa^{-1/2} A \big)^\ell
	\end{align*}
	for all $q\in B_A$.  The series converges provided that $\kappa\gg A^2$.  Therefore, the Green's function $G(x,y;\kappa,W+q)$ is also a continuous function of $x$ and $y$ and we may define
	\begin{equation*}
	g(x;\kappa,W+q) = g(x;\kappa,W) + \sum_{\ell=1}^\infty (-1)^\ell \big\langle \sqrt{R} \del_x, ( \sqrt{R}\, q \sqrt{R})^\ell \sqrt{R} \del_x \big\rangle ,
	\end{equation*}
	where $R = R(\kappa,W)$.  This shows that $g(x;\kappa,W+q)$ is a real analytic functional $B_A \to H^1$.
	
	Next, we check that $g(x;\kappa,W+q) - g(x;\kappa,W)$ is in $H^{s+2}_\kappa$ by estimating the first $s+1\geq 0$ derivatives in $H^1_\kappa$ by duality.  The Green's function for a translated potential is the translation of the original Green's function:
	\begin{equation}
	g(x;\kappa,q(\cdot+h)) = g(x+h;\kappa,q) \quad\tx{for all }h\in\R .
	\label{eq:g trans prop}
	\end{equation}
	Differentiating~\eqref{eq:g trans prop} at $h=0$ and using the resolvent identity, we have
	\begin{equation}
	g^{(j)}(x;\kappa,W+q)
	= \sum_{\ell=0}^\infty (-1)^\ell \big\langle \del_x, [\partial^j, R(\kappa,W) ( q R(\kappa,W) )^\ell\del_x ] \big\rangle .
	\label{eq:g trans prop 2}
	\end{equation}
	Here, $[A,B] = AB-BA$ denotes the commutator and $\partial^j$ denotes $j$ spatial partial derivatives.  Within the summand there are $\ell+1$ factors of $R(\kappa,W)$, and we expand each into the series~\eqref{eq:resolvent series 1} in powers of $W$ indexed by $m_i$.  For $j=0,\dots,s+1$ and $f\in H^{-1}_\kappa$, this yields
	\begin{align*}
	&\left| \int f(x) [ g(\kappa,W+q) - g(\kappa,W) ]^{(j)} (x) \dx \right| \\
	&\leq \sum_{\ell=1}^\infty \sum_{m_0,\dots,m_\ell = 0}^\infty \big| \tr\big\{ f [\partial^j, R_0 (WR_0)^{m_0} qR_0 \cdots qR_0 (WR_0)^{m_\ell} ] \big\} \big| .
	\end{align*}
	We distribute the derivatives $[\partial^j,\cdot]$ using the product rule.  We use the operator estimate~\eqref{eq:hskappa identity 1} for each factor of $\sqrt{R_0}\, q\sqrt{R_0}$ and estimate the remaining factors in operator norm.  Given a multiindex $\sigma\in\N^\ell$ with $|\sigma| \leq j$, H\"older's inequality in Fourier variables yields
	\begin{equation*}
	\prod_{i=1}^\ell \snorm{ q^{(\sigma_j)} }_{H^{-1}_\kappa} 
	\leq \snorm{ q^{(|\sigma|)} }_{H^{-1}_\kappa} \snorm{ q }_{H^{-1}_\kappa}^{\ell-1}
	\leq \snorm{ q }_{H^{j-1}_\kappa} \snorm{ q }_{H^{-1}_\kappa}^{\ell-1}. 
	\end{equation*}
	As $j\leq s+1$, we have
	\begin{align*}
	&\left| \int f(x) [ g(\kappa,W+q) - g(\kappa,W) ]^{(j)} (x) \dx \right| \\
	&\leq \sum_{\ell=1}^\infty \sum_{m_0,\dots,m_\ell = 0}^\infty \frac{ \norm{f}_{H^{-1}_\kappa} \norm{q}_{H^{s}_\kappa} }{\kappa} \left( \frac{ \norm{q}_{H^{-1}_\kappa} }{\kappa^{1/2}} \right)^{\ell-1} \left( \frac{ \norm{W}_{W^{s+1,\infty}} }{\kappa^{2}} \right)^{m_0+\dots+m_\ell} .
	\end{align*}
	First we perform the inner sum over $m_0,\dots,m_\ell$; re-indexing $m = m_0 + \dots + m_\ell$, we have
	\begin{equation}
	\begin{aligned} 
	\sum_{m_{0}, \ldots, m_{\ell} \geq 0} \left( \frac{\norm{W}_{W^{s+1,\infty}}}{\kappa^{2}} \right)^{m_{0}+\cdots+m_{\ell}}
	&= \sum_{m=0}^{\infty} \frac{(\ell+m) !}{\ell !\, m !} \left( \frac{\norm{W}_{W^{s+1,\infty}}}{\kappa^{2}} \right)^{m} \\ 
	&\lesssim \left( 1 - \frac{\norm{W}_{W^{s+1,\infty}}}{\kappa^{2}} \right)^{\ell+1} \leq 1
	\end{aligned}
	\label{eq:double sum inner}
	\end{equation}
	uniformly in $\ell$, provided that $\kappa \gg \norm{W}_{W^{s+1,\infty}}^{1/2}$.  The sum over $\ell \geq 1$ then converges uniformly for $\kappa \gg A^2$, yielding
	\begin{equation*}
	\left| \int f [ g(\kappa,W+q) - g(\kappa,W) ]^{(j)} \dx \right|
	\lesssim \kappa^{-1} \norm{f}_{H^{-1}_{\kappa}} \norm{q}_{H^{s}_\kappa} \quad\tx{for }j=0,\dots,s+1 .
	\end{equation*}
	Taking a supremum over $\norm{f}_{H^{-1}_{\kappa}} \leq 1$, we obtain the estimate~\eqref{eq:diffeo prop}.
\end{proof}

As an offspring of the resolvent $R(\kappa, q)$, the diagonal Green's function comes with some algebraic identities.  In particular, in~\cite[Lem.~2.5--2.6]{Killip2019} it is shown that for Schwartz $q$ we have
\begin{equation}
\int \frac{G(x,y;\kappa,q) G(y,x;\kappa,q)}{2g(y;\kappa,q)^2}\dy = g(x;\kappa,q)
\label{eq:g alg prop 1}
\end{equation}
and
\begin{equation}
\begin{aligned}
&\int G(x,y;\kappa,q) \big[ - f''' + 2qf' + 2(qf)' + 4\kappa^2f' \big](y) G(y,x;\kappa,q)\dy \\
&= 2f'(y) g(x;\kappa,q) - 2f(x)g'(x;\kappa,q)
\end{aligned}
\label{eq:g alg prop 2}
\end{equation}
for all Schwartz $f$.  As is suggested by taking $f = g(\kappa,q)$ in~\eqref{eq:g alg prop 2}, multiplying by $1/2g(x;\kappa,q)^2$, and integrating in $x$, the diagonal Green's function satisfies the ODE
\begin{equation}
g'''(\kappa,q) = 2qg'(\kappa,q) + 2\left[ q g(\kappa,q) \right]' + 4\kappa^2g'(\kappa,q) ;
\label{eq:g alg prop 3}
\end{equation}
see~\cite[Prop.~2.3]{Killip2019} for a proof.

Ultimately, the convergence of the approximate flows will be dominated by the linear and quadratic terms of the series~\eqref{eq:resolvent series 1} for the diagonal Green's function.  Consequently, we will now record some useful operator identities for these two terms:
\begin{lem}
	For $\kappa\geq 1$ we have the operator identities
	\begin{align}
		&16\kappa^5 \langle \del_x, R_0 f R_0 \del_x \rangle = 16\kappa^4 R_0(2\kappa) f = \left[ 4\kappa^2 + \partial^2 + R_0(2\kappa) \partial^4 \right] f ,
		\label{eq:linear op id} \\
		&\begin{aligned} 16\kappa^5 \langle \del_x, R_0 f R_0 h R_0 \del_x \rangle &= 3fh - 3 [R_0(2\kappa)f''][R_0(2\kappa)h''] \\
			&\phantom{={}}+ 4\kappa^2 [R_0(2\kappa)f'][R_0(2\kappa)h'] ( -5 + R_0(2\kappa)\partial^2 ) \\
			&\phantom{={}}+ 4\kappa^2 [R_0(2\kappa)f][R_0(2\kappa)h] (5 \partial^2 + 2 R_0(2\kappa)\partial^4 ) , \end{aligned}
		\label{eq:quad op id}
	\end{align}
	where $R_0 = R_0(\kappa)$.
\end{lem}
\begin{proof}
	From the integral kernel formula for $R_0(\kappa)$ we see that $\langle\del_x,R_0f R_0\del_x \rangle = \kappa^{-1}R_0(2\kappa) f$, which demonstrates the first equality of~\eqref{eq:linear op id}.  The second equality follows from the symbol identity
	\begin{equation*}
	\frac{16 \kappa^{4}}{\xi^{2}+4\kappa^{2}} = 4 \kappa^{2} - \xi^{2} + \frac{\xi^{4}}{\xi^{2}+4 \kappa^{2}} 
	\end{equation*}
	in Fourier variables.
	
	Now we turn to the second identity~\eqref{eq:quad op id}.  In~\cite[Appendix]{Killip2019} the Fourier transform of LHS\eqref{eq:quad op id} is found to be
	\begin{equation*}
	\ft\big( \tx{LHS}\eqref{eq:quad op id} \big)(\xi) = \frac{8\kappa^4}{\sqrt{2 \pi}} \int_{\R} \frac{ [\xi^{2}+(\xi-\eta)^{2}+\eta^{2}+24 \kappa^{2}] \hat{f}(\xi-\eta) \hat{h}(\eta) }{ (\xi^{2}+4 \kappa^{2}) ((\xi-\eta)^{2}+4 \kappa^{2}) (\eta^{2}+4 \kappa^{2}) } \deta .
	\end{equation*}
	The operator identity~\eqref{eq:quad op id} then follows from the equality
	\begin{align*}
	&\frac{8 \kappa^{4}\left[\xi^{2}+(\xi-\eta)^{2}+\eta^{2}+24 \kappa^{2}\right]}{(\xi^{2}+4 \kappa^{2})((\xi-\eta)^{2}+4 \kappa^{2})(\eta^{2}+4 \kappa^{2})} = 3 - \frac{3 \eta^{2}(\xi-\eta)^{2}}{((\xi-\eta)^{2}+4 \kappa^{2})(\eta^{2}+4 \kappa^{2})} \\
	&- \frac{20 \kappa^{2}\left[-\eta(\xi-\eta)+\xi^{2}\right]}{((\xi-\eta)^{2}+4 \kappa^{2})(\eta^{2}+4 \kappa^{2})} + \frac{4 \kappa^{2} \xi^{2}\left[\eta(\xi-\eta)+2 \xi^{2}\right]}{(\xi^{2}+4 \kappa^{2})((\xi-\eta)^{2}+4 \kappa^{2})(\eta^{2}+4 \kappa^{2})} . \qedhere\\
	\end{align*}
\end{proof}

We will also need to know that after extracting the linear and quadratic terms from $\kappa^5 g(\kappa,q+W)$, the remainder tends to zero as $\kappa\to\infty$:
\begin{lem}
	Given an integer $s\geq 1$ and $A>0$, we have
	\begin{equation}
	\begin{aligned}
	\kappa^5 \big\lVert \big\{ &g(\kappa,q+W) + \langle \del_x , R_0(q+W)R_0 \del_x \rangle \\
	&- \langle \del_x , R_0(q+W)R_0(q+W)R_0 \del_x \rangle \big\}^{(s+1)} \big\rVert_{L^2} \to 0 \quad\tx{as }\kappa \to \infty
	\end{aligned}
	\label{eq:tail conv}
	\end{equation}
	uniformly for $\norm{q}_{H^{s}} \leq A$.
\end{lem}
\begin{proof}
	We estimate the $s$th derivative in $H^1$ by duality.  Differentiating the translation identity~\eqref{eq:g trans prop 2} at $h=0$, we have
	\begin{equation*}
	g^{(s)}(x;\kappa,W+q)
	= \sum_{\ell=0}^\infty (-1)^\ell \big\langle \del_x, [\partial^{s}, R(\kappa,W) ( q R(\kappa,W) )^\ell\del_x ] \big\rangle .
	\end{equation*}
	Within the summand there are $\ell+1$ factors of $R(\kappa,W)$, and we expand each into the series~\eqref{eq:resolvent series 1} in powers of $W$ indexed by $m_i$.  For $f\in H^{-1}$ this yields
	\begin{equation}
	\begin{aligned}
	&\begin{aligned}
	\kappa^5 \bigg| \int f(x) \big\{ &g(\kappa,q+W) + \langle \del_x , R_0(q+W)R_0 \del_x \rangle \\[-0.5em]
	&- \langle \del_x , R_0(q+W)R_0(q+W)R_0 \del_x \rangle \big\}^{(s)} \dx \bigg| 
	\end{aligned} \\
	&\leq \kappa^5 \sum_{\substack{ \ell\geq 0,\ m_0,\dots,m_\ell\geq 0 \\ \ell+m_0+\dots+m_\ell \geq 3 }} \big| \tr\big\{ f [\partial^{s}, R_0 (WR_0)^{m_0} qR_0 \cdots qR_0 (WR_0)^{m_\ell} ] \big\} \big| .
	\end{aligned}
	\label{eq:tail 1}
	\end{equation}
	We distribute the derivatives $[\partial^{s},\cdot]$ using the product rule.  We then use the operator estimate~\eqref{eq:hskappa identity 1} and the observation $\norm{ f }_{H^{-1}_\kappa} \lesssim \kappa^{-1} \norm{ f }_{L^2}$ to put the highest order $q$ in $L^2$.  In the instance that there are no factors of $q$, we put the highest order $W$ term in $L^2$ and use that $W'$ is in $H^{s-1}$.  We then estimate all other terms in operator norm; the remaining factors of $q$ have at most $s-1$ derivatives, and thus may be estimated in $L^\infty$ via the embedding $H^1\hookrightarrow L^\infty$.  This yields
	\begin{align*}
	\tx{RHS}\eqref{eq:tail 1}
	\lesssim \kappa^5 \sum_{\substack{ \ell\geq 0,\ m_0,\dots,m_\ell\geq 0 \\ \ell+m_0+\dots+m_\ell \geq 3 }}\ &\frac{ \norm{f}_{H^{-1}} }{\kappa^{1/2}} \frac{ \max\{ \norm{q}_{H^s} , \norm{W'}_{H^{s-1}} \} }{\kappa^{3/2}} \\[-1.5em]
	&\times \left( \frac{ \max\{ \norm{q}_{H^s} , \norm{W}_{W^{s,\infty}} \} }{\kappa^{2}} \right)^{\ell + m_0+\dots+m_\ell - 1} .
	\end{align*}
	We re-index $m = m_0+\dots+m_\ell$ and sum over $\ell + m \geq 3$ as in~\eqref{eq:double sum inner}.  The sum converges provided $\kappa \gg \norm{q}_{H^s}^{1/2}$ and $\kappa \gg \norm{W}_{W^{s,\infty}}^{1/2}$.  The condition $\ell+m\geq 3$ guarantees that when we sum over the parenthetical term we gain a factor $\lesssim (\kappa^{-2})^2$, and so we obtain
	\begin{equation*}
	\tx{RHS}\eqref{eq:tail 1}
	\lesssim \kappa^{-1} \norm{f}_{H^{-1}}
	\end{equation*}
	uniformly for $\norm{q}_{L^2} \leq A$ and $\kappa \geq \kappa_0(A)$.  The claim~\eqref{eq:tail conv} follow by taking a supremum over $\norm{f}_{H^{-1}} \leq 1$.
\end{proof}

\section{Tidal \texorpdfstring{$H_\kappa$}{H-kappa} flow}
\label{sec:thk}

The argument of~\cite{Killip2019} relies upon the Hamiltonians $H_\kappa$ whose flows approximate that of KdV as $\kappa\to\infty$.  Specifically, in~\cite[Prop.~3.2]{Killip2019} it is shown that the $H_\kappa$ flow can be expressed in terms of the diagonal Green's function as
\begin{equation}
\ddt u = 16\kappa^5 g'(\kappa,u) + 4\kappa^2 u' .
\label{eq:hk flow}
\end{equation}
Moreover, the flows at any two energy parameters $\kappa$ and $\varkappa$ commute:
\begin{equation}
\{ H_\kappa , H_\varkappa \} = 0 .
\label{eq:hk flow 2}
\end{equation}

We need an analogous approximate flow for step-like initial data.  Mimicking how we obtained tidal KdV from KdV, we subtract the background $W$ from the function $u$ to obtain the tidal $H_\kappa$ flow
\begin{equation}
\ddt q = 16\kappa^5 g'(\kappa,q+W) + 4\kappa^2 (q+W)'
\label{eq:thk}
\end{equation}
for $q := u-W$.  The tidal $H_\kappa$ flow is also Hamiltonian; however, we will not make use of its Hamiltonian.

In this section we will show that the tidal $H_\kappa$ flow is globally well-posed in $H^s$ for all integers $s\geq 0$.  We restrict our attention to integer $s$ since the result for non-integer $s\geq 0$ follows from interpolation.  Once we obtain well-posedness, the commutativity~\eqref{eq:hk flow 2} of the $H_\kappa$ flows implies that any two tidal $H_\kappa$ flows commute with each other.

We begin with local well-posedness.  The $H_\kappa$ flows are easier to work with because local well-posedness follows from a contraction mapping argument.
\begin{lem}
	\label{thm:thk lwp}
	Given an integer $s\geq -1$ and $A>0$, there exists a constant $\kappa_0$ so that for $\kappa \geq \kappa_0$ the tidal $H_\kappa$ flows~\eqref{eq:thk} with initial data in the closed ball $B_A \subset H^s(\R)$ of radius $A$ are locally well-posed.
\end{lem}
\begin{proof}
	Fix an integer $s\geq -1$.  The solution $q(t)$ to the tidal $H_\kappa$ flow satisfies the integral equation
	\begin{equation*}
	q(t) = e^{t4\kappa^2\partial_x} q(0) + \int_0^t e^{(t-\tau) 4\kappa^2 \partial_x} \left[ 16\kappa^5 g'(\kappa,q(\tau) + W) + 4\kappa^2 W' \right]\,\dd\tau .
	\end{equation*}
	A contraction mapping argument proves local well-posedness, provided we have the Lipschitz estimate
	\begin{align*}
	&\norm{ g'(\kappa,q + W) - g'(\kappa,\tilde{q} + W) }_{H^s} \\
	&\lesssim \norm{ [g(\kappa,q + W) - g(\kappa,W)] - [g(\kappa,\tilde{q} + W) - g(\kappa,W)] }_{H^{s+2}} 
	\lesssim \norm{ q - \tilde{q} }_{H^s}
	\end{align*}
	uniformly on bounded subsets of $H^s$.
	
	Fix $A>0$.  It suffices to show that $f \mapsto \dd [ g(\kappa, \cdot + W)] |_q (f)$ is bounded $H^{s} \to H^{s+2}$ uniformly for $\norm{q}_{H^s} \leq A$.  Using the resolvent identity we calculate
	\begin{equation*}
	d[ g(\kappa, \cdot + W) ]|_q (f)
	= - \langle \del_x , R(\kappa,q+W) f R(\kappa,q+W) \del_x \rangle .
	\end{equation*}
	Just as we did for the single resolvent $\langle \del_x , R(\kappa,q+W) \del_x \rangle$ in~\eqref{eq:diffeo prop}, we estimate the first $s+1$ derivatives in $H^1$ by duality and expand each resolvent into a series.  We conclude that there exists a constant $\kappa_0$ such that
	\begin{equation*}
	\norm{ d[g(\kappa, \cdot + W) ]|_q (f) }_{H^{s+2}} \lesssim \norm{f}_{H^{s}}
	\end{equation*}
	uniformly for $q\in B_A$ and $\kappa\geq \kappa_0$.	
\end{proof}

In order to obtain global well-posedness, we will prove \ti{a priori} estimates in $H^s$ for all integers $s\geq 0$.  Our energy arguments are inspired by those of Bona and Smith~\cite{Bona1975}.  The family of BBM equations which Bona--Smith uses to approximate the KdV flow does not conserve the polynomial conserved quantities of KdV.  One benefit of our method is that in the case $W\equiv 0$, the $H_\kappa$ flows do conserve these quantities (as is suggested by the asymptotic expansion~\eqref{eq:alpha intro} and Poisson commutativity), and consequently the \ti{a priori} estimates are identical to that of KdV.  In particular, in the case $W\equiv 0$ we obtain a new proof of the Bona--Smith theorem using the low-regularity methods from~\cite{Killip2019}.  (This is not subsumed by~\cite[Cor.~5.3]{Killip2019}, which only addresses $H^s(\R)$ for $s\in [-1,0)$.)

Our energy arguments are much simplified in the case $\kappa = \infty$, where the tidal $H_\kappa$ flow becomes tidal KdV.  Our manipulations are motivated by the corresponding tidal KdV terms at $\kappa = \infty$, where operations involving commutators and cycling the trace correspond to more elementary operations involving integration by parts.  In particular, the reason for the restriction $s\geq 3$ is the same as in~\cite{Bona1975}: when estimating $\ddt \snorm{ q^{(s)}(t) }_{L^2}^2$ under the KdV flow, $s=3$ is the smallest integer for which the nonlinear contribution can be estimated in terms of $\snorm{ q^{(s)}(t) }_{L^2}^2$ provided that we already control $q(t)$ in $H^{s-1}$.

We begin with $s=0$:
\begin{prop}
	\label{thm:a priori L2}
	Given $A,T>0$ there exist constants $C$ and $\kappa_0$ such that solutions $q_\kappa(t)$ to the tidal $H_\kappa$ flow~\eqref{eq:thk} obey
	\begin{equation*}
	\norm{q(0)}_{L^2} \leq A
	\quad\implies\quad
	\norm{q_\kappa(t)}_{L^2} \leq C\quad\tx{for all }|t|\leq T\tx{ and }\kappa \geq \kappa_0 .
	\end{equation*}
\end{prop}
\begin{proof}
	By approximation and local well-posedness we may assume that $q(0) \in H^\infty$.  Let
	\begin{equation*}
	E_0(t) := \tfrac{1}{2} \int q_\kappa(t,x)^2 \dx .
	\end{equation*}
	This is the first polynomial conserved quantity of the KdV hierarchy, and in the case $W\equiv 0$ one can directly show that $\ddt E_0 = 0$ under the $H_\kappa$ flow using the ODE~\eqref{eq:g alg prop 3} satisfied by the diagonal Green's function.
	
	To counteract the factor of $\kappa^5$ in the tidal $H_\kappa$ flow and obtain a bound for all $\kappa$ large, we will extract the linear and quadratic terms.  Using the translation identity~\eqref{eq:g trans prop 2}, we write
	\begin{align}
	&\ddt E_0
	\nonumber \\
	&= \int q_\kappa \big\{ {-16}\kappa^5 \langle\del_x, R_0q'_\kappa R_0 \del_x \rangle + 4\kappa^2 q'_\kappa \big\} \dx
	\label{eq:a priori L2 1} \\
	&\phantom{={}}+ \int q_\kappa \big\{ {-16}\kappa^5 \langle\del_x, R_0W'R_0 \del_x \rangle + 4\kappa^2 W' \big\} \dx 
	\label{eq:a priori L2 2} \\
	&\phantom{={}}+ 16\kappa^5 \int q_\kappa \langle\del_x, [\partial, R_0q_\kappa R_0q_\kappa R_0] \del_x \rangle \dx
	\label{eq:a priori L2 3} \\
	&\phantom{={}}+ 16\kappa^5 \int q_\kappa \big\{ \langle\del_x, [\partial, R_0WR_0q_\kappa R_0] \del_x \rangle + \langle\del_x, [\partial, R_0q_\kappa R_0WR_0] \del_x \rangle \big\} \dx
	\label{eq:a priori L2 4} \\
	&\phantom{={}}+ 16\kappa^5 \int q_\kappa \langle\del_x, [\partial, R_0W R_0W R_0] \del_x \rangle \dx
	\label{eq:a priori L2 6} \\
	&\begin{aligned}
	\phantom{={}}+ 16\kappa^5 \int q_\kappa \big\{ g(\kappa,q_\kappa+W) &+ \langle \del_x , R_0(q_\kappa+W)R_0 \del_x \rangle \\[-0.5em]
	&- \langle \del_x , R_0(q_\kappa+W)R_0(q_\kappa+W)R_0 \del_x \rangle \big\}' \dx .
	\end{aligned}
	\label{eq:a priori L2 5}
	\end{align}
	We will estimate the terms \eqref{eq:a priori L2 1}--\eqref{eq:a priori L2 5} separately.
	
	The first linear contribution~\eqref{eq:a priori L2 1} vanishes.  Indeed, using the first operator identity of~\eqref{eq:linear op id} we write
	\begin{equation*}
	\eqref{eq:a priori L2 1}
	= \int q_\kappa \big\{ {-16}\kappa^4 R_0(2\kappa) q'_\kappa + 4\kappa^2 q'_\kappa \big\} \dx .
	\end{equation*}
	This vanishes because the integrand is odd in Fourier variables, or equivalently the integrand is a total derivative.
	
	Now we estimate the linear contribution~\eqref{eq:a priori L2 2} from $W$.  Using the operator identity~\eqref{eq:linear op id} we write
	\begin{align*}
	|\eqref{eq:a priori L2 2}|
	&= \left| \int q_\kappa \big\{ - W''' - \big[ R_0(2\kappa) W^{(5)} \big] \big\} \dx \right| \\
	&\lesssim \norm{ q_\kappa }_{L^2} \big( \norm{W'''}_{L^2} + \kappa^{-2} \snorm{W^{(5)}}_{L^2} \big)
	\lesssim E_0^{1/2} \lesssim E_0 + 1 .
	\end{align*}
	Note that $W'$ is Schwartz, and we allow our implicit constants to depend on the fixed function $W$.
	
	The first quadratic contribution~\eqref{eq:a priori L2 3} also vanishes.  Distributing the derivative $[\partial,\cdot]$ and noting that $[\partial,R_0] = 0$, we write
	\begin{equation*}
	\eqref{eq:a priori L2 3}
	= 16\kappa^5 \big( \tr\{ q_\kappa R_0 [\partial,q_\kappa] R_0 q_\kappa R_0 \} + \tr\{ q_\kappa R_0q_\kappa R_0 [\partial,q_\kappa] R_0 \} \big) .
	\end{equation*}
	Both of these terms vanish by cycling the trace.
	
	Next we turn to the second quadratic contribution~\eqref{eq:a priori L2 4}.  By linearity and cycling the trace, we can ``integrate by parts'' to write
	\begin{align*}
	\eqref{eq:a priori L2 4}
	&= 16\kappa^5 \big( {- \tr}\{ [\partial,q_\kappa] R_0WR_0 q_\kappa R_0 \} + \tr\{ q_\kappa [ \partial , R_0q_\kappa R_0WR_0 ] \}  \big) \\
	&= 16\kappa^5 \tr\{ q_\kappa R_0q_\kappa R_0[\partial,W]R_0 \} .
	\end{align*}
	Using the estimate~\eqref{eq:hskappa identity 1} and the observations $\snorm{ \sqrt{R_0} }_\op \lesssim \kappa^{-1}$ and $\norm{f}_{H^{-1}_\kappa} \lesssim \kappa^{-1} \norm{f}_{L^2}$, we estimate
	\begin{equation*}
	|\eqref{eq:a priori L2 4}|
	\lesssim \kappa^5 \norm{ \sqrt{R_0} q_\kappa \sqrt{R_0} }_{\I_2}^2 \norm{ \sqrt{R_0} W' \sqrt{R_0} }_\op
	\lesssim \norm{W'}_{L^\infty} E_0 .
	\end{equation*}

	The quadratic $W$ contribution~\eqref{eq:a priori L2 6} is easily estimated.  We distribute the derivative and estimate
	\begin{equation*}
	|\eqref{eq:a priori L2 6}|
	\lesssim \kappa^5 \norm{ \sqrt{R_0} q_\kappa \sqrt{R_0} }_{\I_2} \norm{ \sqrt{R_0} W' \sqrt{R_0} }_{\I_2} \norm{ \sqrt{R_0} W \sqrt{R_0} }_{\op} .
	\end{equation*}
	Using the identity~\eqref{eq:hskappa identity 1} and the observation $\norm{f}_{H^{-1}_\kappa} \lesssim \kappa^{-1} \norm{f}_{L^2}$, we obtain
	\begin{equation*}
	|\eqref{eq:a priori L2 6}|
	\lesssim E_0^{1/2} \lesssim E_0 + 1 .
	\end{equation*}
	
	For the series tail~\eqref{eq:a priori L2 5}, we integrate by parts once to put the derivative on $q_\kappa$ and we write
	\begin{align*}
	|\eqref{eq:a priori L2 5}|
	&\leq 16\kappa^5 \sum_{\substack{ \ell\geq 0,\ m_0,\dots,m_\ell\geq 0 \\ \ell+m_0+\dots+m_\ell \geq 3 }} \big| \tr\big\{ q'_\kappa R_0 (WR_0)^{m_0} qR_0 \cdots qR_0 (WR_0)^{m_\ell} ] \big\} \big| .
	\intertext{Observe that the summand vanishes for $m_0 + \dots + m_\ell = 0$ by writing $q'_\kappa = [\partial,q_\kappa]$ and cycling the trace, and so we may insert the condition $m_0+\dots+m_\ell \geq 1$ in the summation.  We use the operator estimate~\eqref{eq:hskappa identity 1} and the observation $\norm{ f }_{H^{-1}_\kappa} \lesssim \kappa^{-1} \norm{ f }_{L^2}$ to put each factor of $q$ in $L^2$, and we put all other factors in operator norm:}
	&\lesssim \kappa^5 \sum_{\substack{ \ell\geq 0,\ m_0+\dots+m_\ell \geq 1 \\ \ell+m_0+\dots+m_\ell \geq 3 }} \frac{ \norm{q'_\kappa}_{H^{-1}} }{\kappa^{1/2}} \left( \frac{ \norm{q}_{L^2} }{\kappa^{3/2}} \right)^{\ell} \left( \frac{ \norm{W}_{L^\infty} }{\kappa^{2}} \right)^{m_0+\dots+m_\ell} .
	\intertext{We split the sum into $\ell = 0$, $\ell=1$, $\ell=2$, and $\ell\geq 3$ terms.  We then re-index $m = m_0+\dots+m_\ell$, sum over $m\geq 1$ as in~\eqref{eq:double sum inner}, and then sum in $\ell$.  The sum converges provided $\kappa \gg E_0^{1/3}(t)$ and $\kappa \gg \norm{W}_{L^\infty}^{1/2}$.  The conditions $m\geq 1$ and $\ell+m\geq 3$ guarantee that when we sum over the two parenthetical terms we gain a factor $\lesssim (\kappa^{-3/2})^2 (\kappa^{-2})$, and so we obtain}
	&\lesssim \kappa^{-1/2} \norm{q_\kappa}_{L^2}
	\end{align*}
	for all $\kappa$ large.  Taking a supremum over $\norm{f}_{H^{-1}} \leq 1$ and restricting to $\kappa$ sufficiently large, we conclude there exists $\kappa_0(E_0(t))$ such that
	\begin{equation*}
	| \eqref{eq:a priori L2 5} |
	\leq E_0^{1/2} \lesssim E_0 + 1
	\quad\tx{uniformly for }\kappa \geq \kappa_0(E_0(t)) .
	\end{equation*}
	
	Altogether, we have shown that there exist constants $C$ and $\kappa_0(E_0(t))$ such that
	\begin{equation*}
	\left| \ddt E_0 \right| \leq C( E_0 + 1 )
	\quad\tx{uniformly for }|t|\leq T\tx{ and }\kappa \geq \kappa_0(E_0(t)) .
	\end{equation*}
	Gr\"onwall's inequality then yields the bound
	\begin{equation*}
	E_0(t) \leq (E_0(0) + 1) e^{CT} - 1
	\quad\tx{uniformly for }|t|\leq T,\ \kappa \geq \kappa_0\big( (E_0(0) + 1) e^{CT} - 1 \big) ,
	\end{equation*}
	which concludes the proof.
\end{proof}

Next, we control the growth of the $H^1$ norm:
\begin{prop}
	\label{thm:a priori H1}
	Given $A,T>0$ there exist constants $C$ and $\kappa_0$ such that solutions $q_\kappa(t)$ to the tidal $H_\kappa$ flow~\eqref{eq:thk} obey
	\begin{equation*}
	\norm{q(0)}_{H^1} \leq A
	\quad\implies\quad
	\norm{q_\kappa(t)}_{H^1} \leq C\quad\tx{for all }|t|\leq T\tx{ and }\kappa \geq \kappa_0 .
	\end{equation*}
\end{prop}
\begin{proof}
	By approximation and local well-posedness we may assume that $q(0) \in H^\infty$.  Let
	\begin{equation*}
	E_1(t) := \int \big\{ \tfrac{1}{2} (q'_\kappa(t,x))^2 + q_\kappa(t,x)^3 \big\} \dx
	\end{equation*}
	denote the next polynomial conserved quantity of KdV.
	
	We multiply the tidal $H_\kappa$ flow~\eqref{eq:thk} by $-q''_\kappa + 3q_\kappa^2$ and integrate in space to obtain an expression for $\ddt E_1$.  We then integrate by parts to remove the derivative from $g(\kappa,q_\kappa+W) -g(\kappa,W)$, expand both diagonal Green's functions using the relation~\eqref{eq:g alg prop 1}, and apply the identity~\eqref{eq:g alg prop 2} to obtain
	\pagebreak
	\begin{align}
	&\ddt E_1
	\nonumber \\
	&= - \int q''_\kappa \big[ 16\kappa^5 g'(\kappa,W) + 4\kappa^2 W' \big] \dx 
	\label{eq:a priori H1 1}\\
	&\phantom{={}}+ 4\kappa^2 \int \big\{ 3W'q^2_\kappa + 16\kappa^5 q'_\kappa \big[ g(\kappa,q_\kappa+W) - g(\kappa,W) \big] \big\} \dx 
	\label{eq:a priori H1 2}\\
	&\phantom{={}}+ 16\kappa^5 \int \big[ 2Wq'_\kappa + 2(Wq_\kappa)' \big] \big[ g(\kappa,q_\kappa+W) - g(\kappa,W) \big] \dx .
	\label{eq:a priori H1 3}
	\end{align}
	Note that in the case $W\equiv 0$, all three integrals vanish and $E_1$ is conserved as expected.  We will estimate the terms \eqref{eq:a priori H1 1}--\eqref{eq:a priori H1 3} separately.
	
	We begin with the term~\eqref{eq:a priori H1 1}.  We integrate by parts once, expand $g(\kappa,W)$ in a series, and extract the linear term:
	\begin{align*}
	\eqref{eq:a priori H1 1}
	&= \int q'_\kappa \big[ {-16}\kappa^5 \langle \del_x,R_0W''R_0 \del_x \rangle + 4\kappa^2 W'' \big] \dx \\
	&\phantom{={}}+ 16\kappa^5 \sum_{m\geq 2} (-1)^m \tr\big\{ q'_\kappa [\partial^2, R_0 (WR_0)^m] \big\} .
	\end{align*}
	For the first term we use the operator identity~\eqref{eq:linear op id} to estimate
	\begin{align*}
	&\left| \int q'_\kappa \big[ {-16}\kappa^5 \langle \del_x,R_0W''R_0 \del_x \rangle + 4\kappa^2 W'' \big] \dx \right| \\
	&= \left| \int q'_\kappa \big[ - W^{(4)} - R_0(2\kappa) W^{(6)} \big] \dx \right| 
	\lesssim \norm{q'_\kappa}_{L^2} \big( \snorm{ W^{(4)} }_{L^2} + \kappa^{-2} \snorm{ W^{(6)} } \big) .
	\end{align*}
	For the second term we distribute the two derivatives $[\partial^2,\cdot]$, use the estimate~\eqref{eq:hskappa identity 1} and the observation $\norm{ f }_{H^{-1}_\kappa} \lesssim \kappa^{-1} \norm{ f }_{L^2}$ to put $q'_\kappa$ and the highest order $W$ term in $L^2$, and put the remaining terms in operator norm:
	\begin{align*}
	&16\kappa^5 \sum_{m\geq 2} \big| \tr\big\{ q'_\kappa [\partial^2, R_0 (WR_0)^m] \big\} \big| \\
	&\lesssim \kappa^5 \sum_{m\geq 2} m^2 \frac{\norm{q'_\kappa}_{L^2}}{\kappa^{3/2}} \frac{\norm{W'}_{H^1}}{\kappa^{3/2}} \left( \frac{\norm{W}_{W^{1,\infty}}}{\kappa^{2}} \right)^{m-1}
	\lesssim \norm{W'}_{H^1} \norm{W}_{W^{1,\infty}} \norm{ q'_\kappa }_{L^2}
	\end{align*}
	uniformly for $\kappa \gg \norm{W}^{1/2}_{W^{1,\infty}}$.  Altogether we conclude
	\begin{equation*}
	|\eqref{eq:a priori H1 1}| \lesssim \norm{q'_\kappa}_{L^2}^2 + 1
	\end{equation*}
	uniformly for $\kappa$ large.
	
	Next we turn to the term~\eqref{eq:a priori H1 2}.  Expanding $g(\kappa,q_\kappa+W)$ and extracting the terms that are linear and quadratic in $q_\kappa$ and $W$, we write
	\begin{align}
	&\eqref{eq:a priori H1 2}
	\nonumber \\
	&= - 64\kappa^7 \int q'_\kappa \langle \del_x, R_0q_\kappa R_0 \del_x \rangle \dx 
	\label{eq:a priori H1 4p1}\\
	&\phantom{={}}+ 64\kappa^7 \int q'_\kappa \langle \del_x, R_0q_\kappa R_0q_\kappa R_0 \del_x \rangle \dx 
	\label{eq:a priori H1 4p2}\\
	&\begin{aligned}
	\phantom{={}}+ 4\kappa^2\int \big\{ 16\kappa^5 q'_\kappa \big( \langle \del_x, R_0WR_0q_\kappa R_0\del_x \rangle + \langle \del_x, R_0q_\kappa R_0WR_0 \del_x \rangle \big)& \\[-0.5em]
	{}+ 3W'q_\kappa^2& \big\} \dx
	\end{aligned}
	\label{eq:a priori H1 5}\\
	&\begin{aligned}
	\phantom{={}}+ 64\kappa^7 \sum_{\substack{ \ell\geq 1,\ m_0,\dots,m_\ell\geq 0 \\ \ell+m_0+\dots+m_\ell \geq 3 }} (-1)^{\ell+m_0+\dots+m_\ell} \tr\big\{ q'_\kappa R_0(WR_0)^{m_0} q_\kappa R_0 \cdots& \\[-1.5em]
	\times q_\kappa R_0 (W R_0)^{m_\ell}& \big\} .
	\end{aligned}
	\label{eq:a priori H1 6}
	\end{align}
	The terms~\eqref{eq:a priori H1 4p1} and~\eqref{eq:a priori H1 4p2} vanish by cycling the trace:
	\begin{align*}
	\eqref{eq:a priori H1 4p1}
	&= - 64\kappa^7 \tr\{ [\partial, q_\kappa] R_0 q_\kappa R_0 \} 
	= 0 , \\
	\eqref{eq:a priori H1 4p2}
	&= 64\kappa^7 \tr\{ [\partial, q_\kappa] R_0 q_\kappa R_0 q_\kappa R_0 \} 
	= 0 .
	\end{align*}
	
	For the term~\eqref{eq:a priori H1 5}, we integrate by parts to replace $3W'q_\kappa^2$ by $-6Wq_\kappa q'_\kappa$.  We then use the operator identity~\eqref{eq:quad op id} and the estimates $\snorm{ R_{0}(2 \kappa) \partial^j }_\op \lesssim \kappa^{j-2}$ for $j=0,1,2$ (the estimate for $j = 0$ is also true as an operator on $L^\infty$ by the explicit kernel formula for $R_0$ and Young's inequality) to conclude
	\begin{equation*}
	|\eqref{eq:a priori H1 5}| \lesssim \norm{ q'_\kappa }_{L^2}^2 + 1 .
	\end{equation*}
	
	For the tail~\eqref{eq:a priori H1 6} we estimate
	\begin{align*}
	|\eqref{eq:a priori H1 6}|
	&\leq 64\kappa^7\sum_{\substack{ \ell\geq 1,\ m_0,\dots,m_\ell\geq 0 \\ \ell+m_0+\dots+m_\ell \geq 3 }} \big| \tr \big\{ q'_\kappa R_0(W R_0)^{m_0} q_\kappa R_0 \cdots q_\kappa R_0 (W R_0)^{m_\ell} \big\} \big| .
	\intertext{We put $q'_\kappa$ and one other $q_\kappa$ in $L^2$ via the estimate~\eqref{eq:hskappa identity 1} and put the remaining terms in operator norm.  We have $\norm{q}_{L^2} \lesssim 1$ uniformly for $|t|\leq T$ and $\kappa$ large by \cref{thm:a priori L2}, and so we obtain}
	&\lesssim \kappa^7\sum_{\substack{ \ell\geq 1,\ m_0,\dots,m_\ell\geq 0 \\ \ell+m_0+\dots+m_\ell \geq 3 }} \frac{ \norm{ q'_\kappa }_{L^2} }{ \kappa^{3} } \left( \frac{\norm{q_\kappa}_{L^\infty}}{ \kappa^2 } \right)^{\ell-1} \left( \frac{\norm{W}_{L^\infty}}{\kappa^2} \right)^{m_0+\dots+m_\ell} .
	\intertext{The condition $\ell+m_0+\dots+m_\ell \geq 3$ yields a gain $\lesssim (\kappa^{-2})^2$ when we sum over the two parenthetical terms, and so we obtain}
	&\lesssim \norm{ q'_\kappa }_{L^2} \lesssim \norm{ q'_\kappa }_{L^2}^2 + 1
	\end{align*}
	provided that $\kappa \gg \norm{ q_\kappa }^{1/2}_{L^\infty}$ and $\kappa \gg \norm{W}^{1/2}_{L^\infty}$.  From \cref{thm:a priori L2} we know that
	\begin{equation}
	\norm{q_\kappa}_{L^2} \lesssim 1 , \qquad
	\norm{q_\kappa}_{L^\infty}
	\leq \norm{q_\kappa}_{L^2}^{1/2} \norm{q'_\kappa}_{L^2}^{1/2}
	\lesssim \norm{q'_\kappa}_{L^2}^{1/2}
	\label{eq:a priori H1 7}
	\end{equation}
	for $\kappa\geq \kappa_0(T, \norm{q(0)}_{L^2} )$ sufficiently large, and so altogether we conclude
	\begin{equation*}
	|\eqref{eq:a priori H1 2}|
	\lesssim \norm{ q'_\kappa}_{L^2}^2 + 1
	\quad\tx{uniformly for }\kappa \geq \kappa_0( \norm{q_\kappa}_{H^1} ) .
	\end{equation*}

	It remains to estimate the term~\eqref{eq:a priori H1 3}.  Expanding $g(\kappa,q_\kappa+W) - g(\kappa, q_\kappa+W)$ and extracting the linear term, we write
	\begin{align}
	&|\eqref{eq:a priori H1 3}|
	\nonumber \\
	&\leq \left| 32\kappa^5 \int \big[ Wq'_\kappa + (Wq_\kappa)' \big] \langle \del_x , R_0q_\kappa R_0 \del_x\rangle \dx \right| 
	\label{eq:a priori H1 8}\\
	&\begin{aligned}
	\phantom{\leq{}}+ 32\kappa^5\sum_{\substack{ \ell\geq 1,\ m_0,\dots,m_\ell\geq 0 \\ \ell+m_0+\dots+m_\ell\geq 2}} \big| \tr \big\{ &\big[ Wq'_\kappa + (Wq_\kappa)' \big] \\[-1.5em]
	&\times R_0(W R_0)^{m_0} q_\kappa R_0 \cdots q_\kappa R_0 (W R_0)^{m_\ell} \big\} \big| .
	\end{aligned}
	\label{eq:a priori H1 9}
	\end{align}
	
	For the first term~\eqref{eq:a priori H1 8} we use the operator identity~\eqref{eq:linear op id} to write
	\begin{equation*}
	\eqref{eq:a priori H1 8} = 8\kappa^2 \int [ Wq'_\kappa + (Wq_\kappa)' ] q_\kappa \dx + \int [ Wq'_\kappa + (Wq_\kappa)' ] [ q''_\kappa + R_0(2\kappa) \partial^2 q''_\kappa ] \dx .
	\end{equation*}
	The first integral vanishes because the integrand is a total derivative.  For the second integral, we integrate by parts to obtain 
	\begin{align*}
	&\int [ Wq'_\kappa + (Wq_\kappa)' ] [ q''_\kappa + R_0(2\kappa) \partial^2 q''_\kappa ] \dx \\
	&= - \int [ 2 W' q'_\kappa + W''q_\kappa ] [ q'_\kappa + R_0(2\kappa) \partial^2 q'_\kappa ] \dx \\
	&\phantom{={}} + \int \big\{ Wq'_\kappa [ R_0(2\kappa) \partial^2 q''_\kappa ] - Wq''_\kappa [R_0(2\kappa) \partial^2 q''_\kappa] \big\} \dx .
	\end{align*}
	Those terms without $q''_\kappa$ can be estimated using Cauchy--Schwarz and the observation $\norm{ R_0(2\kappa) \partial^2 }_{\op} \lesssim 1$.  For the remaining terms, we ``integrate by parts'' in Fourier variables:
	\begin{align*}
	&\left| \int \big\{ Wq'_\kappa [ R_0(2\kappa) \partial^2 q''_\kappa ] - Wq''_\kappa [R_0(2\kappa) \partial^2 q''_\kappa] \big\} \dx \right| \\
	&= (2\pi)^{-\frac{1}{2}} \left| \iint \wh{W}(\xi-\eta) \wh{q'_\kappa}(\eta) \wh{q'_\kappa}(\xi) \frac{ i(\xi - \eta) \xi^2}{ \xi^2 + 4\kappa^2 } \dxi \deta \right| \\
	&\lesssim \iint \left| \wh{W'}(\xi-\eta) \wh{q'_\kappa}(\eta) \wh{q'_\kappa}(\xi) \right| \dxi \deta
	\lesssim \snorm{ \wh{W'} }_{L^1} \snorm{ q'_\kappa }_{L^2}^2 
	\lesssim \snorm{ W' }_{H^1} \snorm{ q'_\kappa }_{L^2}^2 .
	\end{align*}
	In the last inequality, we used Cauchy--Schwarz to estimate
	\begin{equation*}
	\int \big| \wh{ W' } (\xi) \big| \dxi
	\leq \left( \int \frac{\dd \xi}{\xi^2+1} \right)^{\tfrac{1}{2}} \left( \int (\xi^2+1) \big| \wh{ W' } (\xi) \big|^2 \dxi \right)^{\tfrac{1}{2}} .
	\end{equation*}
	Together, we conclude
	\begin{equation*}
	|\eqref{eq:a priori H1 8}| \lesssim \snorm{ q'_\kappa }_{L^2}^2 +1 .
	\end{equation*}
	
	For the tail~\eqref{eq:a priori H1 9} we put $Wq'_\kappa + (Wq_\kappa)'$ and one $q_\kappa$ in $L^2$ using the estimate~\eqref{eq:hskappa identity 1} and the observation $\norm{f}_{H^{-1}_\kappa} \lesssim \kappa^{-1} \norm{f}_{L^2}$, and we put all other terms in operator norm to obtain
	\begin{align*}
	|\eqref{eq:a priori H1 9}|
	&\lesssim \kappa^5\sum_{\substack{ \ell\geq 1,\ m_0,\dots,m_\ell\geq 0 \\ \ell+m_0+\dots+m_\ell\geq 2}} \frac{ \norm{q_\kappa}_{H^1} }{ \kappa^{3} } \left( \frac{\norm{q_\kappa}_{L^\infty}}{\kappa^{2}} \right)^{\ell-1} \left( \frac{\norm{W}_{L^\infty}}{\kappa^{2}} \right)^{m_0+\dots+m_\ell} \lesssim \norm{q_\kappa}_{H^1} 
	\end{align*}
	provided that $\kappa \gg \norm{ q_\kappa }^{1/2}_{L^\infty}$ and $\kappa \gg \norm{W}^{1/2}_{L^\infty}$.  Note the condition $\ell+m_0+\dots+m_\ell \geq 2$ yielded a gain $\lesssim \kappa^{-2}$ when we summed over the parenthetical terms.  Recalling our control~\eqref{eq:a priori H1 7} over the $L^\infty$ norm of $q_\kappa$, we conclude
	\begin{equation*}
	|\eqref{eq:a priori H1 9}|
	\lesssim \norm{ q'_\kappa}_{L^2}^2 + 1
	\quad\tx{uniformly for }\kappa \geq \kappa_0( \norm{q'_\kappa}_{L^2} ) .
	\end{equation*}

	Altogether we have obtained
	\begin{equation*}
	\left| \ddt E_1 \right|
	\lesssim \norm{q'_\kappa}_{L^2}^2 + 1
	\quad\tx{uniformly for }|t|\leq T\tx{ and }\kappa \geq \kappa_0( \norm{q'_\kappa}_{L^2} ) .
	\end{equation*}
	We use $E_1$ and the estimates~\eqref{eq:a priori H1 7} to bound $q'_\kappa$ in $L^2$:
	\begin{equation*}
	\norm{ q'_\kappa}_{L^2}^2
	\lesssim E_1 + \left| \int q_\kappa^3 \dx \right|
	\lesssim E_1 + \norm{ q'_\kappa }_{L^2}^{1/2} .
	\end{equation*}
	Together, we conclude that there exists a constant $C = C(T,A)$ such that
	\begin{equation*}
	\norm{ q'_\kappa(t) }_{L^2}^2 
	\leq C + C\norm{ q'_\kappa }_{L^2}^{1/2} + C\int_0^t \norm{q'_\kappa(s)}_{L^2_x}^2 \ds .
	\end{equation*}
	For $\norm{ q'_\kappa(t) }_{L^2}^2 \gtrsim C^{4/3}$ we can apply Gr\"onwall's inequality to obtain $\norm{ q'_\kappa(t) }_{L^2}^2 \lesssim 1$ for $|t|\leq T$, and so we conclude
	\begin{equation*}
	\norm{ q'_\kappa(t) }_{L^2}^2 \leq C(T, \norm{q(0)}_{H^1})
	\quad\tx{uniformly for }|t|\leq T\tx{ and }\kappa \geq \kappa_0( T, \norm{q(0)}_{H^1} ) . \qedhere
	\end{equation*}
\end{proof}

The last space for which we need to rely upon the corresponding polynomial conserved quantity to obtain an \ti{a priori} estimate is $H^2$.  Starting with $H^3$, the energy arguments are much simplified and the \ti{a priori} estimates are proven inductively.
\begin{prop}
	\label{thm:a priori H2}
	Given $A,T>0$ there exist constants $C$ and $\kappa_0$ such that solutions $q_\kappa(t)$ to the tidal $H_\kappa$ flow~\eqref{eq:thk} obey
	\begin{equation*}
	\norm{q(0)}_{H^2} \leq A
	\quad\implies\quad
	\norm{q_\kappa(t)}_{H^2} \leq C\quad\tx{for all }|t|\leq T\tx{ and }\kappa \geq \kappa_0 .
	\end{equation*}
\end{prop}
\begin{proof}
	By approximation and local well-posedness we may assume that $q(0) \in H^\infty$.  Let
	\begin{equation*}
	E_2(t) := \int \big\{ \tfrac{1}{2} (q''_\kappa(t,x))^2 + 5 q_\kappa(t,x) (q'_\kappa(t,x))^2 + \tfrac{5}{2} q_\kappa(t,x)^4 \big\} \dx
	\end{equation*}
	denote the third energy in the KdV hierarchy of conserved quantities.
	
	We multiply the tidal $H_\kappa$ flow~\eqref{eq:thk} by $q^{(4)}_\kappa - 5(q'_\kappa)^2 - 10 q_\kappa q''_\kappa + 10q_\kappa^3$ and integrate in space to obtain an expression for $\ddt E_2$.  We then integrate by parts to remove the derivative from $g(\kappa,q_\kappa+W) -g(\kappa,W)$, expand both diagonal Green's functions using the relation~\eqref{eq:g alg prop 1}, and apply the identity~\eqref{eq:g alg prop 2} to obtain
	\begin{align*}
	\ddt E_2
	&= \int \big[ q^{(4)}_\kappa - 5(q'_\kappa)^2 - 10 q_\kappa q''_\kappa + 10q_\kappa^3 \big] \big[ 16\kappa^5 g'(\kappa,W) + 4\kappa^2 W' \big] \dx \\
	&\phantom{={}}+ 32\kappa^5 \int \big[ - q'_\kappa q''_\kappa - 2q_\kappa q'''_\kappa + 15 q_\kappa^2 q'_\kappa \big] g(\kappa,W) \dx \\
	&\begin{aligned} \phantom{={}}+ 2\kappa^5 \int \big[ 2\kappa^2( - q'''_\kappa + 6q_\kappa q'_\kappa ) - 2Wq'''_\kappa - W' q''_\kappa + 12Wq_\kappa q'_\kappa + 3W'q_\kappa^2 \big]& \\[-0.5em]
	\times \big[ g(\kappa,q_\kappa+W) - g(\kappa,W) \big]& \dx .
	\end{aligned}
	\end{align*}
	Note that in the case $W\equiv 0$, all three integrals vanish and $E_2$ is conserved as expected.  
	
	In order to exhibit cancellation in the limit $\kappa\to\infty$, we expand $g(\kappa,q_\kappa+W) -g(\kappa,W)$ in powers of $q_\kappa$ and $W$ and regroup terms:
	\begin{align}
	&\ddt E_2
	\nonumber \\
	&= \int \big[ q^{(4)}_\kappa - 5(q'_\kappa)^2 - 10 q_\kappa q''_\kappa \big] \big[ 16\kappa^5 g'(\kappa,W) + 4\kappa^2 W' \big]  
	\label{eq:a priori H2 1}\\ 
	&- 32\kappa^5 \int \big[ q'_\kappa q''_\kappa + 2q_\kappa q'''_\kappa \big] \big[ g(\kappa,W) + \langle \del_x, R_0WR_0\del_x \rangle \big]  
	\label{eq:a priori H2 2}\\
	&+ 64\kappa^7 \int ( - q'''_\kappa + 6q_\kappa q'_\kappa ) \big[ g(\kappa,q_\kappa) - \tfrac{1}{2\kappa} \big]  
	\label{eq:a priori H2 3}\\ 
	&\begin{aligned}
	{}+ 32&\kappa^5 \int \big\{ {-2}\kappa^2 q'''_\kappa \big[ \langle \del_x, R_0W R_0 q_\kappa R_0\del_x \rangle + \langle \del_x, R_0q_\kappa R_0 W R_0 \del_x \rangle \big] \\
	&+ [ 2Wq'''_\kappa + W' q''_\kappa ] \langle \del_x, R_0q_\kappa R_0\del_x \rangle
	+ [ q'_\kappa q''_\kappa + 2q_\kappa q'''_\kappa ] \langle \del_x, R_0WR_0\del_x \rangle \big\} 
	\end{aligned} 
	\label{eq:a priori H2 4}\\ 
	&\begin{aligned}
	{}+ 8 \kappa^2 \int &\big\{ 5 W' q_\kappa^3 - 12\kappa^3 \big[ 4 Wq_\kappa q'_\kappa + W'q_\kappa^2 \big] \langle \del_x, R_0q_\kappa R_0\del_x \rangle \\
	&+ 48\kappa^5 q_\kappa q'_\kappa \big[ \langle \del_x, R_0W R_0 q_\kappa R_0\del_x \rangle + \langle \del_x, R_0q_\kappa R_0 W R_0 \del_x \rangle \big] \big\}  
	\end{aligned} 
	\label{eq:a priori H2 5}\\[0.5em] 
	&\begin{aligned}
	{}+ 64\kappa^7 \sum_{\substack{ \ell \geq 1,\, m_0+ \dots + m_\ell \geq 1 \\ \ell + m_0 + \dots + m_\ell \geq 3 }}  (-1)^{\ell+m_0+\dots+m_\ell} \tr\{ ( - q'''_\kappa + 6q_\kappa q'_\kappa ) R_0 (WR_0)^{m_0} &\\[-1.5em]
	\times q_\kappa R_0 \cdots q_\kappa R_0 (WR_0)^{m_\ell} \}&
	\end{aligned}
	\label{eq:a priori H2 6}\\[0.5em]
	&\begin{aligned}
	{}+ 16\kappa^5 \sum_{\substack{ \ell \geq 1,\, m_0+\dots+m_\ell \geq 0 \\ \ell + m_0 + \dots + m_\ell \geq 2 }}  (-1)^{\ell+m_0+\dots+m_\ell} \tr\{ [ {-4} Wq'''_\kappa - 2W' q''_\kappa + 24 Wq_\kappa q'_\kappa & \\[-1.5em]
	{}+ 6W'q_\kappa^2 ] R_0 (WR_0)^{m_0} q_\kappa R_0 \cdots q_\kappa R_0 (WR_0)^{m_\ell} \} .&
	\end{aligned}
	\label{eq:a priori H2 7}
	\end{align}
	Note that in~\eqref{eq:a priori H2 3} we extracted the terms from~\eqref{eq:a priori H2 6} with no factors of $W$, which is reflected in the condition $m_0+ \dots + m_\ell \geq 1$.  We will estimate each of the terms \eqref{eq:a priori H2 1}--\eqref{eq:a priori H2 7} separately.
	
	For the term~\eqref{eq:a priori H2 1} we expand $g(\kappa,W)$ in powers of $W$:
	\begin{align*}
	|\eqref{eq:a priori H2 1}|
	&\leq \left| \int \big[ q^{(4)}_\kappa - 5(q'_\kappa)^2 - 10 q_\kappa q''_\kappa \big] \big[ - \langle \del_x, R_0W'R_0 \del_x \rangle +  4\kappa^2 W' \big] \right| \\
	&\phantom{\leq{}}+ 16\kappa^5 \sum_{m\geq 2} \big| \tr\big\{ \big[ q^{(4)}_\kappa - 5(q'_\kappa)^2 - 10 q_\kappa q''_\kappa \big] [\partial, R_0 (WR_0)^m ]\big\} \big| .
	\end{align*}
	For the integral, we use the operator identity~\eqref{eq:linear op id} to write
	\begin{align*}
	&\int \big[ q^{(4)}_\kappa - 5(q'_\kappa)^2 - 10 q_\kappa q''_\kappa \big] \big[ - \langle \del_x, R_0W'R_0 \del_x \rangle +  4\kappa^2 W' \big] \\
	&= \int \big[ q^{(4)}_\kappa - 5(q'_\kappa)^2 - 10 q_\kappa q''_\kappa \big] \big[ -W''' - R_0(2\kappa) W^{(5)} \big] .
	\end{align*}
	For the term $q^{(4)}_\kappa$ we integrate by parts twice.  As $W'$ is Schwartz and $\norm{q_\kappa}_{H^1} \lesssim 1$ by \cref{thm:a priori H1}, then Cauchy--Schwarz yields
	\begin{equation*}
	\left| \int \big[ q^{(4)}_\kappa - 5(q'_\kappa)^2 - 10 q_\kappa q''_\kappa \big] \big[ -W''' - R_0(2\kappa) W^{(5)} \big] \right|
	\lesssim \norm{q''_\kappa}_{L^2} + 1
	\lesssim \norm{q''_\kappa}_{L^2}^2 + 1 .
	\end{equation*}
	For the tail, we again integrate by parts twice for $q^{(4)}_\kappa$.  We then estimate the $q_\kappa$ terms and the highest order $W$ term in $L^2$ using the estimate~\eqref{eq:hskappa identity 1} and $\norm{f}_{H^{-1}_{\kappa}}\lesssim \kappa^{-1} \norm{f}_{L^2}$ and the remaining terms in $L^\infty$.  This yields
	\begin{align*}
	&16\kappa^5 \sum_{m\geq 2} \big| \tr\big\{ \big[ q^{(4)}_\kappa - 5(q'_\kappa)^2 - 10 q_\kappa q''_\kappa \big] [\partial, R_0 (WR_0)^m ]\big\} \big| \\
	&\lesssim \kappa^5 \sum_{m\geq 2} \frac{ \norm{q''_\kappa}_{L^2} \norm{W'}_{H^2} }{ \kappa^3 } \left( \frac{ \norm{W}_{W^{1,\infty}} }{ \kappa^2 } \right)^{m-1}
	\lesssim \norm{q''_\kappa}_{L^2} 
	\lesssim \norm{q''_\kappa}_{L^2}^2 + 1
	\end{align*}
	provided that $\kappa \gg \norm{W}_{W^{1,\infty}}^{1/2}$.  
	
	For the term~\eqref{eq:a priori H2 2} we integrate by parts once to write
	\begin{equation*}
	|\eqref{eq:a priori H2 2}|
	\leq 32\kappa^5 \sum_{m\geq 2} \big| \tr \big\{ \big[ \tfrac{1}{2} (q'_\kappa)^2 - 2q_\kappa q''_\kappa \big] [\partial, R_0(WR_0)^m ] \big\} \big| .
	\end{equation*}
	We estimate the $q_\kappa$ terms and the one factor of $W'$ in $L^2$ using the estimate~\eqref{eq:hskappa identity 1} and the observation $\norm{f}_{H^{-1}_{\kappa}}\lesssim \kappa^{-1} \norm{f}_{L^2}$, we estimate the remaining terms in operator norm.  By \cref{thm:a priori H1} we have
	\begin{equation*}
	\norm{q'_\kappa}_{L^\infty}
	\leq \norm{q'_\kappa}_{L^2}^{1/2} \norm{q''_\kappa}_{L^2}^{1/2}
	\lesssim \norm{q''_\kappa}_{L^2}^{1/2} 
	\lesssim \norm{q''_\kappa}_{L^2} + 1.
	\end{equation*}
	Together, we obtain
	\begin{equation*}
	|\eqref{eq:a priori H2 2}|
	\lesssim \kappa^5 \sum_{m\geq 2} \frac{ (\norm{q''_\kappa}_{L^2} +1)\norm{W'}_{L^2} }{ \kappa^3 } \left( \frac{ \norm{W}_{L^{\infty}} }{ \kappa^2 } \right)^{m-1}
	\lesssim \norm{q''_\kappa}_{L^2} 
	\lesssim \norm{q''_\kappa}_{L^2}^2 + 1
	\end{equation*}
	provided that $\kappa \gg \norm{W}_{L^{\infty}}^{1/2}$.  
	
	The term~\eqref{eq:a priori H2 3} vanishes.  Indeed, after integrating by parts and adding a total derivative we have
	\begin{align*}
	\eqref{eq:a priori H2 3}
	&= - 64\kappa^7 \int ( - q''_\kappa + 3q_\kappa^2  ) g'(\kappa,q_\kappa) \dx \\
	&= - 4\kappa^2 \int ( - q''_\kappa + 3q_\kappa^2  ) \big[ 16\kappa^5 g'(\kappa,q_\kappa) + 4\kappa^2 q'_\kappa \big] \dx .
	\end{align*}
	The integral on the RHS is $\ddt E_1$ in the case $W\equiv 0$, which we observed to vanish in \cref{thm:a priori H1}.
	
	For the term~\eqref{eq:a priori H2 4}, we integrate by parts to write
	\begin{align*}
	\eqref{eq:a priori H2 4} = 32\kappa^5 \int \big\{ &2\kappa^2 q''_\kappa \big[ \langle \del_x, R_0W R_0 q'_\kappa R_0\del_x \rangle + \langle \del_x, R_0q'_\kappa R_0 W R_0 \del_x \rangle \big]\\[-0.5em]
	&- 2W q''_\kappa \langle \del_x, R_0q'_\kappa R_0\del_x \rangle - q'_\kappa q''_\kappa \langle \del_x, R_0WR_0\del_x \rangle \\
	&+ 2\kappa^2 q''_\kappa \big[ \langle \del_x, R_0W' R_0 q_\kappa R_0\del_x \rangle + \langle \del_x, R_0q_\kappa R_0 W' R_0 \del_x \rangle \big]\\
	&- W'q''_\kappa \langle \del_x, R_0q_\kappa R_0\del_x \rangle - 2q_\kappa q''_\kappa \langle \del_x, R_0W'R_0\del_x \rangle  \big\} .
	\end{align*}
	We use the operator identities~\eqref{eq:linear op id} and~\eqref{eq:quad op id}.  Observe that the leading order contributions as $\kappa\to\infty$ (i.e. $4\kappa^2f$ in~\eqref{eq:linear op id} and $3fh$ in~\eqref{eq:quad op id}) cancel out.  The remainder is easily estimated, yielding
	\begin{equation*}
	|\eqref{eq:a priori H2 4}| \lesssim \norm{q''_\kappa}_{L^2}^2 + 1 .
	\end{equation*}
	
	For the term~\eqref{eq:a priori H2 5} we write
	\begin{align*}
	\eqref{eq:a priori H2 5} &= 8 \kappa^2 \int \big\{ + 48\kappa^5 q_\kappa q'_\kappa \big[ \langle \del_x, R_0W R_0 q_\kappa R_0\del_x \rangle + \langle \del_x, R_0q_\kappa R_0 W R_0 \del_x \rangle \big] \\
	&\phantom{={}}- 15 W q_\kappa^2 q'_\kappa - 24\kappa^3 Wq_\kappa q'_\kappa \langle \del_x, R_0q_\kappa R_0\del_x \rangle + 12\kappa^3 Wq_\kappa^2 \langle \del_x, R_0q'_\kappa R_0\del_x \rangle \big\} .
	\end{align*}
	We use the operator identities~\eqref{eq:linear op id} and~\eqref{eq:quad op id}.  Observe that the leading order contributions as $\kappa\to\infty$ (i.e. $4\kappa^2f$ in~\eqref{eq:linear op id} and $3fh$ in~\eqref{eq:quad op id})) cancel out.  The remainder is easily estimated, yielding
	\begin{equation*}
	|\eqref{eq:a priori H2 5}| \lesssim \norm{q''_\kappa}_{L^2}^2 + 1 .
	\end{equation*}
	
	For the tail~\eqref{eq:a priori H2 6}, we integrate by parts once to obtain
	\begin{align*}
	&|\eqref{eq:a priori H2 6}| \\
	&\lesssim \kappa^7 \sum_{\substack{ \ell \geq 1,\ m_0+ \dots + m_\ell \geq 1 \\ \ell + m_0 + \dots + m_\ell \geq 3 }} \big| \tr\big\{ ( - q''_\kappa + 3q_\kappa^2 ) [ \partial, R_0 (WR_0)^{m_0} q_\kappa R_0 \cdots q_\kappa R_0 (WR_0)^{m_\ell} ] \big\} \big| . \\
	\intertext{We put $- q''_\kappa + 3q_\kappa^2$ and the highest order $q_\kappa$ in $L^2$ using the identity~\eqref{eq:hskappa identity 1} and the observation $\norm{f}_{H^{-1}_\kappa} \lesssim \kappa^{-1} \norm{f}_{L^2}$, and we estimate the remaining terms in operator norm:}
	&\lesssim \kappa^7 \sum_{\substack{ \ell \geq 1,\ m_0+ \dots + m_\ell \geq 1 \\ \ell + m_0 + \dots + m_\ell \geq 3 }} \frac{ \norm{q''_\kappa}_{L^2} +1 }{\kappa^3} \left( \frac{\norm{q_\kappa}_{H^1}}{\kappa^{2}} \right)^{\ell-1} \left( \frac{\norm{W}_{W^{1,\infty}}}{\kappa^2} \right)^{m_0+\dots+m_\ell} .
	\end{align*}
	We re-index $m = m_0+\dots+m_\ell$ and sum over $\ell + m$ as in~\eqref{eq:double sum inner}.  The condition $\ell + m_0 + \dots + m_\ell \geq 3$ guarantees a gain $\lesssim (\kappa^{-2})^2$ when we sum over the two parenthetical terms, and so we obtain an acceptable bound.
	
	For the tail~\eqref{eq:a priori H2 7}, we estimate
	\begin{align*}
	|\eqref{eq:a priori H2 7}| \lesssim
	\kappa^5 \sum_{\substack{ \ell \geq 1,\ m_0,\dots,m_\ell \geq 0 \\ \ell + m_0 + \dots + m_\ell \geq 2 }} \big| \tr\big\{ \big[ {-4} Wq'''_\kappa - 2W' q''_\kappa + 24 Wq_\kappa q'_\kappa + 6W'q_\kappa^2 \big]& \\[-1.5em]
	\times R_0 (WR_0)^{m_0} q_\kappa R_0 \cdots q_\kappa R_0 (WR_0)^{m_\ell} &\big\} \big| .
	\end{align*}
	For the term $q'''_\kappa$ we integrate by parts once.  We then put the square-bracketed term and the highest order factor of $q_\kappa$ in $L^2$ using the identity~\eqref{eq:hskappa identity 1} and the observation $\norm{f}_{H^{-1}_\kappa} \lesssim \kappa^{-1} \norm{f}_{L^2}$, and we estimate the remaining terms in operator norm:
	\begin{equation*}
	|\eqref{eq:a priori H2 7}|\lesssim \kappa^5 \sum_{\substack{ \ell \geq 1,\ m_0,\dots,m_\ell \geq 0 \\ \ell + m_0 + \dots + m_\ell \geq 2 }} \frac{ \norm{q''_\kappa}_{L^2} +1 }{\kappa^{3/2}} \left( \frac{\norm{q_\kappa}_{H^1}}{\kappa^{3/2}} \right)^{\ell} \left( \frac{\norm{W}_{W^{1,\infty}}}{\kappa^2} \right)^{m_0+\dots+m_\ell} .
	\end{equation*}
	We re-index $m = m_0+\dots+m_\ell$ and sum over $\ell + m$ as in~\eqref{eq:double sum inner}.  The condition $\ell + m_0 + \dots + m_\ell \geq 2$ guarantees a gain $\lesssim \kappa^{-3/2} \cdot \kappa^{-2}$ when we sum over the two parenthetical terms, and so we conclude
	\begin{equation*}
	|\eqref{eq:a priori H2 6}| 
	\lesssim \norm{q''_\kappa}_{L^2} + 1
	\lesssim \norm{q''_\kappa}_{L^2}^2 + 1
	\end{equation*}
	provided that $\kappa$ is sufficiently large (independently of $\norm{q''_\kappa}_{L^2}$).
	
	Altogether, we have obtained
	\begin{equation*}
	\left| \ddt E_2 \right| \lesssim \norm{q''_\kappa}_{L^2}^2 + 1 
	\quad\tx{uniformly for }|t|\leq T\tx{ and }\kappa\geq \kappa_0 ,
	\end{equation*}
	where $\kappa_0$ depends only on $T$ and $\norm{q(0)}_{H^1}$.  Using \cref{thm:a priori H1}, we can then bound
	\begin{equation*}
	\norm{q''_\kappa}_{L^2}^2
	\lesssim E_2 + \left| \int q_\kappa (q'_\kappa)^2 \dx \right| + \left| \int q_\kappa^4 \dx \right|
	\lesssim E_2 + 1.
	\end{equation*}
	Together, we conclude that there exists a constant $C = C(T,A)$ such that
	\begin{equation*}
	\norm{q''_\kappa(t)}_{L^2}^2
	\leq C + C \int_0^t \norm{q''_\kappa(s)}_{L^2_x}^2 \ds
	\end{equation*}
	uniformly for $|t|\leq T$ and $\kappa\geq \kappa_0$.  Gr\"onwall's inequality then yields
	\begin{equation*}
	\norm{ q''_\kappa(t) }_{L^2}^2 \leq C(T, \norm{q(0)}_{H^2})
	\quad\tx{uniformly for }|t|\leq T\tx{ and }\kappa \geq \kappa_0( T, \norm{q(0)}_{H^1} ) , \qedhere
	\end{equation*}
	as desired.
\end{proof}

For $H^s$, $s\geq 3$ we proceed by induction:
\begin{prop}
	\label{thm:a priori Hs}
	Given an integer $s\geq 3$ and $A,T>0$ there exist constants $C$ and $\kappa_0$ such that solutions $q_\kappa(t)$ to the tidal $H_\kappa$ flow~\eqref{eq:thk} obey
	\begin{equation*}
		\norm{q(0)}_{H^s} \leq A
		\quad\implies\quad
		\norm{q_\kappa(t)}_{H^s} \leq C\quad\tx{for all }|t|\leq T\tx{ and }\kappa \geq \kappa_0 .
	\end{equation*}
\end{prop}
\begin{proof}
	We induct on $s$, with the base case given by \cref{thm:a priori H2}.  Assume the result holds for $s-1$.  
	
	By approximation and local well-posedness we may assume that $q(0) \in H^\infty$.  We define
	\begin{equation*}
	E_s(t) := \tfrac{1}{2} \int (q^{(s)}_\kappa(t,x))^2 \dx .
	\end{equation*}
	Expanding $g(\kappa,q_\kappa+W)$ in powers of $q_\kappa$ and $W$, we write
	\begin{align}
	&\ddt E_s
	\nonumber \\
	&= \int q_\kappa^{(s)} \big\{ {-16}\kappa^5 \langle\del_x, R_0q^{(s+1)}_\kappa R_0 \del_x \rangle + 4\kappa^2 q^{(s+1)}_\kappa \big\} \dx
	\label{eq:a priori Hs 1} \\
	&\phantom{={}}+ \int q_\kappa^{(s)} \big\{ {-16}\kappa^5 \langle\del_x, R_0W^{(s+1)}R_0 \del_x \rangle + 4\kappa^2 W^{(s+1)} \big\} \dx 
	\label{eq:a priori Hs 2} \\
	&\phantom{={}}+ 16\kappa^5 \int q_\kappa^{(s)} \langle\del_x, [\partial^{s+1}, R_0q_\kappa R_0q_\kappa R_0] \del_x \rangle \dx
	\label{eq:a priori Hs 3} \\
	&\begin{aligned}
	\phantom{={}}+16\kappa^5 \int q_\kappa^{(s)} \big\{ &\langle\del_x, [\partial^{s+1}, R_0WR_0q_\kappa R_0] \del_x \rangle \\[-0.5em]
	&+ \langle\del_x, [\partial^{s+1}, R_0q_\kappa R_0WR_0] \del_x \rangle \big\} \dx
	\end{aligned}
	\label{eq:a priori Hs 4} \\
	&\phantom{={}}+ 16\kappa^5 \int q_\kappa^{(s)} \langle\del_x, [\partial^{s+1}, R_0W R_0W R_0] \del_x \rangle \dx
	\label{eq:a priori Hs 6} \\
	&\begin{aligned}
	\phantom{={}}+ 16\kappa^5 \int q_\kappa^{(s)} \big\{ &g(\kappa,q_\kappa+W) + \langle \del_x , R_0(q_\kappa+W)R_0 \del_x \rangle \\[-0.5em]
	&- \langle \del_x , R_0(q_\kappa+W)R_0(q_\kappa+W)R_0 \del_x \rangle \big\}^{(s+1)} \dx .
	\end{aligned}
	\label{eq:a priori Hs 5}
	\end{align}
	We will estimate the terms \eqref{eq:a priori Hs 1}--\eqref{eq:a priori Hs 5} separately.
	
	The first linear contribution~\eqref{eq:a priori Hs 1} vanishes.  To see this, we use the first operator identity of~\eqref{eq:linear op id} to write
	\begin{equation*}
	\eqref{eq:a priori Hs 1}
	= \int q_\kappa^{(s)} \big\{ {-16}\kappa^4 R_0(2\kappa) q^{(s+1)}_\kappa + 4\kappa^2 q^{(s+1)}_\kappa \big\} \dx
	= 0 .
	\end{equation*}
	In the last equality we noted that the integrand is odd in Fourier variables, or equivalently that the integrand of~\eqref{eq:a priori Hs 1} is a total derivative.
	
	Now we estimate the linear contribution~\eqref{eq:a priori Hs 2} from $W$.  Using the operator identity~\eqref{eq:linear op id} and recalling that $W'$ is Schwartz, we estimate
	\begin{align*}
	|\eqref{eq:a priori Hs 2}|
	&= \left| \int q_\kappa^{(s)} \{ - W^{(s+3)} - [ R_0(2\kappa) W^{(s+5)} ] \} \dx \right| \\
	&\lesssim \snorm{ q_\kappa^{(s)} }_{L^2} \big( \snorm{W^{(s+3)}}_{L^2} + \kappa^{-2} \snorm{W^{(s+5)}}_{L^2} \big)
	\lesssim E_s^{1/2} \lesssim E_s + 1 .
	\end{align*}
	
	In the first quadratic term~\eqref{eq:a priori Hs 3} we distribute the derivatives $[\partial^{s+1},\cdot]$.  For the terms with $q_\kappa^{(s+1)}$, we ``integrate by parts'' to write
	\begin{align*}
	&16\kappa^5 \big( \tr\big\{ q_\kappa^{(s)} R_0q_\kappa^{(s+1)} R_0 q_\kappa R_0 \big\} + \tr\big\{ q_\kappa^{(s)} R_0q_\kappa R_0 q_\kappa^{(s+1)} R_0 \big\}  \big) \\
	&= 16\kappa^5 \tr\big\{ \big[\partial,q_\kappa^{(s)} R_0 q_\kappa^{(s)} R_0 \big] q_\kappa R_0 \big\}
	= - 16\kappa^5 \tr\big\{ q_\kappa^{(s)} R_0q_\kappa^{(s)} R_0[\partial,q_\kappa ]R_0 \big\} .
	\end{align*}
	This leaves
	\begin{equation*}
	|\eqref{eq:a priori Hs 3}|
	\lesssim \kappa^5 \sum_{j=1}^s \big|\tr\big\{ q_\kappa^{(s)}R_0 q_\kappa^{(j)}R_0 q_{\kappa}^{(s+1-j)} R_0 \big\}\big| + \big|\tr\big\{ q_\kappa^{(s)}R_0 q_\kappa^{(s-1)}R_0 q_{\kappa}^{(s-1)} R_0 \big\}\big| .
	\end{equation*}
	The last term only appears in the case $s=3$, but we can see that it vanishes by writing $q_\kappa^{(s)} = [ \partial , q_\kappa^{(s-1)} ]$ and cycling the trace.  Note that all copies of $q_\kappa$ now have at most $s$ derivatives.  We put the two highest order factors of $q_\kappa$ in $L^2$ using the identity~\eqref{eq:hskappa identity 1} and the observation $\norm{f}_{H^{-1}_\kappa} \lesssim \kappa^{-1} \norm{f}_{L^2}$.  As $s\geq 3$, the third factor $q^{(j)}_\kappa$ has order $j\leq s-2$ and may be estimated in operator norm because $\snorm{ q_\kappa^{(j)} }_{L^\infty} \leq \norm{q_\kappa}_{H^{s-1}} \lesssim 1$ by inductive hypothesis.  This yields
	\begin{equation*}
	|\eqref{eq:a priori Hs 3}|
	\lesssim \snorm{ q^{(s)}_\kappa }_{L^2}^2 + \snorm{ q^{(s)}_\kappa }_{L^2}
	\lesssim E_s + 1 .
	\end{equation*}
	
	The second quadratic contribution~\eqref{eq:a priori Hs 4} is similar.  For the terms with $q_\kappa^{(s+1)}$ we ``integrate by parts'' to write
	\begin{align*}
	&16\kappa^5 \big( \tr\big\{ q_\kappa^{(s)} R_0q_\kappa^{(s+1)} R_0 W R_0 \big\} + \tr\big\{ q_\kappa^{(s)} R_0 W R_0 q_\kappa^{(s+1)} R_0 \big\}  \big) \\
	&= 16\kappa^5 \tr\big\{ \big[\partial,q_\kappa^{(s)} R_0 q_\kappa^{(s)} R_0 \big] W R_0 \big\}
	= - 16\kappa^5 \tr\big\{ q_\kappa^{(s)} R_0q_\kappa^{(s)} R_0[\partial,W]R_0 \big\} .
	\end{align*}
	In all cases we put the two factors of $q_\kappa$ in $L^2$ using the identity~\eqref{eq:hskappa identity 1} and the observation $\norm{f}_{H^{-1}_\kappa} \lesssim \kappa^{-1} \norm{f}_{L^2}$, and the remaining factors in operator norm.  This yields
	\begin{equation*}
	|\eqref{eq:a priori Hs 4}|
	\lesssim \snorm{ q^{(s)}_\kappa }_{L^2}^2 + \snorm{ q^{(s)}_\kappa }_{L^2}
	\lesssim E_s + 1 .
	\end{equation*}
	
	The quadratic $W$ contribution~\eqref{eq:a priori Hs 6} is easily estimated.  We put $q_\kappa^{(s)}$ and the higher order $W$ term in $L^2$ using the identity~\eqref{eq:hskappa identity 1} and the observation $\norm{f}_{H^{-1}_\kappa} \lesssim \kappa^{-1} \norm{f}_{L^2}$, and we put the remaining factor of $W$ in $L^\infty$.  This yields
	\begin{equation*}
	|\eqref{eq:a priori Hs 6}|
	\lesssim \snorm{ q^{(s)}_\kappa }_{L^2}
	\lesssim E_s + 1 .
	\end{equation*}
	
	Next we turn to the series tail~\eqref{eq:a priori Hs 5}.  Applying the tail convergence~\eqref{eq:tail conv} to $q = q_\kappa$, we know there exists a constant $\kappa_0(E_s(t))$ so that 
	\begin{align*}
	16\kappa^5 \big\lVert \big\{ &g(\kappa,q_\kappa+W) + \langle \del_x , R_0(q_\kappa+W)R_0 \del_x \rangle \\
	&- \langle \del_x , R_0(q_\kappa+W)R_0(q_\kappa+W)R_0 \del_x \rangle \big\}^{(s+1)} \big\rVert_{L^2} \leq 1 
	\end{align*}
	uniformly for $\kappa\geq \kappa_0(E_s(t))$.  Therefore, by Cauchy--Schwarz we have
	\begin{equation*}
	| \eqref{eq:a priori Hs 5} |
	\leq (2E_s)^{1/2}
	\lesssim E_s + 1
	\quad\tx{uniformly for }\kappa \geq \kappa_0(E_s(t)) .
	\end{equation*}
	
	Altogether, we have shown that there exists a constant $C = C(T,A)$ such that
	\begin{equation*}
	\left| \ddt E_s \right| \leq C( E_s + 1 )
	\quad\tx{uniformly for }|t|\leq T\tx{ and }\kappa \geq \kappa_0(E_s(t)) .
	\end{equation*}
	Gr\"onwall's inequality then yields the bound
	\begin{equation*}
	E_s(t) \leq (E_s(0) + 1) e^{CT} - 1
	\quad\tx{uniformly for }|t|\leq T,\ \kappa \geq \kappa_0\big( (E_s(0) + 1) e^{CT} - 1 \big) ,
	\end{equation*}
	which concludes the inductive step.
\end{proof}

As a consequence, we are able to upgrade local well-posedness to global well-posedness:
\begin{cor}
	\label{thm:thk gwp}
	Given an integer $s\geq 0$ and $A,T>0$, there exists a constant $\kappa_0$ so that for $\kappa \geq \kappa_0$ the tidal $H_\kappa$ flows~\eqref{eq:thk} with initial data in the closed ball $B_A \subset H^s(\R)$ of radius $A$ are globally well-posed.
\end{cor}
\begin{proof}
	Fix $A, T>0$, let $C$ be the constant guaranteed by \cref{thm:a priori L2,thm:a priori H1,thm:a priori H2,thm:a priori Hs}, and consider the closed ball $B_C \subset H^s$ of radius $C$.  By local well-posedness (cf.~\cref{thm:thk lwp}) we know there exists $\del>0$ such that the integral equation is a contraction on $C_tB_C([-\del,\del]\times\R)$, and hence there exists a unique fixed point $q_\kappa$.  However, by \cref{thm:a priori L2,thm:a priori H1,thm:a priori H2,thm:a priori Hs} we know that $q_\kappa(t)$ is in $B_C$ as long as $|t|\leq T$.  Therefore, we may iterate the contraction argument to construct a unique solution in $C_tH^s([-T,T]\times\R)$ that depends continuously upon the initial data.
\end{proof}

\section{Convergence at low regularity}
\label{sec:low reg}

Ultimately, we want to show that for initial data in $H^s$ with $s\geq 3$ the solutions $q_\kappa(t)$ to the tidal $H_\kappa$ flows converge in $H^s$.  Although the linear and quadratic terms of the tidal $H_\kappa$ flow formally converge to tidal KdV as $\kappa\to\infty$, the first term in the error contains $q^{(5)}_\kappa$ (cf.~\eqref{eq:linear op id}).  Consequently, we will first demonstrate convergence in $H^{-2}$ so that we may absorb these five extra derivatives:
\begin{prop}
	\label{thm:conv low reg}
	Given $T>0$ and a bounded set $Q\subset H^3$ of initial data, the corresponding solutions $q_\kappa(t)$ to the tidal $H_\kappa$ flows~\eqref{eq:thk} are Cauchy in $C_tH^{-2}([-T,T]$ $\times\R)$ as $\kappa\to\infty$ uniformly for $q(0)\in Q$.
\end{prop}
\begin{proof}
	In the following all spacetime norms will be taken over the slab $[-T, T] \times \R$.  Let $\kappa_0$ denote the constant from \cref{thm:thk gwp} for $s=3$, so that for $\kappa\geq \kappa_0$ the solutions $q_\kappa(t)$ to the $H_\kappa$ flows exist in $C_tH^3$.
	
	Consider the difference $q_\kappa - q_\varkappa$ of two of these solutions with $\varkappa\geq \kappa\geq \kappa_0$.  Recall that the tidal $H_\kappa$ and tidal $H_\varkappa$ flows commute (cf.~\eqref{eq:hk flow 2}).  Letting $\thk$ denote the tidal $H_\kappa$ flow Hamiltonian, this allows us to write
	\begin{equation*}
	q_\varkappa(t)
	= e^{tJ\nabla \thvk} q(0)
	= e^{tJ\nabla( \thvk - \thk )} e^{tJ\nabla\thk} q(0) .
	\end{equation*}
	Consequently, we estimate
	\begin{equation*}
	\norm{ q_{\kappa}-q_{\varkappa} }_{C_{t} H^{-1}} \leq \sup _{q \in Q^*_T(\kappa)}\, \sup _{\varkappa \geq \kappa}\ \snorm{ e^{t J \nabla(\thvk - \thk)} q-q }_{C_{t} H^{-1}} ,
	\end{equation*}
	for the set
	\begin{equation*}
	Q^*_T(\kappa) := \{ e^{tJ\nabla\thk} q(0) : |t|\leq T,\ q(0)\in Q \}
	\end{equation*}
	of tidal $H_\kappa$ flows.  By the fundamental theorem of calculus, it suffices to show that under the difference flow $\thvk - \thk$ we have
	\begin{equation*}
	\sup_{q\in Q^*_T(\kappa)}\, \sup_{\varkappa\geq \kappa}\ \norm{ \frac{dq}{dt} }_{C_tH^{-2}} \to 0 \quad\tx{as }\kappa\to\infty .
	\end{equation*}
	Note that $Q^*_T(\kappa)$ is a bounded subset of $H^3$ by the \ti{a priori} estimate of \cref{thm:a priori Hs}.
	
	Given initial data $q(0) \in Q^*_T(\kappa)$, let $q(t)$ denote the corresponding solution to the difference flow $\thvk - \thk$.  Then $q$ solves
	\begin{equation*}
	\ddt q = 16\varkappa^5 g'(\kappa,q+W) + 4\varkappa^2(q+W)' - 16\kappa^5 g'(\varkappa,q+W) - 4\kappa^2(q+W)' .
	\end{equation*}
	To exhibit cancellation in the limit $\varkappa, \kappa \to \infty$, we expand $g'(\kappa,q+W)$ into a series in $q$ and $W$ and extract the linear and quadratic terms:
	\begin{align}
	&\ddt q 
	\nonumber \\
	&\begin{aligned}
	= \big\{ &{-16}\varkappa^5 \langle \del_x, R_0(\varkappa)(q+W)R_0(\varkappa) \del_x \rangle + 4\varkappa^2(q+W) \\
	&+16\kappa^5 \langle \del_x, R_0(\kappa)(q+W)R_0(\kappa) \del_x \rangle - 4\kappa^2(q+W) \big\}'
	\end{aligned}
	\label{eq:diff flow 1}\\
	&\begin{aligned}
	\phantom{={}} + \big\{ &16\varkappa^5 \langle \del_x, R_0(\varkappa)(q+W)R_0(\varkappa)(q+W)R_0(\varkappa) \del_x \rangle \\
	&- 16\kappa^5 \langle \del_x, R_0(\kappa)(q+W)R_0(\kappa)(q+W)R_0(\kappa) \del_x \rangle \big\}' 
	\end{aligned}
	\label{eq:diff flow 2}\\
	&\phantom{={}} + \sum (\tx{terms with 3 or more }q\tx{ or }W) .
	\label{eq:diff flow 3}
	\end{align}
	We will show that each of the terms \eqref{eq:diff flow 1}--\eqref{eq:diff flow 3} converge to zero.
	
	For the linear term~\eqref{eq:diff flow 1}, we use the operator identity~\eqref{eq:linear op id} to estimate
	\begin{align*}
	&\norm{ \eqref{eq:diff flow 1} }_{H^{-2}}
	= \snorm{ [ {-R_0}(2\varkappa) + R_0(2\kappa) ] (q+W)^{(5)} }_{H^{-2}} \\
	&\lesssim (\varkappa^{-2} + \kappa^{-2}) \big( \snorm{ q^{(5)} }_{H^{-2}} + \snorm{ W^{(5)} }_{H^{-2}} \big)
	\lesssim \kappa^{-2} \big( \norm{ q }_{H^{3}} + \norm{ W''' }_{L^2} \big)
	\end{align*}
	uniformly for $\varkappa\geq \kappa$.  As $q\in Q^*_T(\kappa)$ is bounded in $H^3$, we conclude that
	\begin{equation*}
	\sup_{q\in Q^*_T(\kappa)}\, \sup_{\varkappa\geq \kappa}\ \norm{ \eqref{eq:diff flow 1} }_{C_tH^{-2}} \to 0 \quad\tx{as }\kappa\to\infty .
	\end{equation*}
	
	For the quadratic term~\eqref{eq:diff flow 2}, we add and subtract the corresponding tidal KdV term $6(q+W)(q+W)'$ and estimate
	\begin{align*}
	\norm{ \eqref{eq:diff flow 2} }_{H^{-2}} 
	&\lesssim \norm{ 16\varkappa^5 \langle \del_x, R_0(\varkappa)(q+W)R_0(\varkappa)(q+W)R_0(\varkappa) \del_x \rangle - 3(q+W)^2 }_{H^{-1}} \\
	&\phantom{\lesssim{}} + \norm{ 16\kappa^5 \langle \del_x, R_0(\kappa)(q+W)R_0(\kappa)(q+W)R_0(\kappa) \del_x \rangle - 3(q+W)^2 }_{H^{-1}} .
	\end{align*}
	Using the operator identity~\eqref{eq:quad op id} and the estimates $\snorm{ R_{0}(2 \kappa) \partial^j }_\op \lesssim \kappa^{j-2}$ for $j=0,1,2$ (the estimate for $j = 0$ is also true as an operator on $L^\infty$ by the explicit kernel formula for $R_0$ and Young's inequality), one can easily prove by duality that
	\begin{equation*}
	\norm{ 16\kappa^5 \langle \del_x, R_0(\kappa)fR_0(\kappa)gR_0(\kappa) \del_x \rangle - 3fg }_{L^2}
	\lesssim \kappa^{-2} \norm{ f }_{W^{2,\infty}} \norm{ g }_{H^2} .
	\end{equation*}
	Moreover, the roles of $f$ and $g$ can be exchanged since the identity~\eqref{eq:quad op id} is symmetric in $f$ and $g$.  Therefore, expanding the products $(q+W)(q+W)$ we have
	\begin{equation*}
	\norm{ \eqref{eq:diff flow 2} }_{H^{-2}}
	\lesssim (\varkappa^2 + \kappa^2) \big( \norm{q}^2_{H^3} + \norm{W}_{W^{3,\infty}} \norm{q}_{H^3} + \norm{W}_{W^{2,\infty}} \norm{W'}_{H^2} \big) .
	\end{equation*}
	As $q\in Q^*_T(\kappa)$ is bounded in $H^3$, we conclude that
	\begin{equation*}
	\sup_{q\in Q^*_T(\kappa)}\, \sup_{\varkappa\geq \kappa}\ \norm{ \eqref{eq:diff flow 2} }_{C_tH^{-2}} \to 0 \quad\tx{as }\kappa\to\infty .
	\end{equation*}
	
	It only remains to show that the tails~\eqref{eq:diff flow 3} converge to zero in $C_tH^{-2}$.  In fact, by~\eqref{eq:tail conv} we have convergence in the stronger $C_tL^2$ norm:
	\begin{equation*}
	\sup_{q\in Q^*_T(\kappa)}\, \sup_{\varkappa\geq \kappa}\ \norm{ \eqref{eq:diff flow 3} }_{C_tL^2} \to 0 \quad\tx{as }\kappa\to\infty . \qedhere
	\end{equation*}
\end{proof}

\section{Equicontinuity}
\label{sec:equicty}

We want to upgrade the $H^{-2}$ convergence of the previous section to $H^s$, $s\geq 3$.  This will be accomplished via the estimate
\begin{equation*}
\norm{ q_\varkappa - q_\kappa }_{H^s}^2
\lesssim (N+1)^{s+2} 	\norm{ q_\varkappa - q_\kappa }_{H^{-2}}^2 + \norm{ q_\varkappa - q_\kappa }_{H^s(|\xi|\geq N)}^2 .
\end{equation*}
In this section, we will show that we can pick $N$ sufficiently large so that the second term on the RHS is arbitrarily small uniformly for $\kappa,\varkappa$ large.  It then follows from \cref{thm:conv low reg} that the first term on the RHS converges to zero as $\kappa,\varkappa \to \infty$.

Uniform control over Fourier tails is called equicontinuity.  Specifically, a set $Q\subset H^s$ is \ti{equicontinuous} in $H^s$ if
\begin{equation*}
\int_{|\xi|\geq N} (\xi^2 + 1)^s | \hat{q} (\xi) |^2 \dxi
\to 0 \quad\tx{as }N\to\infty,\tx{ uniformly for }q\in Q. 
\end{equation*}
This is equivalent to the notion of equicontinuity in the $L^p$ precompactness theorem (cf. \cite[Lem.~4.2]{Killip2019}).  In particular, precompact subsets of $H^s$ are equicontinuous in $H^s$.

It would suffice to show that the tidal $H_\kappa$ flows $\{ q_\kappa(t) : \kappa \geq \kappa_0 \}$ on bounded time intervals are equicontinuous.  With the presence of the background wave $W$ in tidal KdV we expect the quantity $\norm{ q_\kappa(t) }_{H^s}$ to grow, and so we must estimate this growth.  Expanding the diagonal Green's function in powers of $q_\kappa$ and $W$, we are able to control the linear and quadratic terms as we would for tidal KdV; however, it remains to control the higher order contributions which vanish in the limit $\kappa\to\infty$.  Consequently, instead of honest equicontinuity for the tidal $H_\kappa$ flows $q_\kappa(t)$, we will require $\kappa \geq N$ in \cref{thm:equicty} so that $\bigo(\kappa^{-1})$ contributions as $\kappa\to\infty$ are also $\bigo(N^{-1})$ as $N\to\infty$.

In order to control the Fourier tail growth we will use a smooth Littlewood-Paley decomposition.  We define Littlewood--Paley pieces via the following $L^2$-based partition of unity.  Fix a $C^{\infty}$ function $\phi: \R \to [0,1]$ that satisfies
\begin{equation*}
\phi(\xi) = \begin{cases}
1 & |\xi|\leq 1 ,\\
0 & |\xi|\geq 2 .
\end{cases}
\end{equation*}
Then the function
\begin{equation*}
\psi(\xi) := \sqrt{ \phi(\xi) - \phi(2\xi) }
\quad\tx{satisfies}\quad
\sum_{N \in 2^{\Z}} \psi^2(\tfrac{\xi}{N}) = 1 \quad\tx{for all }\xi\neq 0.
\end{equation*}
Sums over capitalized indices will always be over the set $2^{\Z} := \{ 2^n : n\in\Z \}$.  For Schwartz functions $f$ we define
\begin{equation*}
\wh{P_Nf}(\xi)= \psi(\tfrac{\xi}{N})\hat{f}(\xi) ,\qquad
\wh{P^2_{\geq N}f}(\xi) = \sum_{K\geq N} \psi^2(\tfrac{\xi}{K})\hat{f}(\xi) , \qquad
P^2_{< N} = 1 - P^2_{\geq N} .
\end{equation*}
Our choice of partition of unity ensures that the square sum $\sum P_N^2 f$ converges to $f$ in $L^p$ for $p\in (1,\infty)$.  We choose a square-sum decomposition because we will ultimately measure $\snorm{ P_{\geq N} q_\kappa^{(s)} }_{L^2}^2$, which we may write as the $L^2$-pairing of $P_{\geq N}^2 q_\kappa^{(s)}$ and $q_\kappa^{(s)}$.

We remark that directly estimating the growth of $\snorm{ P_{\geq N} q^{(s)} }_{L^2}^2$ would fail due to the quadratic term of tidal KdV.  Indeed, if we compute $\ddt \snorm{ P_{\geq N} q^{(s)} }_{L^2}^2$ under the tidal KdV flow, we obtain a term of the form
\begin{equation*}
\int \big( P_{\geq N}^2 q^{(s)} \big) \left( 3 q^2 \right)^{(s+1)} \dx .
\end{equation*}
Decomposing each factor of $q = P_{\geq N}^2 q + P_{< N}^2 q$, the terms with at least one copy of $P_{\geq N}^2 q$ can be estimated by two factors of $\snorm{ P_{\geq N} q^{(s)} }_{L^2}$.  However, the high-low-low term
\begin{equation*}
\int \big( P_{\geq N}^2 q^{(s)} \big) \left[ 3 \left( P_{< N}^2 q \right) \left( P_{< N}^2 q \right) \right]^{(s+1)} \dx
\end{equation*}
only contributes one factor of $\snorm{ P_{\geq N} q^{(s)} }_{L^2}$, which does not guarantee that initially small Fourier tails remain small.

To overcome this, we introduce a more gradual high-frequency cutoff.  Given an integer $s\geq 3$ and a Schwartz function $f$, we define the Fourier multiplier 
\begin{equation}
	\wh{ \pihi f } (\xi) = \mhi(\tfrac{\xi}{N}) \wh{f}(\xi) , \qquad
	\mhi(\xi) = \sum_{K<1} K^s \psi^2(\tfrac{\xi}{K}) + \sum_{K\geq 1} \psi^2(\tfrac{\xi}{K}) .
	\label{eq:proj op}
\end{equation}
The power of $s$ in the definition~\eqref{eq:proj op} will provide us with the replacement~\eqref{eq:bernstein pihi} for the Bernstein inequality satisfied by $P^2_{\geq N}$.  We also define
\begin{equation*}
\wh{ \pilo f } (\xi) = \sqrt{ 1 - \mhi^2(\tfrac{\xi}{N}) } \wh{f}(\xi)
\qquad\tx{so that}\qquad
\pilo^2 + \pihi^2 = 1 .
\end{equation*}

For the Littlewood--Paley operators we have the familiar Bernstein inequalities
\begin{equation}
\begin{aligned}
&\snorm{ P_N f^{(j)} }_{L^p} \sim N^j \snorm{ P_N f }_{L^p} \quad\tx{for }p\in(1,\infty),\ j\in\Z, \\ 
&\snorm{ P_N f^{(j)} }_{L^\infty} \lesssim N^j \snorm{ P_N f }_{L^\infty} \quad\tx{for }j>0 .
\end{aligned}
\label{eq:bernstein 1}
\end{equation}
Summing over $N\in 2^\N$, we obtain the high and low frequency projection estimates
\begin{align}
&\snorm{ P^2_{<N} f^{(j)} }_{L^p} \lesssim N^j \snorm{ P^2_{<N} f }_{L^p} \quad\tx{for }p\in[1,\infty],\ j>0 ,
\label{eq:bernstein 3}\\
&\snorm{ P^2_{\geq N} f }_{L^p} \lesssim N^{-j} \snorm{ P^2_{\geq N} f^{(j)} }_{L^p} \quad\tx{for }p\in(1,\infty),\ j>0 .
\label{eq:bernstein 2}
\end{align}

We will now obtain analogous Bernstein inequalities for our projection operators $\pihi$ and $\pilo$:
\begin{lem}
	Fix an integer $s\geq 3$.  Then the operators $\pihi$ defined in~\eqref{eq:proj op} are bounded on $L^p$ for $p\in[1,\infty]$ uniformly in $N$, and we have the estimates
	\begin{align}
	&\snorm{ \pilo^2 q^{(s+j)} }_{L^p} \lesssim N^j \snorm{ P^2_{<2N} q^{(s)} }_{L^p} \quad\tx{for }p\in[1,\infty],\ j>0 ,
	\label{eq:bernstein pilo}\\
	&\snorm{ \pihi^2 q^{(s-j)} }_{L^p} \lesssim N^{-j} \snorm{ \pihi q^{(s)} }_{L^p} \quad\tx{for }p\in(1,\infty),\ 0<j\leq s.
	\label{eq:bernstein pihi}
	\end{align}
\end{lem}
\begin{proof}
	Boundedness on $L^p$ follows from Young's inequality.  Indeed, if we let
	\begin{equation*}
	\mlo(\tfrac{\xi}{N}) = \sqrt{ 1 - \mhi^2(\tfrac{\xi}{N}) }
	\end{equation*}
	denote the Fourier symbol of $\pilo$, then we have $\mlo \in C_c^\infty$ and
	\begin{equation*}
	\norm{ \pilo f }_{L^{p}}
	= \norm{ N^{d} \mlo^\vee (N \cdot) * f }_{L^{p}} 
	\lesssim \norm{ N^{d} \mlo^\vee(N \cdot) }_{L^{1}} \norm{ f }_{L^{p}}
	= \norm{ \mlo^\vee }_{L^{1}} \norm{ f }_{L^{p}} 
	\end{equation*}
	for any $p\in [1,\infty]$.
	
	For the inequality~\eqref{eq:bernstein pilo} we may now assume that $q$ is Schwartz by approximation.  We use the Bernstein inequality~\eqref{eq:bernstein 1} to estimate
	\begin{align*}
	\snorm{ \pilo^2 q^{(s+j)} }_{L^p}
	&\leq \sum_{K<N} \snorm{ P^2_K q^{(s+j)} }_{L^p}
	\lesssim \sum_{K<N} K^j \snorm{ P^2_K q^{(s)} }_{L^p}\\
	&\lesssim \sum_{K<N} K^j \snorm{ P^2_{<2N} q^{(s)} }_{L^p}
	\lesssim N^j \snorm{ P^2_{<2N} q^{(s)} }_{L^p} .
	\end{align*}
	Note that in the second line we inserted the operator $P^2_{<2N}$ since $P^2_KP^2_{<2N} = P_K^2$ for $K<N$, and then used the boundedness of the operators $P^2_K$.
	
	For the inequality~\eqref{eq:bernstein pihi}, we use the Bernstein inequalities~\eqref{eq:bernstein 1} and \eqref{eq:bernstein 2} to estimate
	\begin{align*}
	&\snorm{ \pihi^2 q^{(s-j)} }_{L^p}
	\leq \sum_{K<N} \tfrac{K^s}{N^s} \snorm{ P_K^2 \pihi q^{(s-j)} }_{L^p} + \snorm{ P^2_{\geq N} \pihi q^{(s-j)} }_{L^p} \\
	&\lesssim \sum_{K<N} \tfrac{K^{s-j}}{N^{s}} \snorm{ \pihi q^{(s)} }_{L^p} + N^{-j} \snorm{ P_{\geq N}^2 \pihi q^{(s)} }_{L^p}
	\lesssim N^{-j} \snorm{ \pihi q^{(s)} }_{L^p}
	\end{align*}
	for Schwartz $q$.  Note that in the second line we spent a factor of $K^j$ to insert $j$ derivatives on $q$, and then used the boundedness of the operators $P_K$.
\end{proof}

Next, we will prove an estimate for a commutator involving $\pihi$ and $\pilo$:
\begin{lem}
	\label{thm:comm lem}
	Let $\fatp_M^2 = \sum_{K =M/4}^{4M} P_{K}^2$ denote a fattened Littlewood--Paley projection.  Then for all bounded functions $w\in L^\infty(\R^2)$ and Schwartz functions $f,g,h$ we have
	\begin{align*}
	&\begin{aligned}
	\bigg| \iint \big[ &\big( \wh{P_M^2 \pihi^2 f} \big)(\xi) \big( \wh{\pilo^2 h} \big)(\xi-\eta) \\
	&- \big( \wh{ P_M \pihi \pilo f } \big)(\xi) \big( \wh{ P_M \pihi \pilo h } \big)(\xi-\eta) \big] \big( \wh{ P^2_{< \frac{M}{8}} g } \big)(\eta) w(\xi,\eta) \dxi\deta \bigg| 
	\end{aligned} \\
	&\lesssim \norm{w}_{L^\infty} \snorm{ P_M \pihi f }_{L^2} \snorm{ P_{<\frac{M}{8}}^2 g' }_{H^1} \big( \tfrac{M^2}{N^3} \snorm{ \fatp_M^2 \pilo^2 h }_{L^2} + \snorm{ P_M \pihi \pilo h }_{L^2} \big)
	\end{align*}
	uniformly for $\kappa$ large.
\end{lem}
\begin{proof}
	Within the square brackets, we are interchanging a factor of $P_M\pihi$ and $\pilo$ between $f$ and $h$.  We change to Fourier variables and break this maneuver into two steps, first moving $P_M\pihi$ and then moving $\pilo$:
	\begin{align}
	&\begin{multlined}
	\bigg| \iint \big[ \big( \wh{P_M^2 \pihi^2 f} \big)(\xi) \big( \wh{\pilo^2 h} \big)(\xi-\eta) \\
	- \big( \wh{ P_M \pihi \pilo f } \big)(\xi) \big( \wh{ P_M \pihi \pilo h } \big)(\xi-\eta) \big] \big( \wh{ P^2_{< \frac{M}{8}} g } \big)(\eta) w(\xi,\eta) \dxi\deta \bigg| 
	\end{multlined} \nonumber \\
	&\begin{aligned}
	= \iint \big( \wh{ P_M \pihi f } \big) (\xi) \big[ \psi(\tfrac{\xi}{M}) \mhi &(\tfrac{\xi}{N}) - \psi(\tfrac{\xi-\eta}{M}) \mhi(\tfrac{\xi-\eta}{N}) \big] \\ 
	&\times \big( \wh{ \pilo^2 h } \big)(\xi-\eta) \big( \wh{ P^2_{< \frac{M}{8}} g } \big)(\eta) w(\xi,\eta) \dxi\deta 
	\end{aligned}
	\label{eq:comm lem 1}\\
	&\begin{aligned}
	\phantom{={}}+ \iint \big( \wh{ P_M \pihi f } \big) (\xi) &\big[ \mlo(\tfrac{\xi-\eta}{N})- \mlo(\tfrac{\xi}{N}) \big] \\ 
	&\times \big( \wh{ P_M \pihi \pilo h } \big)(\xi-\eta) \big( \wh{ P^2_{< \frac{M}{8}} g } \big)(\eta) w(\xi,\eta) \dxi\deta ,
	\end{aligned}
	\label{eq:comm lem 2}
	\end{align}
	where $\psi$, $\mhi$, and $\mlo$ are the Fourier multipliers for the operators $P_M$, $\pihi$, and $\pilo$ respectively.  Observe that he RHS of the desired inequality vanishes for $M\geq 8N$.  Consequently, we will estimate the terms~\eqref{eq:comm lem 1} and~\eqref{eq:comm lem 2} for $M\leq 4N$ and note that they vanish for $M\geq 8N$.
	
	Observe that the integrand of the first term~\eqref{eq:comm lem 1} is supported in the region $\tfrac{M}{2} \leq |\xi| \leq 2M$, $|\eta|\leq \tfrac{M}{8}$.  On this region we have
	\begin{align*}
	|\xi-\eta| &\geq |\xi| - |\eta| \geq \tfrac{M}{2} - \tfrac{M}{8} \geq \tfrac{M}{4} , \qquad
	|\xi-\eta| \leq |\xi| + |\eta| \leq 2M + \tfrac{M}{8} \leq 4M .
	\end{align*}
	Therefore we can insert $\sum_{K =M/4}^{4M} \psi^2(\tfrac{\xi-\eta}{K})$ into the integrand, which is the Fourier multiplier for the fattened Littlewood--Paley projection $\fatp_M^2 = \sum_{K =M/4}^{4M} P_{K}^2$ applied to $h$.  Now $\fatp_M^2 \pilo^2 h$ vanishes for $M\geq 8N$, and so we may assume $M\leq 4N$.
	
	Next, we will estimate the first term~\eqref{eq:comm lem 1}.  By the fundamental theorem of calculus,
	\begin{align*}
	\big| \psi(\tfrac{\xi}{M}) \mhi(\tfrac{\xi}{N}) - \psi(\tfrac{\xi-\eta}{M}) \mhi(\tfrac{\xi-\eta}{N}) \big|
	&\leq \int_0^1 s|\eta| \big| \big( \psi(\tfrac{\cdot}{M}) \mhi(\tfrac{\cdot}{N}) \big)'(\xi-s\eta) \big| \ds \\
	&\lesssim |\eta| \tfrac{M^{s-1}}{N^s} \quad\tx{for }M\leq N .
	\end{align*}
	In the last inequality, we note that $\psi(\tfrac{\cdot}{M}) \mhi(\tfrac{\cdot}{N})$ is a function with amplitude $M^s/N^s$ supported in an annulus of width $M$; indeed, for $M\leq N$ we have
	\begin{equation*}
	\big| \big( \psi(\tfrac{\xi}{M}) \mhi(\tfrac{\xi}{N}) \big)' \big|
	\leq \big| \psi(\tfrac{\xi}{M})' \mhi(\tfrac{\xi}{N}) \big| + \big| \psi(\tfrac{\xi}{M}) \mhi(\tfrac{\xi}{N})' \big|
	\lesssim M^{-1} \cdot \tfrac{M^s}{N^s} + 1 \cdot \tfrac{M^{s-1}}{N^s} .
	\end{equation*}
	This yields
	\begin{align*}
	|\eqref{eq:comm lem 1}|
	&\lesssim \norm{w}_{L^\infty} \tfrac{M^{s-1}}{N^s} \snorm{ \wh{ P_M \pihi f } }_{L^2} \snorm{ \wh{ P^2_{< \frac{M}{8}} g' } }_{L^1} \snorm{ \wh{ \fatp_M^2 \pilo^2 h } }_{L^2} \\
	&\lesssim \norm{w}_{L^\infty} \tfrac{M^{s-1}}{N^s} \snorm{ P_M \pihi f }_{L^2} \snorm{ P^2_{< \frac{M}{8}} g' }_{H^1} \snorm{ \fatp_M^2 \pilo^2 h }_{L^2} .
	\end{align*}
	In the last inequality, we used Cauchy--Schwarz to estimate
	\begin{equation*}
	\int \big| \big( \wh{ P^2_{< \frac{M}{8}} g' } \big) (\xi) \big| \dxi
	\leq \left( \int \frac{\dd \xi}{\xi^2+1} \right)^{\tfrac{1}{2}} \left( \int (\xi^2+1) \big| \big( \wh{ P^2_{< \frac{M}{8}} g' } \big) (\xi) \big|^2 \dxi \right)^{\tfrac{1}{2}} .
	\end{equation*}
	
	For the second term~\eqref{eq:comm lem 2}, we note that the Fourier support of $\pihi\pilo h$ is bounded by $N$; in particular, $P_M \pihi\pilo h$ vanishes for $M\geq 8N$.  For $M\leq 4N$ we estimate
	\begin{align*}
	\big| \mlo(\tfrac{\xi-\eta}{N}) - \mlo(\tfrac{\xi}{N}) \big|
	&\leq \int_0^1 s|\eta| \big| \big( \mlo(\tfrac{\cdot}{N}) \big)'(\xi-s\eta) \big| \ds
	\lesssim |\eta| N^{-1} .
	\end{align*}
	This yields
	\begin{align*}
	|\eqref{eq:comm lem 2}|
	&\lesssim \norm{w}_{L^\infty} N^{-1} \snorm{ \wh{ P_M \pihi f } }_{L^2} \snorm{ \wh{ P^2_{< \frac{M}{8}} g' } }_{L^1} \snorm{ \wh{ P_M\pihi\pilo h } }_{L^2} \\
	&\lesssim \norm{w}_{L^\infty} \tfrac{M^{s-1}}{N^s} \snorm{ P_M \pihi f }_{L^2} \snorm{ P^2_{< \frac{M}{8}} g' }_{H^1} \snorm{ P_M\pihi\pilo h }_{L^2} .
	\end{align*}
	Combining this with the estimate of~\eqref{eq:comm lem 1}, the claim follows.
\end{proof}

We are now equipped to prove our equicontinuity statement.  Let $Q(N) \subset H^s$ for $N\in 2^{\N}$ be bounded sets of initial data that satisfy
\begin{equation}
Q(M) \supset Q(N) \tx{ for }M\leq N,
\quad\tx{and}\quad
\lim_{N\to\infty}\,\sup_{q(0)\in Q(N)}\, \snorm{ \pihi q(0) }_{H^s} = 0.
\label{eq:equicty hyp}
\end{equation}
\begin{prop}
	\label{thm:equicty}
	Fix an integer $s\geq 3$ and define the corresponding projection operator~\eqref{eq:proj op}.  Given $T>0$ and bounded sets $Q(N)\subset H^s$ of initial data satisfying~\eqref{eq:equicty hyp}, the corresponding solutions $q_\kappa(t)$ to the tidal $H_\kappa$ flow~\eqref{eq:thk} obey
	\begin{equation*}
	\lim_{N\to\infty}\, \sup_{q(0)\in Q(N)}\, \sup_{\kappa \geq N}\ \snorm{ \pihi q_\kappa(t) }_{C_tH^s([-T,T]\times\R)} = 0 .
	\end{equation*}
\end{prop}
\begin{proof}
	Expanding $g(\kappa,q_\kappa+W)$ in powers of $q_\kappa$ and $W$, we write
	\begin{align}
	&\ddt \left( \snorm{ \pihi q^{(s)}_\kappa }_{L^2}^2 \right)
	\nonumber \\
	&= \int \big( \pihi^2 q^{(s)}_\kappa \big) \big\{ {-16}\kappa^5 \langle\del_x, R_0q_\kappa R_0 \del_x \rangle + 4\kappa^2 q_\kappa \big\}^{(s+1)} \dx
	\label{eq:equicty 1} \\
	&\phantom{={}}+ \int \big( \pihi^2 q^{(s)}_\kappa \big) \big\{ {-16}\kappa^5 \langle\del_x, R_0W R_0 \del_x \rangle + 4\kappa^2 W \big\}^{(s+1)} \dx 
	\label{eq:equicty 2} \\
	&\phantom{={}}+ 16\kappa^5 \int \big( \pihi^2 q^{(s)}_\kappa \big) \langle\del_x, R_0q_\kappa R_0q_\kappa R_0 \del_x \rangle^{(s+1)} \dx
	\label{eq:equicty 3} \\
	&\phantom{={}}+ 16\kappa^5 \int \big( \pihi^2 q^{(s)}_\kappa \big) \big\{ \langle\del_x, ( R_0WR_0q_\kappa R_0 + R_0q_\kappa R_0WR_0 ) \del_x \rangle \big\}^{(s+1)} \dx
	\label{eq:equicty 4} \\
	&\phantom{={}}+ 16\kappa^5 \int \big( \pihi^2 q^{(s)}_\kappa \big) \langle\del_x, R_0W R_0W R_0 \del_x \rangle^{(s+1)} \dx
	\label{eq:equicty 5} \\
	&\begin{aligned}
	\phantom{={}}+ 16\kappa^5 \int \big( \pihi^2 q^{(s)}_\kappa \big) \big\{ &g(\kappa,q_\kappa+W) + \langle \del_x , R_0(q_\kappa+W)R_0 \del_x \rangle \\[-0.5em]
	&- \langle \del_x , R_0(q_\kappa+W)R_0(q_\kappa+W)R_0 \del_x \rangle \big\}^{(s+1)} \dx .
	\end{aligned}
	\label{eq:equicty 6}
	\end{align}
	We will estimate the terms \eqref{eq:equicty 1}--\eqref{eq:equicty 6} separately.
	
	The first linear term \eqref{eq:equicty 1} vanishes.  To see this, we use the first operator identity of~\eqref{eq:linear op id} to write
	\begin{equation*}
	\eqref{eq:equicty 1}
	= \int \big( \pihi^2 q^{(s)}_\kappa \big) \big\{ {-16}\kappa^4 R_0(2\kappa) q_\kappa + 4\kappa^2 q_\kappa \big\}^{(s+1)} \dx 
	= 0 .
	\end{equation*}
	In the last equality we note that the integrand is odd in Fourier variables, or equivalently that the integrand is a total derivative because differentiation commutes with the Fourier multipliers $\pihi$ and $R_0$.
	
	Now we estimate the linear contribution~\eqref{eq:equicty 2} from $W$.  Using the operator identity~\eqref{eq:quad op id}, we write
	\begin{align*}
	|\eqref{eq:equicty 2}|
	&= \left| \int \big( \pihi^2 q^{(s)}_\kappa \big) \big\{ - W^{(s+3)} - R_0(2\kappa) W^{(s+5)} \big\} \dx \right| \\
	&\lesssim \snorm{\pihi q^{(s)}_\kappa}_{L^2} \big( \snorm{ \pihi W^{(s+3)}}_{L^2} + \kappa^{-2} \snorm{ W^{(s+5)} }_{L^2} \big) .
	\intertext{Recalling that $W'$ is Schwartz and $\kappa \geq N$, we obtain}
	&\lesssim \snorm{\pihi q^{(s)}_\kappa}_{L^2} \cdot N^{-2}
	\lesssim \snorm{\pihi q^{(s)}_\kappa}_{L^2}^2 + N^{-4} .
	\end{align*}

	Next, we turn to the first quadratic contribution~\eqref{eq:equicty 3}, which is nonvanishing due to the presence of the frequency cutoff $\pihi^2$.  We write
	\begin{align*}
	|\eqref{eq:equicty 3}|
	&= 16\kappa^5 \left| \tr\big\{ \big( \pihi^2 q^{(s)}_\kappa \big) R_0 \big[ \partial^{s+1} ,  q_\kappa R_0 q_\kappa R_0 \big] \big\} \right| \\
	&\lesssim \sum_{j=0}^{s+1} \kappa^5 \left| \tr\big\{ \big( \pihi^2 q^{(s)}_\kappa \big) R_0 q_\kappa^{(j)} R_0 q_\kappa^{(s+1-j)} R_0 \big\} \right| .
	\end{align*}
	Decomposing the highest order $q_\kappa = \pihi^2 q_\kappa + \pilo^2 q_\kappa$ we have
	\begin{align}
	&|\eqref{eq:equicty 3}|
	\nonumber \\
	&\begin{aligned} 
	\lesssim \sum_{j=0}^{\lfloor \frac{s+1}{2} \rfloor} \kappa^5 \big| &\tr\big\{ \big( \pihi^2 q^{(s)}_\kappa \big) R_0 q_\kappa^{(j)} R_0 \big( \pihi^2 q_\kappa^{(s+1-j)} \big) R_0 \big\}\\[-0.5em]
	&+ \tr\big\{ \big( \pihi^2 q^{(s)}_\kappa \big) R_0 \big( \pihi^2 q_\kappa^{(s+1-j)} \big) R_0 q_\kappa^{(j)} R_0 \big\} \big|
	\end{aligned}
	\label{eq:equicty 7}\\
	&\begin{aligned} 
	\phantom{\lesssim{}} + \sum_{j=0}^{\lfloor \frac{s+1}{2} \rfloor} \kappa^5 \big| &\tr\big\{ \big( \pihi^2 q^{(s)}_\kappa \big) R_0 q_\kappa^{(j)} R_0 \big( \pilo^2 q_\kappa^{(s+1-j)} \big) R_0 \big\} \\[-0.5em]
	&+ \tr\big\{ \big( \pihi^2 q^{(s)}_\kappa \big) R_0 \big( \pilo^2 q_\kappa^{(s+1-j)} \big) R_0 q_\kappa^{(j)} R_0 \big\} \big| .
	\end{aligned}
	\label{eq:equicty 8}
	\end{align}
	
	First we will estimate the high-frequency contribution~\eqref{eq:equicty 7}.  We can ``integrate by parts'' to eliminate the terms with $q_\kappa^{(s+1)}$.  Specifically, by cycling the trace we have
	\begin{align*}
	&\tr\big\{ \big( \pihi^2 q^{(s)}_\kappa \big) R_0 \big( \pihi^2 q_\kappa^{(s+1)} \big) R_0 q_\kappa R_0 \big\} + \tr\big\{ \big( \pihi^2 q^{(s)}_\kappa \big) R_0 q_\kappa R_0 \big( \pihi^2 q_\kappa^{(s+1)} \big) R_0 \big\} \\
	&= \tr\big\{ \big[ \partial, \big( \pihi^2 q^{(s)}_\kappa \big) R_0 \big( \pihi^2 q_\kappa^{(s)} \big) R_0 \big] q_\kappa R_0 \big\} \\
	&= - \tr\big\{ \big( \pihi^2 q^{(s)}_\kappa \big) R_0 \big( \pihi^2 q_\kappa^{(s)} \big) R_0 q'_\kappa R_0 \big\} .
	\end{align*}
	For the remaining terms we use the Hilbert--Schmidt norm estimate~\eqref{eq:hskappa identity 1} and the observation $\norm{f}_{H^{-1}_\kappa} \lesssim \kappa^{-1} \norm{f}_{L^2}$ to put the two highest order terms in $L^2$, and we put the remaining terms in operator norm:
	\begin{equation*}
	|\eqref{eq:equicty 7}| 
	\lesssim \sum_{j=1}^{\lfloor \frac{s+1}{2} \rfloor} \snorm{ \pihi^2 q^{(s)}_\kappa }_{L^2} \snorm{ q_\kappa^{(j)} }_{L^\infty} \snorm{ \pihi^2 q_\kappa^{(s+1-j)} }_{L^2} .
	\end{equation*}
	As $s\geq 3$ then the index $j$ is at most $s-1$, and so the term $\snorm{ q_\kappa^{(j)} }_{L^\infty}$ is uniformly bounded for $|t|\leq T$ and $\kappa \geq \kappa_0$ by the embedding $H^1\hookrightarrow L^\infty$ and the \ti{a priori} estimate of \cref{thm:a priori Hs}.  The remaining term $\snorm{ \pihi^2 q_\kappa^{(s+1-j)} }_{L^2}$ either matches the first factor $\snorm{ \pihi^2 q^{(s)}_\kappa }_{L^2}$ or is $\lesssim N^{-1}$ by the Bernstein inequality~\eqref{eq:bernstein pihi}.  Altogether we conclude
	\begin{equation*}
	|\eqref{eq:equicty 7}|
	\lesssim \snorm{ \pihi^2 q^{(s)}_\kappa }_{L^2}^2 + \snorm{ \pihi^2 q^{(s)}_\kappa }_{L^2} \cdot N^{-1}
	\lesssim \snorm{ \pihi q^{(s)}_\kappa }_{L^2}^2 + N^{-2} .
	\end{equation*}
	
	The low-frequency contribution~\eqref{eq:equicty 8} requires more manipulation.  We will push one factor of $\pihi$ onto the low-frequency term and the resulting frequency cancellation will yield an acceptable contribution.  As $\pihi$ is not a sharp frequency cutoff, we divide the first factor $\pihi^2 q_\kappa^{(s)}$ into its frequency scales:
	\begin{equation}
	\begin{aligned}
	|\eqref{eq:equicty 8}| \lesssim \sum_{M} \sum_{j=0}^{\lfloor \frac{s+1}{2} \rfloor} \kappa^5 \big| &\tr\big\{ \big( P_M^2 \pihi^2 q^{(s)}_\kappa \big) R_0 q_\kappa^{(j)} R_0 \big( \pilo^2 q_\kappa^{(s+1-j)} \big) R_0 \big\} \\[-0.5em]
	&+ \tr\big\{ \big( P_M^2 \pihi^2 q^{(s)}_\kappa \big) R_0 \big( \pilo^2 q_\kappa^{(s+1-j)} \big) R_0 q_\kappa^{(j)} R_0 \big\} \big| .
	\end{aligned}
	\label{eq:equicty 14}
	\end{equation}
	
	Consider the first summand of~RHS\eqref{eq:equicty 14}.  We split $q_\kappa^{(j)} = P^2_{\geq \frac{M}{8}} q_\kappa^{(j)} + P^2_{< \frac{M}{8}} q_\kappa^{(j)}$ into high and low frequencies; the high-frequency contribution can be estimated directly, and for the low-frequency term we trade factors of $P_M\pihi$ and $\pilo$ between $q_\kappa^{(s)}$ and $q_\kappa^{(s+1-j)}$ to create a commutator:
	\begin{align}
	&\kappa^5 \tr\big\{ \big( P_M^2 \pihi^2 q^{(s)}_\kappa \big) R_0 q_\kappa^{(j)} R_0 \big( \pilo^2 q_\kappa^{(s+1-j)} \big) R_0 \big\}
	\nonumber \\
	&= \kappa^5 \tr\big\{ \big( P_M^2 \pihi^2 q^{(s)}_\kappa \big) R_0 \big( P^2_{\geq \frac{M}{8}} q_\kappa^{(j)} \big) R_0 \big( \pilo^2 q_\kappa^{(s+1-j)} \big) R_0 \big\} 
	\label{eq:equicty 9}\\
	&+ \kappa^5 \tr\big\{ \big( P_M \pihi \pilo q^{(s)}_\kappa \big) R_0 \big( P^2_{< \frac{M}{8}} q_\kappa^{(j)} \big) R_0 \big( P_M \pihi \pilo q_\kappa^{(s+1-j)} \big) R_0 \big\}
	\label{eq:equicty 10}\\
	&\begin{aligned}
	{}+ \kappa^5 &\tr\big\{ \big[ \big( \pilo^2 q_\kappa^{(s+1-j)} \big) R_0 \big( P_M^2 \pihi^2 q^{(s)}_\kappa \big) R_0 \\
	&- \big( P_M \pihi \pilo q_\kappa^{(s+1-j)} \big) R_0 \big( P_M \pihi \pilo q^{(s)}_\kappa \big) R_0 \big] \big( P^2_{< \frac{M}{8}} q_\kappa^{(j)} \big) R_0 \big\} .
	\end{aligned}
	\label{eq:equicty 11}
	\end{align}
	
	For the term~\eqref{eq:equicty 9} we put the two highest order terms in $L^2$ and the lowest order term in $L^\infty$.  This yields
	\begin{equation*}
	|\eqref{eq:equicty 9}|
	\lesssim \begin{cases} \min\{ \tfrac{M^s}{N^s} , 1 \} \snorm{ P^2_M \pihi q_\kappa^{(s)} }_{L^2} \cdot M^{-2} \cdot N &\tx{if }j=0 ,\\[0.3em]
	\min\{ \tfrac{M^s}{N^s} , 1 \}  \snorm{ P^2_M \pihi q_\kappa^{(s)} }_{L^2} \cdot M^{-1} \cdot 1 &\tx{if }j\geq 1 .
	\end{cases}
	\end{equation*}
	
	For the term~\eqref{eq:equicty 10}, we can now integrate by parts for the $j=0$ case:
	\begin{align*}
	&\tr\big\{ \big( P_M \pihi \pilo q^{(s)}_\kappa \big) R_0 \big( P^2_{< \frac{M}{8}} q_\kappa \big) R_0 \big( P_M \pihi \pilo q_\kappa^{(s+1)} \big) R_0 \big\}\\
	&\phantom{={}} + \tr\big\{ \big( P_M \pihi \pilo q^{(s)}_\kappa \big) R_0 \big( P_M \pihi \pilo q_\kappa^{(s+1)} \big) R_0 \big( P^2_{< \frac{M}{8}} q_\kappa \big) R_0 \big\} \\
	&= \tr\big\{ \big[ \partial, \big( P_M \pihi \pilo q^{(s)}_\kappa \big) R_0 \big( P_M \pihi \pilo q^{(s)}_\kappa \big) R_0 \big] \big( P^2_{< \frac{M}{8}} q_\kappa \big) R_0 \big\} \\
	&= - \tr\big\{ \big( P_M \pihi \pilo q^{(s)}_\kappa \big) R_0 \big( P_M \pihi \pilo q^{(s)}_\kappa \big) R_0 \big( P^2_{< \frac{M}{8}} q'_\kappa \big) R_0 \big\} ,
	\end{align*}
	which is now the summand for $j=1$.  For $j\geq 1$ we put the two highest order terms in $L^2$ and the lowest order term in $L^\infty$ to obtain
	\begin{equation*}
	|\eqref{eq:equicty 10}|
	\lesssim \begin{cases} \snorm{ P_M \pihi q_\kappa^{(s)} }_{L^2}^2 \cdot 1 &\tx{if }j=1 ,\\[0.3em]
	\snorm{ P_M \pihi q_\kappa^{(s)} }_{L^2} \cdot 1 \cdot N^{-1} \min\{ \tfrac{M^s}{N^s},  1 \} &\tx{if }j\geq 2 .
	\end{cases}
	\end{equation*}

	For the commutator term~\eqref{eq:equicty 11} we will apply the estimate of \cref{thm:comm lem} to the functions $f = q_\kappa^{(s)}$, $g = q_\kappa^{(j)}$, and $h = q_\kappa^{(s+1-j)}$.  Writing the trace as an iterated integral and changing to Fourier variables, we have
	\begin{align*}
	&\begin{aligned}
	\eqref{eq:equicty 11} = \kappa^5 \tr\big\{ \big[ &\big( \pilo^2 h \big) R_0 \big( P_M^2 \pihi^2 f \big) R_0 \\
	&- \big( P_M \pihi \pilo h \big) R_0 \big( P_M \pihi \pilo f \big) R_0 \big] \big( P^2_{< \frac{M}{8}} g \big) R_0 \big\} .
	\end{aligned} \\
	&\begin{aligned}
	= \frac{ \kappa^5 }{ (2\pi)^{\frac{3}{2}} } \iiint &\big[ \big( \wh{ \pilo^2 h} \big)(\xi_1 - \xi_3) \big( \wh{ P_M^2 \pihi^2 f } \big) (\xi_3-\xi_2) \\[-0.5em]
	&- \big( \wh{ P_M \pihi \pilo h } \big) (\xi_1-\xi_3) \big( \wh{ P_M \pihi \pilo f } \big) (\xi_3-\xi_2) \big] \\
	&\times \frac{ \big( \wh{ P^2_{< \frac{M}{8}} g } \big) (\xi_2-\xi_1) }{ (\xi_3^2 + \kappa^2) (\xi_2^2 + \kappa^2) (\xi_1^2 + \kappa^2) } \dxi_1\dxi_2\dxi_3 .
	\end{aligned}
	\end{align*}
	Changing variables $\eta_1 = \xi_2-\xi_1$, $\eta_2 = \xi_3 - \xi_2$, $\eta_3 = \xi_3$, this becomes
	\begin{align*}
	\eqref{eq:equicty 11} = \frac{ \kappa^5 }{ (2\pi)^{\frac{3}{2}} } \iiint &\big[ \big( \wh{ \pilo^2 h} \big)(-\eta_1 - \eta_2) \big( \wh{ P_M^2 \pihi^2 f } \big) (\eta_2) \\
	&- \big( \wh{ P_M \pihi \pilo h } \big) (-\eta_1-\eta_2) \big( \wh{ P_M \pihi \pilo f } \big) (\eta_2) \big] \\
	&\times \frac{ \big( \wh{ P^2_{< \frac{M}{8}} g } \big) (\eta_1) \deta_1\deta_2\deta_3 }{ (\eta_3^2 + \kappa^2) ((\eta_3-\eta_2)^2 + \kappa^2) ((\eta_3-\eta_1-\eta_2)^2 + \kappa^2) }  .
	\end{align*}
	The functions $f$, $g$, and $h$ are now independent of $\eta_3$, and so we may evaluate the $\eta_3$ integral using residue calculus:
	\begin{align*}
	\eqref{eq:equicty 11} = \frac{ \kappa^4 }{ 2 (2\pi)^{\frac{1}{2}} } \iiint &\big[ \big( \wh{ \pilo^2 h} \big)(-\eta_1 - \eta_2) \big( \wh{ P_M^2 \pihi^2 f } \big) (\eta_2) \\
	&- \big( \wh{ P_M \pihi \pilo h } \big) (-\eta_1-\eta_2) \big( \wh{ P_M \pihi \pilo f } \big) (\eta_2) \big] \\
	&\times \frac{ \big( \wh{ P^2_{< \frac{M}{8}} g } \big) (\eta_1) (24\kappa^2 + \eta_1^2 + \eta_2^2 + (\eta_1+\eta_2)^2 ) }{ (\eta_1^2 + 4\kappa^2) (\eta_2^2 + 4\kappa^2) ((\eta_1+\eta_2)^2 + \kappa^2) } \deta_1\deta_2 .
	\end{align*}
	This is now of the form of \cref{thm:comm lem} for the multiplier
	\begin{equation*}
	w(\xi,\eta) = \frac{ \kappa^4 (24\kappa^2 + \eta_1^2 + \eta_2^2 + (\eta_1+\eta_2)^2 ) }{ 2(2\pi)^{\frac{1}{2}} (\eta_1^2 + 4\kappa^2) (\eta_2^2 + 4\kappa^2) ((\eta_1+\eta_2)^2 + \kappa^2) } .
	\end{equation*}
	Moreover, this multiplier is bounded uniformly in $\kappa$:
	\begin{equation*}
	\norm{ w }_{L^\infty} = \tfrac{3}{16} (2\pi)^{-\frac{1}{2}} \quad\tx{for all }\kappa > 0.
	\end{equation*}
	Therefore, by \cref{thm:comm lem} and the Bernstein inequalities~\eqref{eq:bernstein pihi} and~\eqref{eq:bernstein pilo} we have
	\begin{equation*}
	|\eqref{eq:equicty 11}|
	\lesssim \begin{cases}
	\snorm{ P_M\pihi q_\kappa^{(s)} }_{L^2} \snorm{ \fatp_M^2 \pihi q_\kappa^{(s)} }_{L^2} + \snorm{ P_M\pihi q_\kappa^{(s)} }_{L^2}^2 & j=0 ,\\[0.3em]
	\tfrac{M^{s-1}}{N^s} \snorm{P_M \pihi q_\kappa^{(s)} }_{L^2} + N^{-1} \snorm{ P_M\pihi q_\kappa^{(s)} }_{L^2}^2 & j\geq 1 ,
	\end{cases}
	\end{equation*}
	for $M\leq 4N$.  
	
	We repeat the decomposition~\eqref{eq:equicty 9}--\eqref{eq:equicty 11} for the second term in the summand of~RHS\eqref{eq:equicty 14}.  At each step we obtain the same estimates; indeed, although we cannot commute the operators within the trace, we still obtain the same integral because $w$ was symmetric in $\xi$ and $\eta$.
	
	Altogether, we obtain the following estimate of the low-frequency quadratic contribution~\eqref{eq:equicty 8}:
	\begin{equation*}
	|\eqref{eq:equicty 8}|
	\lesssim \sum_M \snorm{ \fatp_M \pihi q_\kappa^{(s)} }_{L^2}^2 + \sum_{M\leq 4N} \tfrac{M}{N^2} + \sum_{M\geq N} \tfrac{1}{M} 
	\lesssim \snorm{ \pihi q_\kappa^{(s)} }_{L^2}^2 + N^{-1} .
	\end{equation*}
	In the last inequality, we noted that the sum of the multipliers in Fourier variables is bounded.
	
	For the quadratic term~\eqref{eq:equicty 4} involving $q_\kappa$ and $W$ we can repeat the decomposition \eqref{eq:equicty 7}--\eqref{eq:equicty 11}.  Previously we put $q_\kappa^{(0)}$ in $L^\infty$ and not $L^2$ since it was the lowest order term, and consequently the same estimates apply because $W\in L^\infty$ and $W'$ is Schwartz.
	
	The quadratic term~\eqref{eq:equicty 5} for $W$ can be estimated directly.  Extracting the leading term as $\kappa\to\infty$, we write
	\begin{align}
	&\eqref{eq:equicty 5} 
	\nonumber \\
	&= \int \big( \pihi^2 q^{(s)}_\kappa \big) (3W^2)^{(s+1)} \dx 
	\label{eq:equicty 15}\\
	&\phantom{={}}+ \int \big( \pihi^2 q^{(s)}_\kappa \big) \big\{ 16\kappa^5 \langle\del_x, R_0W R_0W R_0 \del_x \rangle - 3W^2 \big\}^{(s+1)} \dx .
	\label{eq:equicty 16}
	\end{align}
	For~\eqref{eq:equicty 15} we distribute the $s+1$ derivatives and move one $\pihi$ off of $q_\kappa$:
	\begin{align*}
	|\eqref{eq:equicty 15}|
	&\lesssim \sum_{j=0}^{s+1} \left| \int \big( \pihi q^{(s)}_\kappa \big)\, \pihi\big( W^{(j)} W^{(s+1-j)} \big) \dx \right|\\
	&\lesssim \snorm{\pihi q^{(s)}_\kappa}_{L^2} \cdot N^{-1}
	\lesssim \snorm{\pihi q^{(s)}_\kappa}_{L^2}^2 + N^{-2} .
	\end{align*}
	In the second line we noted that $W^{(j)} W^{(s+1-j)}$ is Schwartz since $W'$ is Schwartz and $W\in L^\infty$ is smooth.  For~\eqref{eq:equicty 16} we use the operator identity~\eqref{eq:quad op id} and the estimates $\snorm{ R_{0}(2 \kappa) \partial^j }_\op \lesssim \kappa^{j-2}$ for $j=0,1,2$ (the estimate for $j = 0$ is also true as an operator on $L^\infty$ by the explicit kernel formula for $R_0$ and Young's inequality) to prove by duality that
	\begin{equation*}
	\norm{ 16\kappa^5 \langle \del_x, R_0(\kappa)fR_0(\kappa)hR_0(\kappa) \del_x \rangle - 3fg }_{L^2}
	\lesssim \kappa^{-2} \norm{ f }_{W^{2,\infty}} \norm{ h }_{H^2} .
	\end{equation*}
	Moreover, the roles of $f$ and $h$ can be exchanged since the identity~\eqref{eq:quad op id} is symmetric in $f$ and $h$.  Distributing the $s+1$ derivatives and recalling $\kappa \geq N$, we estimate
	\begin{equation*}
	|\eqref{eq:equicty 16}|
	\lesssim N^{-2} \snorm{\pihi q^{(s)}_\kappa}_{L^2} \norm{ W }_{W^{s+3,\infty}} \norm{ W' }_{H^{s+3}}
	\lesssim \snorm{\pihi q^{(s)}_\kappa}_{L^2}^2 + N^{-4} .
	\end{equation*}
	
	Finally, we estimate the tail~\eqref{eq:equicty 6} using Cauchy--Schwarz and~\eqref{eq:tail conv}:
	\begin{equation*}
	| \eqref{eq:equicty 6} |
	\lesssim \snorm{\pihi^2 q_\kappa^{(s)}}_{L^2} \cdot o(1)
	\lesssim \snorm{\pihi^2 q_\kappa^{(s)}}_{L^2}^2 + o(1)
	\end{equation*}
	uniformly for $\kappa \geq N$ as $N\to\infty$.  Note that $o(1)$ as $\kappa\to\infty$ implies $o(1)$ as $N\to\infty$ due to the restriction $\kappa \geq N$.

	Altogether, we have shown there exists a constant $C$ such that
	\begin{equation*}
	\left| \ddt \snorm{ \pihi q^{(s)}_\kappa(t) }_{L^2}^2 \right| \leq C \snorm{ \pihi q^{(s)}_\kappa(t) }_{L^2}^2 + o(1) \quad\tx{as }N\to\infty,
	\end{equation*}
	uniformly for $|t|\leq T$, $\kappa\geq N$, and $q(0)\in Q(N)$.  By Gr\"onwall's inequality, we then have
	\begin{equation*}
	\snorm{ \pihi q^{(s)}_\kappa(t) }_{L^2}^2
	\leq e^{CT} \snorm{ \pihi q^{(s)}(0) }_{L^2}^2 + o(1)
	\quad\tx{as }N\to\infty ,
	\end{equation*}
	uniformly for $|t|\leq T$, $\kappa\geq N$, and $q(0)\in Q(N)$.  By~\eqref{eq:equicty hyp}, the term $\snorm{ \pihi q^{(s)}(0) }_{L^2}$ converges to zero as $N\to\infty$ uniformly for $q(0) \in Q(N)$.  Therefore we conclude
	\begin{equation*}
	\sup_{q(0)\in Q(N)}\, \sup_{\kappa \geq N}\ \snorm{ \pihi q_\kappa(t) }_{C_tH^s([-T,T]\times\R)} \to 0 \quad\tx{as }N\to\infty ,
	\end{equation*}
	as desired.
\end{proof}

\section{Well-posedness}
\label{sec:wp}

The goal of this section is to prove our two main results, \cref{thm:intro H-1,thm:intro H3}.  The first step is show that the tidal $H_\kappa$ flows converge in $H^s$ as $\kappa\to\infty$ by combining the low-regularity convergence of \cref{thm:conv low reg} and the uniform Fourier tail control from \cref{thm:equicty}:
\begin{prop}
	\label{thm:conv}
	Fix an integer $s\geq 3$ and $T>0$.  Given bounded sets $Q(\kappa)\subset H^s$ of initial data satisfying~\eqref{eq:equicty hyp}, the corresponding tidal $H_\kappa$ solutions $q_\kappa(t)$ are Cauchy in $C_tH^s([-T,T]\times\R)$ as $\kappa\to\infty$ uniformly for $q(0) \in Q(\kappa)$.
\end{prop}
\begin{proof}
	In the following all spacetime norms will be over the slab $[-T,T]\times \R$.  Splitting at a large frequency $N$ to be chosen, we estimate
	\begin{equation}
	\norm{ q_\varkappa - q_\kappa }_{C_tH^s}^2
	\lesssim (N+1)^{s+2} 	\norm{ q_\varkappa - q_\kappa }_{C_tH^{-2}}^2 + \norm{ q_\varkappa - q_\kappa }_{C_tH^s(|\xi|\geq N)}^2 .
	\label{eq:conv 1}
	\end{equation}
	For the second term we estimate 
	\begin{equation}
	\norm{ q_\varkappa - q_\kappa }_{C_tH^s(|\xi|\geq N)}^2
	\leq 2 \big( \norm{ \Pi_{\geq N} q_\varkappa }_{C_tH^s}^2 + \norm{ \Pi_{\geq N} q_\kappa }_{C_tH^s}^2 \big) .
	\label{eq:conv 2}
	\end{equation}
	Fix $\eps>0$.  First, by \cref{thm:equicty} we take $N = N_0$ sufficiently large to ensure that RHS\eqref{eq:conv 2} is bounded by $\eps/2$ for all $\varkappa,\kappa \geq N_0$.  With $N_0$ fixed, we then use \cref{thm:conv low reg} to pick $\kappa_0 \geq N_0$ so that the first term of RHS\eqref{eq:conv 1} is bounded by $\eps/2$ for all $\varkappa,\kappa \geq \kappa_0$.  Together, we conclude that $\norm{ q_\varkappa - q_\kappa }_{H^s}^2 \leq \eps$ for all $\varkappa,\kappa \geq \kappa_0$.
\end{proof}

Next, we show that the limits guaranteed by \cref{thm:conv} solve tidal KdV:
\begin{prop}
	\label{thm:tkdv exist}
	Fix an integer $s\geq 3$ and $T>0$.  Given initial data $q(0) \in H^s(\R)$, there exists a corresponding solution $q(t)$ to tidal KdV~\eqref{eq:tkdv} in $(C_tH^s\cap C^1_tH^{s-3})([-T,T]\times\R)$.
\end{prop}
\begin{proof}
	 In the following all spacetime norms will be taken over the slab $[-T,T]\times\R$.  Applying \cref{thm:conv} to the single function $Q = \{ q(0) \}$, we define $q(t)$ to be $\lim_{\kappa\to\infty} q_\kappa(t)$ which we know exists in $C_tH^s$.  It remains to show that $\ddt q$ is in $C_tH^{s-3}$ and is equal to the RHS of tidal KdV~\eqref{eq:tkdv}.  We already know that the RHS\eqref{eq:tkdv} is in $C_tH^{s-3}$, so it suffices to show that $\ddt q_\kappa$ converges to RHS\eqref{eq:tkdv} in the lower regularity norm $C_tH^{-1}$.
	
	We will extract the linear and quadratic terms of the tidal $H_\kappa$ flow to witness its convergence to tidal KdV.  Using the translation identity~\eqref{eq:g trans prop 2}, we write
	\begin{align}
	&\ddt q_\kappa
	\nonumber \\
	&= - 16\kappa^5 \langle \del_x , R_0 q'_\kappa R_0 \del_x \rangle + 4\kappa^2 q'_\kappa 
	\label{eq:qdot conv 1}\\
	&\phantom{={}}- 16\kappa^5 \langle \del_x , R_0 W' R_0 \del_x \rangle + 4\kappa^2 W'
	\label{eq:qdot conv 2}\\
	&\phantom{={}}+ 16\kappa^5 \langle \del_x , [\partial, R_0 q_\kappa R_0 q_\kappa R_0] \del_x \rangle
	\label{eq:qdot conv 3}\\
	&\phantom{={}}+ 16\kappa^5 \big\{ \langle \del_x , [\partial, R_0 W R_0 q_\kappa R_0] \del_x \rangle +  \langle \del_x , [\partial, R_0 q_\kappa R_0W R_0] \del_x \rangle \big\}
	\label{eq:qdot conv 4}\\
	&\phantom{={}}+ 16\kappa^5 \langle \del_x , [\partial, R_0 W R_0 W R_0] \del_x \rangle
	\label{eq:qdot conv 5}\\
	&\begin{aligned}
	\phantom{={}} + 16\kappa^5 \big\{ g(\kappa,q_\kappa+W) &+ \langle \del_x , R_0(q_\kappa+W)R_0 \del_x \rangle \\
	&- \langle \del_x , R_0(q_\kappa+W)R_0(q_\kappa+W)R_0 \del_x \rangle \big\}' .
	\end{aligned}
	\label{eq:qdot conv 6}
	\end{align}
	We will show that the first five terms \eqref{eq:qdot conv 1}--\eqref{eq:qdot conv 5} converge in $C_tH^{-1}$ to the terms of tidal KdV~\eqref{eq:tkdv}, and the tail~\eqref{eq:qdot conv 6} converges to zero as $\kappa\to\infty$.
	
	We begin with the linear contribution~\eqref{eq:qdot conv 1} from $q_\kappa$.  Using the operator identity~\eqref{eq:linear op id}, we write
	\begin{equation*}
	\eqref{eq:qdot conv 1}
	= -q'''_\kappa  - R_0(2\kappa)\partial^2 (q_\kappa-q)''' - R_0(2\kappa) \partial^2 q''' .
	\end{equation*}
	As $q_\kappa\to q$ in $C_tH^s$, the first term on the RHS converges to $-q'''$ in $C_tH^{s-3}$ and the second term converges to zero in $C_tH^{s-3}$ since $\snorm{ R_0(2\kappa)\partial^2 }_{\op} \lesssim 1$ uniformly in $\kappa$.  The last term converges to zero since the operator $R_0(2\kappa) \partial^2$ is readily seen via Fourier variables to converge strongly to zero as $\kappa\to\infty$.  As the regularity $s-3\geq 0$ is greater than $-1$, we conclude
	\begin{equation*}
	\eqref{eq:qdot conv 1} \to - q''' \quad\tx{in }C_tH^{-1}\tx{ as }\kappa\to\infty.
	\end{equation*}
	
	For the linear contribution~\eqref{eq:qdot conv 2} from $W$, we again use the operator identity~\eqref{eq:linear op id} we write
	\begin{equation*}
	\eqref{eq:qdot conv 2}
	= -W'''  - R_0(2\kappa)\partial^2 W''' .
	\end{equation*}
	As $W'$ is Schwartz and the operator $R_0(2\kappa) \partial^2$ converges strongly to zero as $\kappa\to\infty$, the second term converges to zero in $C_tH^s$ and hence in $C_tH^{-1}$.  Consequently,
	\begin{equation*}
	\eqref{eq:qdot conv 2} \to - W''' \quad\tx{in }C_tH^{-1}\tx{ as }\kappa\to\infty.
	\end{equation*}
	
	Next we turn to the first quadratic term~\eqref{eq:qdot conv 3}.  We write
	\begin{equation*}
	\eqref{eq:qdot conv 3}
	= 6q_\kappa q'_\kappa +  \big\{ 16\kappa^5 \langle \del_x , [\partial, R_0 q_\kappa R_0 q_\kappa R_0] \del_x \rangle - 6q_\kappa q'_\kappa \big\} .
	\end{equation*}
	As $q_\kappa\to q$ in $C_tH^s$, then the first term of the RHS above converges to $6qq'$ in $C_tH^{s-1}$ and hence in $C_tH^{-1}$ as well.  For the second term we estimate in $H^{-1}$ by duality.  For $\phi\in H^1$ we distribute the derivative $[\partial,\cdot]$ using the product rule and use the operator identity~\eqref{eq:quad op id} to obtain
	\begin{align*}
	&\int \big\{ 16\kappa^5 \langle \del_x , [\partial, R_0 q_\kappa R_0 q_\kappa R_0 ] \del_x \rangle - 6q_\kappa q'_\kappa \big\} \phi \dx \\
	&\begin{aligned} = \int \big\{ {-6} [R_0(2\kappa)q''_\kappa][R_0(2\kappa)q'''_\kappa] \phi+ 8\kappa^2 [R_0(2\kappa)q'_\kappa][R_0(2\kappa)q''_\kappa] ( {-5} \phi + R_0(2\kappa) \partial^2 \phi ) & \\
	{}+ 8\kappa^2 [R_0(2\kappa)q_\kappa][R_0(2\kappa)q'_\kappa] (5 \phi'' + 2 R_0(2\kappa) \partial^2 \phi'') \big\} \dx . &\end{aligned}
	\end{align*}
	For each term on the RHS, we put two terms in $L^2$ and the remaining term in $L^\infty$.  For those terms with $\phi''$ we integrate by parts once, we put all factors of $\phi'$ in $L^2$, and we put $\phi$ in $L^\infty \supset H^1$.  We put the highest order $q_\kappa$ term in $L^2$ and the lower order term in $L^2$ or $L^\infty$ as needed.  Using $\snorm{ R_{0}(2 \kappa) \partial^j }_\op \lesssim \kappa^{j-2}$ for $j=0,1,2$ (the estimate for $j = 0$ is also true as an operator on $L^\infty$ by the explicit kernel formula for $R_0$ and Young's inequality), we obtain
	\begin{equation*}
	\left| \int \big\{ 16\kappa^5 \langle \del_x , [\partial, R_0 q_\kappa R_0 q_\kappa R_0 \del_x \rangle - 6q_\kappa q_\kappa \big\} \phi \dx \right|
	\lesssim \kappa^{-2} \norm{\phi}_{H^1} \norm{q_\kappa}_{H^s}^2 .
	\end{equation*}
	Taking a supremum over $\norm{\phi}_{H^1}\leq 1$, we conclude
	\begin{equation*}
	\eqref{eq:qdot conv 3} \to 6qq' \quad\tx{in }C_tH^{-1}\tx{ as }\kappa\to\infty.
	\end{equation*}

	The second quadratic term~\eqref{eq:qdot conv 4} is similar, but now we must put $W$ in $L^\infty$.  First we write
	\begin{align*}
	\eqref{eq:qdot conv 4}
	= 6(W q_\kappa)' +  \big\{ &16\kappa^5 \langle \del_x , [\partial, R_0 W R_0 q_\kappa R_0] \del_x \rangle \\ 
	&+ 16\kappa^5 \langle \del_x , [\partial, R_0 q_\kappa R_0 W R_0] \del_x \rangle - 6(W q_\kappa)' \big\} .
	\end{align*}
	As $q_\kappa \to q$ in $C_tH^s$, the first term of the RHS above converges to $6(Wq)'$ in $C_tH^{s-1}$ and hence in $C_tH^{-1}$ as well.   For the second term we estimate in $H^{-1}$ by duality.  For $\phi\in H^1$ we distribute the derivatives $[\partial,\cdot]$ fusing the product rule and use the operator identity~\eqref{eq:quad op id}.  For the term $\langle \del_x , R_0 W R_0 q'_\kappa R_0 \del_x \rangle$ this yields
	\begin{align*}
	&\int \big\{ 16\kappa^5 \langle \del_x , R_0 W R_0 q'_\kappa R_0 \del_x \rangle - 3W q'_\kappa \big\} \phi \dx \\
	&\begin{aligned}= \int \big\{ {-3} [R_0(2\kappa)W''][R_0(2\kappa)q'''_\kappa] \phi + 4\kappa^2 [R_0(2\kappa)W'][R_0(2\kappa)q''_\kappa] ( {-5} \phi + R_0(2\kappa) \partial^2 \phi ) &\\
	{}+ 4\kappa^2 [R_0(2\kappa)W][R_0(2\kappa)q'_\kappa] (5 \phi'' + 2 R_0(2\kappa) \partial^2 \phi'') \big\} \dx .& \end{aligned}
	\end{align*}
	This equality also holds for the second term $\langle \del_x , R_0 q'_\kappa R_0 W R_0 \del_x \rangle$ because the identity~\eqref{eq:quad op id} is symmetric in $f$ and $h$.  For those terms with $\phi''$ we integrate by parts once to obtain $\phi'$ which we put in $L^2$, we put all factors of $W$ in $L^\infty$, and we put the remaining terms in $L^2$.  This yields
	\begin{equation*}
	\left| \int \big\{ 16\kappa^5 \langle \del_x , R_0 W R_0 q'_\kappa R_0 \del_x \rangle - 3W q'_\kappa \big\} \phi \dx \right|
	\lesssim \kappa^{-2} \norm{\phi}_{H^1} \norm{q_\kappa}_{H^s} ,
	\end{equation*}
	and similarly for the term $\langle \del_x , R_0 q'_\kappa R_0 W R_0 \del_x \rangle$.  The remaining two contributions from $\langle \del_x , R_0 W' R_0 q_\kappa R_0 \del_x \rangle$ and $\langle \del_x , R_0 q_\kappa R_0 W' R_0 \del_x \rangle$ are even easier, since $W'$ is Schwartz and $q_\kappa$ has one less derivative.  Taking a supremum over $\norm{\phi}_{H^1}\leq 1$, we conclude
	\begin{equation*}
	\eqref{eq:qdot conv 4} \to 6(Wq)' \quad\tx{in }C_tH^{-1}\tx{ as }\kappa\to\infty.
	\end{equation*}
	
	The third quadratic term~\eqref{eq:qdot conv 5} is similar.  We write
	\begin{equation*}
	\eqref{eq:qdot conv 5}
	= 6WW' + \big\{ 16\kappa^5 \langle \del_x , [\partial, R_0 W R_0 W R_0] \del_x \rangle - 6WW' \big\} .
	\end{equation*}
	We easily estimate the second term above using the operator identity~\eqref{eq:quad op id} and noting that $W\in L^\infty$ and $W'$ is Schwartz.  This yields
	\begin{equation*}
	\eqref{eq:qdot conv 5} \to 6WW' \quad\tx{in }C_tH^{-1}\tx{ as }\kappa\to\infty.
	\end{equation*}
	
	Lastly, we show that the tail~\eqref{eq:qdot conv 6} converges to zero in $C_tH^{-1}$.   We will estimate in $H^{-1}$ by duality.  For $\phi\in H^1$ we write
	\begin{align*}
	&\left| \int \phi \cdot \eqref{eq:qdot conv 6} \dx \right| \\
	&\leq 16\kappa^5 \sum_{\substack{ \ell\geq 0,\ m_0,\dots,m_\ell \geq 0 \\ \ell+m_0+\dots+m_\ell \geq 3}} \big| \tr \big\{ \phi [\partial, R_0(W R_0)^{m_0} q_\kappa R_0 \cdots q_\kappa R_0 (W R_0)^{m_\ell} ] \big\} \big| .
	\intertext{Recall that we first expanded $g(\kappa,q_\kappa+W)$ in powers of $q_\kappa$, the $\ell$th term having $\ell$-many factors of $q_\kappa R(\kappa,W)$, and then expanded each $R(\kappa,W)$ into a series in $W$ indexed by $m_i$.  The condition $\ell+m_0+\dots+m_\ell \geq 3$ reflects that we have already accounted for all of the summands with one and two $q_\kappa$ or $W$.  We distribute the derivative $[\partial,\cdot]$, use the estimate~\eqref{eq:hskappa identity 1} and the observation $\norm{ f }_{H^{-1}_\kappa} \lesssim \kappa^{-1} \norm{ f }_{L^2}$ to put $\phi$ and all copies of $q_\kappa$ in $L^2$, and then estimate $W$ in operator norm to obtain}
	&\lesssim \kappa^5 \sum_{\substack{ \ell\geq 0,\ m_0,\dots,m_\ell \geq 0 \\ \ell+m_0+\dots+m_\ell \geq 3}} \frac{\norm{\phi}_{L^2}}{\kappa^{3/2}} \left( \frac{\norm{q_\kappa}_{H^{1}}}{\kappa^{3/2}} \right)^\ell \left( \frac{\norm{W}_{W^{1,\infty}}}{\kappa^2} \right)^{m_0 + \dots + m_\ell} .
	\intertext{We first sum over the indices $m_0,\dots,m_\ell \geq 0$ as we did in~\eqref{eq:double sum inner} using that $W \in W^{1,\infty}$, and then we sum over $\ell\geq 1$ since $q_\kappa$ is bounded in $C_tH^s$ for all $\kappa$ large.  In doing so, the condition $\ell+m_0+\dots+m_\ell \geq 3$ guarantees that summing over the two pararenthetical terms yields a gain $\lesssim (\kappa^{-3/2})^3$, from which we conclude}
	&\lesssim \kappa^{-1} \norm{\phi}_{H^1} .
	\end{align*}
	Taking a supremum over $\norm{\phi}_{H^1} \leq 1$ we obtain
	\begin{equation*}
	\eqref{eq:qdot conv 6} \to 0 \quad\tx{in }C_tH^{-1}\tx{ as }\kappa\to\infty. \qedhere
	\end{equation*}
\end{proof}

We now use a classical $L^2$ energy argument to show that we have unconditional uniqueness for  initial data in $H^s$, $s\geq 3$:
\begin{lem}
	\label{thm:tkdv unique}
	Fix $T>0$.  Given an initial data $q(0) \in H^3$, there is at most one corresponding solution to tidal KdV~\eqref{eq:tkdv} in $(C_tH^3 \cap C^1_t L^2)([-T,T]\times\R)$.
\end{lem}
\begin{proof}
	Suppose $q(t)$ and $\tilde{q}(t)$ are both in $(C_tH^3 \cap C^1_tL^2)([-T,T]\times\R)$, solve tidal KdV, and have the same initial data $q(0) = \tilde{q}(0)$.  From the differential equation, we see that the difference obeys
	\begin{align*}
	\left| \ddt \int \tfrac{1}{2} (q - \tilde{q})^2\dx \right|
	&= \left| \int (q-\tilde{q}) \{ - (q-\tilde{q})''' + 3(q^2 - \tilde{q}^2)' + [6W(q-\tilde{q})]' \} \dx \right| .
	\intertext{The first term $(q-\tilde{q})'''$ contributes a total derivative and vanishes, while the remaining terms can be integrated by parts to obtain}
	&= \left| \int (q-\tilde{q})^2 \{ \tfrac{3}{2} (q + \tilde{q})' + 3W' \}(t,x) \dx \right| \\
	&\leq \big( \tfrac{3}{2} \norm{q'}_{L^\infty} + \tfrac{3}{2} \norm{\tilde{q}'}_{L^\infty} + 3 \norm{W'}_{L^\infty} \big) \norm{ q-\tilde{q} }_{L^2}^2 .
	\end{align*}
	Estimating $\norm{q'}_{L^\infty} \lesssim \norm{q}_{H^2}$, $\norm{\tilde{q}'}_{L^\infty} \lesssim \norm{\tilde{q}}_{H^2}$ and noting that $W'$ is Schwartz, we conclude that there exists a constant $C$ depending on $W$ and the norm of $q$ and $\tilde{q}$ in $C_tH^3([-T,T]\times\R)$ such that 
	\begin{equation*}
	\left| \ddt \norm{ q(t)-\tilde{q}(t) }_{L^2}^2 \right| \leq C \norm{ q(t) -\tilde{q}(t) }_{L^2}^2 .
	\end{equation*}
	Gr\"onwall's inequality then yields
	\begin{equation*}
	\norm{ q(t) -\tilde{q}(t) }_{L^2}^2 \leq \norm{ q(0) - \tilde{q}(0) }_{L^2}^2 e^{CT}
	\end{equation*}
	uniformly for $|t|\leq T$.  As the RHS vanishes by premise, we conclude that $\tilde{q}(t) = q(t)$ for all $|t|\leq T$.
\end{proof}

We are now ready to prove our two main results.  First, we complete the proof of continuous dependence upon initial data in $H^s$, $s\geq 3$:
\begin{proof}[Proof of \cref{thm:intro H3}]
	Fix an integer $s\geq 3$.  We want to show that tidal KdV~\eqref{eq:tkdv} is globally well-posed for initial data $q(0) \in H^s(\R)$.
	
	Fix $T>0$ and a convergent sequence $q_n(0)$ of initial data in $H^{s}(\R)$.  It suffices to show that the corresponding solutions $q_n(t)$ of tidal KdV~\eqref{eq:tkdv} constructed in \cref{thm:tkdv exist} are Cauchy in $C_tH^{s}([-T,T]\times\R)$ as $n\to\infty$.

	Consider the set $Q := \{ q_n(0) : n\in\N \}$ of initial data, which is bounded and equicontinuous in $H^s$ since it is convergent in $H^s$.  Let $\thk$ denote the Hamiltonian for the tidal $H_\kappa$ flow.  We estimate
	\begin{equation}
	\begin{aligned}
	\norm{ q_n(t) - q_m(t) }_{C_tH^{s}}
	&\leq \snorm{ e^{tJ\nabla\thk}q_n(0) - e^{tJ\nabla\thk}q_m(0) }_{C_tH^{s}} \\
	&\phantom{\leq{}} + 2\, \sup_{q\in Q}\, \sup_{\varkappa\geq \kappa}\ \snorm{ e^{tJ\nabla\thvk} q - e^{tJ\nabla\thk} q }_{C_tH^{s}} ,
	\end{aligned}
	\label{eq:wellposed 1}
	\end{equation}
	where the spacetime norms are over the slab $[-T,T]\times\R$.  By \cref{thm:conv}, the second term of RHS\eqref{eq:wellposed 1} can be made arbitrarily small uniformly in $n,m$ by picking $\kappa$ sufficiently large.  The first term of RHS\eqref{eq:wellposed 1} then converges to zero as $n,m\to\infty$ due to the well-posedness of the tidal $H_\kappa$ flow (cf. \cref{thm:thk gwp}).
\end{proof}

From our understanding of tidal KdV at high-regularity, we are now able to conclude that KdV is well-posed for $H^{-1}(\R)$ perturbations of step-like solutions:
\begin{proof}[Proof of \cref{thm:intro H-1}]
	Let $V(t) = W + q(t)$ be the solution to KdV~\eqref{eq:kdv} corresponding to the tidal KdV solution with initial data $q(0) \equiv 0$.  We want to show that KdV~\eqref{eq:kdv} is globally well-posed for initial data $u(0) \in V(0) + H^{-1}(\R)$.  By our general result~\cite{Laurens2021}, it suffices to show that for every $T>0$ the following conditions are satisfied:
	\begin{enumerate}[label=(\roman*)]
		\item $V$ solves KdV and is bounded in $W^{2, \infty}(\R_{x})$ uniformly for $|t| \leq T$,
		\item The solutions $V_\kappa(t)$ to the $H_{\kappa}$ flows with initial data $V(0)$ are bounded in $W^{4, \infty}(\R_{x})$ uniformly for $|t| \leq T$ and $\kappa$ sufficiently large,
		\item $V_\kappa - V\to 0$ in $W^{2,\infty}(\R_x)$ as $\kappa\to\infty$ uniformly for $|t|\leq T$ and initial data in the set $\{V_\varkappa(t) : |t|\leq T,\ \varkappa \geq \kappa \}$.
	\end{enumerate}

	Fix $T>0$.  As $q(0) \equiv 0$ is in $H^5$, the \ti{a priori} estimate of \cref{thm:a priori Hs} guarantees that the tidal $H_\kappa$ flows $q_\kappa(t)$ are bounded in $C_tH^5([-T,T]\times\R)$ uniformly for $\kappa$ large.  By definition of the tidal $H_\kappa$ flow we have that $V_\kappa(t) = W + q_\kappa(t)$ solves the $H_\kappa$ flow.  Combined with the embedding $H^1\hookrightarrow L^\infty$, this shows that (ii) is satisfied.
	
	By \cref{thm:equicty}, we know that the sets $Q(\kappa) := \{ q_\varkappa(t) : |t|\leq T,\ \varkappa\geq\kappa \}$ obey~\eqref{eq:equicty hyp}.  Therefore, by \cref{thm:conv} we know that $q_\kappa \to q$ in $C_tH^5([-T,T]\times\R)$ as $\kappa\to\infty$ uniformly for initial data in $Q(\kappa)$.  Consequently $V_\kappa(t)$ converges to $V(t) = W + q(t)$ in $C_tW^{4,\infty}([-T,T]\times\R)$, which shows that (iii) is satisfied.  
	
	Finally, by \cref{thm:tkdv exist} we know $q(t)$ is in $C_tH^5([-T,T]\times\R)$ and solves tidal KdV.  Therefore $V(t)$ solves KdV and is in $C_tW^{4,\infty}([-T,T]\times\R)$, which shows that (i) is satisfied.
\end{proof}

Lastly, we record the following reformulation of well-posedness for $H^{-1}(\R)$ perturbations of $W$:
\begin{cor}
	\label{thm:intro H-1 2}
	Fix a sequence of initial data $u_n(0) \in W + H^3(\R)$ with $u_n(0) - W$ convergent in $H^{-1}(\R)$ as $n\to\infty$, and let $u_n(t)$ denote the corresponding solutions to KdV~\eqref{eq:kdv} guaranteed by \cref{thm:intro H3}.  Then there exists a continuous function $u : \R_t \to  W + H^{-1}(\R)$ so that $u_n(t) - u(t) \to 0$ in $H^{-1}(\R)$ as $n\to\infty$ uniformly on bounded time intervals.
\end{cor}

\bibliography{kdv_gwp_steplike}

\end{document}